\documentclass[11pt, reqno]{amsart}
\usepackage{txfonts}

\usepackage{amsmath,amssymb,amsthm,mathrsfs,enumerate,bm,xcolor,multirow,pbox}
\usepackage{graphicx,color,framed,tikz,caption,subcaption}
\usepackage{enumitem}
\setlist{leftmargin=5mm}
\usepackage[colorlinks,linkcolor=black,citecolor=black,urlcolor=black]{hyperref}
\allowdisplaybreaks[4]
\numberwithin{equation}{section}
\newcommand{\N}{\mathbb{N}}
\newcommand{\R}{\mathbb{R}}

\newcommand{\E}{\mathbb{E}}
\newcommand{\Prob}{\mathbb{P}}

\newcommand{\pnorm}[2]{\lVert#1\rVert_{#2}}
\newcommand{\bigpnorm}[2]{\big\lVert#1\big\rVert_{#2}}
\newcommand{\biggpnorm}[2]{\bigg\lVert#1\bigg\rVert_{#2}}
\newcommand{\abs}[1]{\lvert#1\rvert}
\newcommand{\bigabs}[1]{\big\lvert#1\big\rvert}
\newcommand{\biggabs}[1]{\bigg\lvert#1\bigg\rvert}
\newcommand{\iprod}[2]{\left\langle#1,#2\right\rangle}
\renewcommand{\epsilon}{\varepsilon}

\renewcommand{\d}[1]{\mathrm{d}#1}

\newcommand{\smallop}{\mathfrak{o}_{\mathbf{P}}}

\newcommand{\bigo}{\mathcal{O}}

\renewcommand{\hat}{\widehat}

\DeclareMathOperator{\tr}{tr}
\DeclareMathOperator{\var}{Var}

\DeclareMathOperator{\op}{op}
\DeclareMathOperator{\na}{LNW}
\DeclareMathOperator{\jo}{J}
\DeclareMathOperator{\cm}{CM}
\DeclareMathOperator{\lrt}{LRT}
\let\liminf\relax
\DeclareMathOperator*\liminf{\underline{lim}}
\let\limsup\relax
\DeclareMathOperator*\limsup{\overline{lim}}

\newcommand{\beq}{\begin{equation}}
\newcommand{\eeq}{\end{equation}}
\newcommand{\beqa}{\begin{equation} \begin{aligned}}
\newcommand{\eeqa}{\end{aligned} \end{equation}}
\newcommand{\beqas}{\begin{equation*} \begin{aligned}}
\newcommand{\eeqas}{\end{aligned} \end{equation*}}

\newcommand{\bit}{\begin{itemize}}
	\newcommand{\eit}{\end{itemize}}
\newcommand{\bmat}{\begin{bmatrix}}
	\newcommand{\emat}{\end{bmatrix}}

\theoremstyle{definition}\newtheorem{problem}{Problem}[section]
\theoremstyle{definition}
\theoremstyle{remark}

\theoremstyle{remark}\newtheorem{remark}[problem]{Remark}
\theoremstyle{definition}
\theoremstyle{plain}\newtheorem{theorem}[problem]{Theorem}
\theoremstyle{plain}
\theoremstyle{plain}\newtheorem{lemma}[problem]{Lemma}
\theoremstyle{plain}\newtheorem{proposition}[problem]{Proposition}
\theoremstyle{plain}\newtheorem{corollary}[problem]{Corollary}
\theoremstyle{plain}

\AtBeginDocument{%
	\def\MR#1{}
}

\begin{document}

\title[Contiguity under high dimensional Gaussianity]{Contiguity under high dimensional Gaussianity with applications to covariance testing}
\thanks{The research of Q. Han is partially supported by NSF grants DMS-1916221 and DMS-2143468. The research of T. Jiang is partially supported by NSF grant DMS-1916014.}

\author[Q. Han]{Qiyang Han}

\address[Q. Han]{
Department of Statistics, Rutgers University, Piscataway, NJ 08854, USA.
}
\email{qh85@stat.rutgers.edu}

\author[T. Jiang]{Tiefeng Jiang}

\address[T. Jiang]{
	School of Statistics, University of Minnesota, Minneapolis, MN 55455, USA.
}
\email{jiang040@umn.edu}

\author[Y. Shen]{Yandi Shen}

\address[Y. Shen]{
Department of Statistics, University of Chicago, Chicago, IL 60615, USA.
}
\email{ydshen@uchicago.edu}

\date{\today}


\sloppy

\begin{abstract}
Le Cam's third/contiguity lemma is a fundamental probabilistic tool to compute the limiting distribution of a given statistic $T_n$ under a non-null sequence of probability measures $\{Q_n\}$, provided its limiting distribution under a null sequence $\{P_n\}$ is available, and the log likelihood ratio $\{\log (\d Q_n/\d P_n)\}$ has a distributional limit. Despite its wide-spread applications to low-dimensional statistical problems, the stringent requirement of Le Cam's third/contiguity lemma on the distributional limit of the log likelihood ratio makes it challenging, or even impossible to use in many modern high-dimensional statistical problems. 

This paper provides a non-asymptotic analogue of Le Cam's third/contiguity lemma under high dimensional normal populations. Our contiguity method is particularly compatible with sufficiently regular statistics $T_n$: the regularity of $T_n$ effectively reduces both the problems of (i) obtaining a null (Gaussian) limit distribution and of (ii) verifying our new quantitative contiguity condition, to those of derivative calculations and moment bounding exercises. More important, our method bypasses the need to understand the precise behavior of the log likelihood ratio, and therefore possibly works even when it necessarily fails to stabilize---a regime beyond the reach of classical contiguity methods.

As a demonstration of the scope of our new contiguity method, we obtain asymptotically exact power formulae for a number of widely used high-dimensional covariance tests, including the likelihood ratio tests and trace tests, that hold uniformly over all possible alternative covariance under mild growth conditions on the dimension-to-sample ratio. These new results go much beyond the scope of previous available case-specific techniques, and exhibit new phenomenon regarding the behavior of these important class of covariance tests.
\end{abstract}

\maketitle

\setcounter{tocdepth}{1}
\tableofcontents

\section{Introduction}

\subsection{Le Cam's third/contiguity lemma: a review}

For each $n \in \N$, let $(\Omega_n, \mathcal{B}_n)$ be a measurable space, on which a real-valued random variable $T_n$, and a pair of probability measures $(P_n,Q_n)$ are defined. Here $n$ is a generic index for asymptotics, which is usually related to `sample size' in statistics literature. Le Cam's third/contiguity lemma \cite{lecam1960locally}, which is essentially an asymptotic change of variable formula, computes the limiting law of $\{T_n\}$ under the laws of $\{Q_n\}$, provided that its limiting law under the laws of $\{P_n\}$ can be computed, and the distributions of the log-likelihood ratio $\{\log (\d Q_n/\d P_n)\}$ under the laws of $\{P_n\}$ can be precisely evaluated. The most common form of Le Cam's third/contiguity lemma states the following: Suppose that $\{T_n\}$ is asymptotically normal under the laws of $\{P_n\}$, i.e.,
\begin{align}\label{intro:lecam1}
\frac{T_n-m_{P_n}}{\sigma_{P_n}} \stackrel{P_n}{\rightsquigarrow} \mathcal{N}(0,1)
\end{align}
for some $\{m_{P_n} \in \R\}$ and $\{\sigma_{P_n}>0\}$. Here and below we use $\rightsquigarrow$ to denote weak convergence. Then the limiting law of the normalized random variable $\{(T_n-m_{P_n})/\sigma_{P_n}\}$ under the laws of $\{Q_n\}$ can be computed as
\begin{align}\label{intro:lecam2}
\frac{T_n-m_{P_n}}{\sigma_{P_n}} \stackrel{Q_n}{\rightsquigarrow} \mathcal{N}(\tau,1),
\end{align}
provided that $\sigma_{P_n} \to \sigma_P$ for some $\sigma_P > 0$ and Le Cam's contiguity condition
\begin{align}\label{intro:lecam_cond}
\binom{T_n-m_{P_n}}{\log (\d Q_n/\d P_n)}\stackrel{P_n}{\rightsquigarrow} \mathcal{N}\left(\binom{0}{-\sigma^2/2}, 
\begin{pmatrix}
\sigma_P^2 & \tau\sigma_P\\
\tau \sigma_P & \sigma^2
\end{pmatrix}
\right)
\end{align}
holds for some $\sigma^2>0$.
The condition that $\sigma_{P_n} \to \sigma_P$ can typically be ensured by rescaling $T_n$ appropriately, so the real non-trivial condition is the asymptotic distributional expansions of the statistic $T_n$ and the log likelihood ratio in (\ref{intro:lecam_cond}). We refer the reader to \cite[Chapter 6]{van2000asymptotic} for an in-depth treatment of Le Cam's contiguity theory.

Le Cam's third/contiguity lemma, as stated above, has played a fundamental role in several major developments of estimation and testing theory in mathematical statistics. For instance, convolution and asymptotic minimax theorems for parametric models, which quantify the fundamental information theoretic limits of any regular statistical estimators, are proved with an essential use of Le Cam's third/contiguity lemma. In testing theory, Le Cam's third lemma also facilitates the computation of exact `power function' (to be defined ahead) of any statistical tests. See, e.g., \cite[Chapters 7, 8, 15]{van2000asymptotic} for a textbook treatment on these by-now classical topics.

From (\ref{intro:lecam1})-(\ref{intro:lecam_cond}), it is clear that a successful application of Le Cam's third/contiguity lemma relies heavily on two crucial ingredients: (A) a central limit theorem (CLT) for $\{T_n\}$ under the laws of $\{P_n\}$ in (\ref{intro:lecam1}), and (B) an easy-to-handle log likelihood ratio $\log (\d Q_n/\d P_n)$. This is indeed fairly straightforward in classical models. For instance, a standard application is the study of maximum likelihood estimator (MLE) in parametric statistical models: Let $X_1,\ldots,X_n$ be i.i.d. real-valued random variables from a probability distribution $P_{\theta_0}$ in a parametric class $\mathcal{P}\equiv \{P_\theta:\theta \in \Theta \subset \R\}$, where $\Theta$ is an open set in $\R$, and $T_n\equiv \sqrt{n}(\hat{\theta}-\theta_0)$ where $\hat{\theta}\equiv \hat{\theta}(X_1,\ldots,X_n)$ is the MLE for $\theta_0$ within $\mathcal{P}$. We wish to compute the limiting law of $\{T_n\}$ under the local laws $\{P_{\theta_0+h/\sqrt{n}}\}$ for a fixed $h \in \R$. To apply Le Cam's third lemma, we take $\{P_n\equiv P_{\theta_0}\}$ and $\{Q_n\equiv P_{\theta_0+h/\sqrt{n}}\}$. Now part A can be easily tackled---it is classical knowledge that when $\mathcal{P}$ is `smooth enough', then a CLT for $\{T_n\}$ holds with $m_{P_n}=0$ and $\sigma_{P_n}= (I_{\theta_0})^{-1/2}$ where $I_{\theta_0}$ is the Fisher information of $\mathcal{P}$ at $\theta_0$. Part B can also be handled easily, for instance a direct Taylor expansion of the log likelihood for sufficiently smooth $\mathcal{P}$\footnote{The celebrated local asymptotic normality (LAN) (cf. \cite{lecam1960locally}) condition can also be used for this purpose to weaken smoothness requirements.} concludes, after some calculations, that $\tau=h I_{\theta}^{1/2}$. The limiting law of $\{T_n=\sqrt{n}(\hat{\theta}-\theta_0)\}$ under $\{Q_n=P_{\theta_0+h/\sqrt{n}}\}$ now follows immediately from (\ref{intro:lecam2}).

Although applied in a wide range of contexts with great success in classical low-dimensional statistical problems, Le Cam's third/contiguity lemma faces a key challenge in its stringent requirement for the exact distributional behavior of the log likelihood $\{\log (\d Q_n/\d P_n)\}$ under $\{P_n\}$ as in (\ref{intro:lecam_cond}). In many modern high-dimensional statistical applications, the distributional expansion of the log likelihood can sometimes be extremely difficult, or even impossible to handle. One leading example is given by high dimensional normal populations with spiked covariance, i.e., $P_n=$ (resp. $Q_n$) law of $n$ i.i.d. observations of $p$-dimensional Gaussian vectors with covariance $I_p$ (resp. $I_p+\Delta_p$), where $\Delta_p=\sum_{\ell=1}^r h_\ell v_\ell v_\ell^\top$ is a (fixed) rank $r$ perturbation matrix with $h_\ell \geq 0$ and $\pnorm{v_\ell}{}=1$ for all $1\leq \ell \leq r$. The log likelihood $\{\log (\d Q_n/\d P_n)\}$ can be computed both over the original data and over the maximal invariant (i.e., the eigenvalues of the sample covariance):
\begin{itemize}
	\item For the log likelihood over the original data, fairly straightforward calculations show that, already in the simplest possible rank one case $r=1$, $\{\log (\d Q_n/\d P_n)\}$ can stochastically stabilize under $\{P_n\}$ only if $h_1=\bigo(n^{-1/2})$, a regime of almost no practical relevance in high dimensional settings, say, $\lim (p/n)=y \in (0,\infty)$. 
	\item For the log likelihood over the maximal invariant, \cite{onatski2013asymptotic,onatski2014signal} showed that a distributional limit of $\{\log (\d Q_n/\d P_n)\}$ exists in the sub-critical regime below the Baik-Ben Arous-P\'{e}ch\'{e} (BBP) phase transition \cite{baik2005phase} $\max_{1\leq \ell \leq r} h_\ell <\sqrt{y}$, where $y\equiv \lim (p/n)$. The hard threshold $\sqrt{y}$ is necessary as the weak limit does not exist when $\max_{1\leq \ell \leq r} h_\ell >\sqrt{y}$.
\end{itemize}
Consequently, even for the fixed-rank spiked covariance alternatives, Le Cam's contiguity condition (\ref{intro:lecam_cond}) already fails to obtain non-null distributions of the type (\ref{intro:lecam2}) except for a highly restrictive set of alternatives. Fundamentally, such restrictions arise as it is more than necessary and in fact far too strong to require a weak limit, or even only stochastic stabilization of the log likelihood ratio as in (\ref{intro:lecam_cond}), for the purpose of computing non-null distributions for a given statistics, in particular in high dimensional settings.

\subsection{An analogue to Le Cam's third/contiguity lemma under high dimensional Gaussianity}
In this paper, we establish a non-asymptotic analogue of Le Cam's third/contiguity lemma in the form of (\ref{intro:lecam1})-(\ref{intro:lecam2}), without the requirement for an exact distributional evaluation or even stochastic stabilization of the limiting log likelihood ratio as in (\ref{intro:lecam_cond}), in the setting where $T_n$ is a sufficient `regular' function of $n$ i.i.d. observations of a $p$-dimensional normal distribution. 

Formally, let $X_1,\ldots,X_n$ be i.i.d. samples from a $p$-dimensional normal distribution $\mathcal{N}_p(\mu,\Sigma)$, where $(\mu,\Sigma) \in \R^p\times \mathcal{M}_p$ with $\mathcal{M}_p$ denoting the set of all $p\times p$ covariance matrices. Let $X = [X_1,\ldots,X_n]^\top\in\R^{n\times p}$ be the data matrix. We will be concerned with the statistic $T_n=T(X)$ for some $T: \R^{n\times p}\to \R$ living in the Sobolev space $W^{1,2}(\gamma_{n\times p})$, where $\gamma_{n\times p}$ is the standard Gaussian measure on $\R^{n\times p}$ (precise definitions can be found in Section \ref{section:notation}). Let (throughout the paper we use the symbol $\equiv$ for definition)
\begin{align}\label{def:mean_variance_generic}
m_{(\mu,\Sigma)}\equiv \E_{(\mu,\Sigma)} T(X),\quad \sigma_{(\mu,\Sigma)}^2\equiv \var_{(\mu,\Sigma)}\big(T(X)\big)
\end{align}
be the mean and variance of $T(X)$ under $\mathcal{N}_p(\mu,\Sigma)$, respectively. We always assume that the two quantities in (\ref{def:mean_variance_generic}) are finite. In a similar spirit, we use the subscript $(\mu,\Sigma)$ in $\E_{(\mu,\Sigma)}$ and other probabilistic notations to indicate that the evaluation is under measure $\mathcal{N}_p(\mu,\Sigma)$. 

To motivate the formulation of our results, let us take a pause to see how one may interpret Le Cam's formulation (\ref{intro:lecam1})-(\ref{intro:lecam2}) without going through explicitly the quantities appearing in the contiguity condition (\ref{intro:lecam_cond}). The key observation is that, under mild additional integrability, the asymptotics in (\ref{intro:lecam1}) and (\ref{intro:lecam2}) necessarily entail that $\E_{P_n} (T_n-m_{P_n})/\sigma_{P_n} \approx 0$ and $\E_{Q_n} (T_n-m_{P_n})/\sigma_{P_n} \approx \tau$. This gives $\tau \approx (\E_{Q_n} T_n-\E_{P_n} T_n)/\sigma_{P_n}$, and therefore we may interpret (\ref{intro:lecam1})-(\ref{intro:lecam2}) as
\begin{align}
&\frac{T_n-m_{P_n}}{\sigma_{P_n}}\stackrel{d}{\approx} \mathcal{N}(0,1) \quad \textrm{under } P_n \label{intro:lecam3_1}\\ \stackrel{(\ref{intro:lecam_cond})}{\Rightarrow}\quad & \frac{T_n-m_{P_n}}{\sigma_{P_n}}\stackrel{d}{\approx} \mathcal{N}\bigg(\frac{\E_{Q_n} T_n-\E_{P_n} T_n}{\sigma_{P_n}},1\bigg)\quad \textrm{under } Q_n. \label{intro:lecam3_2}
\end{align}
Compared to (\ref{intro:lecam1})-(\ref{intro:lecam2}), the above formulation does not involve parameters appearing in (\ref{intro:lecam_cond}). Our first main result of this paper establishes an analogue of this formulation (\ref{intro:lecam3_1})-(\ref{intro:lecam3_2}) for $T_n=T(X)$: Under mild regularity conditions on $T$, for any pair $(\mu_0,\Sigma_0),(\mu,\Sigma) \in \R^p\times \mathcal{M}_p$,
\begin{align}
&\frac{T(X)-m_{(\mu_0,\Sigma_0)}}{\sigma_{(\mu_0,\Sigma_0)}}\stackrel{d}{\approx} \mathcal{N}(0,1) \quad \textrm{under } (\mu_0,\Sigma_0) \label{intro:highd_contiguity_1}
\\ \stackrel{(\ast\ast)}{\Rightarrow}\quad & \frac{T(X)-m_{(\mu_0,\Sigma_0)}}{\sigma_{(\mu_0,\Sigma_0)}}\stackrel{d}{\approx} \mathcal{N}\bigg(\frac{m_{(\mu,\Sigma)}-m_{(\mu_0,\Sigma_0)} }{\sigma_{(\mu_0,\Sigma_0)}},1\bigg) \quad \textrm{under } (\mu,\Sigma), \label{intro:highd_contiguity_2}
\end{align}
whereas our contiguity condition $(\ast\ast)$ reads
\begin{align}\label{intro:highd_contiguity_cond}
\overline{\mathrm{err}}_{(\mu,\Sigma);(\mu_0,\Sigma_0)}\equiv \frac{V_{(\mu,\Sigma);(\mu_0,\Sigma_0)}}{\max\big\{|m_{(\mu,\Sigma)} - m_{(\mu_0,\Sigma_0)}|,\sigma_{(\mu_0,\Sigma_0)}\big\}}\to 0.
\end{align}
See Theorem \ref{thm:general} for a formal statement. The `variance' parameter $V_{(\mu,\Sigma);(\mu_0,\Sigma_0)}$, formally defined in (\ref{def:V_generic}) ahead, characterizes the order of stochastic dispersion of $T(X)$ under $(\mu,\Sigma)$ compared to that under $(\mu_0,\Sigma_0)$. Compared to Le Cam's contiguity condition (\ref{intro:lecam_cond}) that requires an exact and strict distributional limit for the log likelihood ratio, (\ref{intro:highd_contiguity_cond}) typically holds in a far broader regime than the prescribed regime in which a distributional limit of the log likelihood ratio exists. For example, in the special case of fixed-rank spiked covariance alternatives, (\ref{intro:highd_contiguity_cond}) already holds for all possible $h_\ell \geq 0$ with all the covariance test statistics studied in this paper, as opposed to the highly restrictive regime imposed by the existence of the weak limit of the log likelihood ratio. In fact, a striking advantage of (\ref{intro:highd_contiguity_cond}) in all the studied examples is its uniform validity over all possible covariance matrices without the need of specifying any particular structure (e.g. spiked alternatives). 

In addition, for sufficiently regular $T_n=T(X)$ with $T(X) \in W^{1,2}(\gamma_{n\times p})$, verification of the contiguity condition (\ref{intro:highd_contiguity_cond}) also has major operational advantages compared to the original Le Cam's contiguity condition (\ref{intro:lecam_cond}). In particular, upper and lower bounds for the stochastic dispersion $V_{(\mu,\Sigma);(\mu_0,\Sigma_0)}$, the mean difference $\abs{m_{(\mu,\Sigma)}-m_{(\mu_0,\Sigma_0)}}$ and the null standard deviation $\sigma_{(\mu_0,\Sigma_0)}$ can usually be reduced to derivative calculations and their moment bounds via efficient applications of Poincar\'e inequalities (or other Fourier techniques in classical Gaussian analysis). If a bit further regularity persists in $T_n$ in that $T(X) \in W^{2,4}(\gamma_{n\times p})$, which is the case for all examples considered in this paper, then a null CLT (\ref{intro:highd_contiguity_1}) can also be reduced to the same derivative calculations and moment bounding exercises, via the renowned second-order Poincar\'e inequality \cite{chatterjee2009fluctuations}. In essence, when the given statistic $T_n$ possesses sufficient regularity, our contiguity method (\ref{intro:highd_contiguity_1})-(\ref{intro:highd_contiguity_2}) can be used to derive its non-null distributions in a rather `mechanical' way by evaluating derivatives and their moment (upper and lower) bounds.

It should be mentioned that while our approach (\ref{intro:highd_contiguity_1})-(\ref{intro:highd_contiguity_2}) here appears to be particularly effective with the sufficient regularity of $T_n$ that naturally postulates a null CLT, it seems less useful when such regularity fails and a different limit occurs under the null; see Remark \ref{rmk:generalize_contiguity} for some technical discussions. Whether a general, effective contiguity approach as (\ref{intro:highd_contiguity_1})-(\ref{intro:highd_contiguity_2}) exists in the `low regularity' regime of $T_n$ remains an interesting open question.

\subsection{Power formula of tests with high-dimensional normal population}

Similar to the wide applicability of Le Cam's third/contiguity lemma in (\ref{intro:lecam1})-(\ref{intro:lecam2}) in classical statistical testing problems, our contiguity result in (\ref{intro:highd_contiguity_1})-(\ref{intro:highd_contiguity_2})  can be applied to many modern high-dimensional statistical problems. Here is a general formulation of the testing problem with normal populations:
\begin{align}\label{intro:test}
H_0: (\mu,\Sigma) \in \mathscr{H}_0 \quad \textrm{versus}\quad H_1: \hbox{$H_0$ does not hold}
\end{align}
where $\mathscr{H}_0$ is a subset of $\R^p \times \mathcal{M}_p$. 

Let $T(X)$ be a generic test statistic whose distribution is invariant under $H_0$, i.e., the law of $T(X)$ remains the same for any $(\mu,\Sigma) \in \mathscr{H}_0$ in (\ref{intro:test}). Due to the distributional invariance of $T(X)$, its mean and variance under the null
\begin{align}\label{def:mean_variance_null_generic}
m_{H_0}\equiv m_{(\mu_0,\Sigma_0)},\quad \sigma_{H_0}^2\equiv \sigma_{(\mu_0,\Sigma_0)}^2
\end{align}
are well-defined for any specification of $(\mu_0,\Sigma_0) \in \mathscr{H}_0$. Suppose further that $T(X)$ verifies the CLT in (\ref{intro:highd_contiguity_1}); this indeed holds in all examples considered in this paper due to their sufficient regularity, and is also anticipated as many covariance tests statistics depend on `sufficient average' of eigenvalues of the sample covariance matrix. Then an asymptotically exact test can be constructed immediately: for any prescribed $\alpha\in(0,1)$,
\begin{align}\label{def:test_generic_intro}
\Psi(X)\equiv \Psi(X;m_{H_0},\sigma_{H_0}) \equiv \bm{1}\bigg(\frac{T(X)-m_{H_0}}{\sigma_{H_0}}>z_\alpha\bigg),
\end{align}
where $z_\alpha$ is the normal quantile such that $\Prob(\mathcal{N}(0,1)>z_\alpha) = \alpha$. Here, following the convention in statistics literature, $\Psi(X)=1$ (resp. $\Psi(X)=0$) indicates rejection (resp. acceptance) of the null hypothesis $H_0$ when $X$ is observed. The quantities $m_{H_0}$ and $\sigma^2_{H_0}$ are usually known in closed forms, at least asymptotically. Even not amenable to exact expression, these quantities can be simulated easily as well.

The quality of the test $\Psi(X)$ is measured by the \emph{power function}, defined for each $(\mu,\Sigma)$ as
\begin{align*}
\hbox{Power of $\Psi(X)$ at $(\mu,\Sigma)$}\equiv \Prob_{(\mu,\Sigma)}\big(\Psi(X) \hbox{ rejects the null $H_0$}\big)=\E_{(\mu,\Sigma)}\Psi(X).
\end{align*}
Applying our contiguity method (\ref{intro:highd_contiguity_1})-(\ref{intro:highd_contiguity_2}), we get the following power formula for the test $\Psi(X)$ associated with the test statistic $T(X)$: Under the assumed CLT condition (\ref{intro:highd_contiguity_1}) and (a slight variation of) the contiguity condition (\ref{intro:highd_contiguity_cond}), 
\begin{align}\label{intro:power}
\E_{(\mu,\Sigma)}\Psi(X)\approx 1-\Phi\bigg(z_\alpha-\frac{m_{(\mu,\Sigma)}-m_{H_0} }{\sigma_{H_0}}\bigg).
\end{align} 
Interestingly, the contiguity condition (\ref{intro:highd_contiguity_cond}) is usually verified both for the case where $(\mu,\Sigma)$ is away from null set $\mathscr{H}_0$ in which the mean difference $|m_{(\mu,\Sigma)} - m_{(\mu_0,\Sigma_0)}|$ dominates the stochastic dispersion $V_{(\mu,\Sigma);(\mu_0,\Sigma_0)}$,  and for the case where $(\mu,\Sigma)$ is very close to the null in which the null standard deviation $\sigma_{(\mu_0,\Sigma_0)}$ dominates $V_{(\mu,\Sigma);(\mu_0,\Sigma_0)}$. As such, the power formula (\ref{intro:power}) can usually be strengthened uniformly over all $(\mu,\Sigma) \in \R^p \times \mathcal{M}_p$.

\subsection{Two concrete applications of (\ref{intro:power})}

We give two concrete applications of (\ref{intro:power}) in the context of covariance testing, as a demonstration of the power of our contiguity result (\ref{intro:highd_contiguity_1})-(\ref{intro:highd_contiguity_2}). 

The first application of (\ref{intro:power}) is the test for identity $\Sigma = I$. In the growing $p$ setting, this problem has been extensively studied in the literature, see e.g.,  \cite{ledoit2002some,srivastava2005some,bai2009corrections,chen2010tests,jiang2012likelihood,cai2013optimal,jiang2013central,zheng2015substitution,chen2018study}. Among the tests studied in the above works, we apply our general theory (\ref{intro:power}) to the following two tests: Likelihood Ratio Test (LRT) (see Section \ref{section:LRT_covariance}) and Ledoit-Nagao-Wolf's test \cite{nagao1973some,ledoit2002some} (see Section \ref{section:nagao}). As an example, the LRT, denoted by $\Psi_{\lrt}(X)$, is shown to admit the following asymptotic power formula (see Theorem \ref{thm:power_covariance_test}): under $\min\{n,p\}\rightarrow\infty$ with $\limsup(p/n) < 1$, 
\begin{align}\label{ineq:power_lrt_intro}
\E_{(\mu,\Sigma)} \Psi_{\lrt}(X) \sim  1- \Phi\bigg(z_\alpha-\frac{ \mathcal{L}_S(\Sigma,I)}{ \sqrt{2\big(-\frac{p}{n-1}-\log\big(1-\frac{p}{n-1}\big)\big)}  }\bigg).
\end{align}
Here $a\sim b$ stands for $a/b\rightarrow 1$ under the prescribed asymptotics, and $\mathcal{L}_S(\cdot,\cdot)$ is the matrix Stein loss to be defined in (\ref{def:stein_loss}) ahead. 

To give a flavor of how (\ref{ineq:power_lrt_intro}) follows from (\ref{intro:power}), recall that the key step in applying (\ref{intro:power}) is to establish that the contiguity condition $\overline{\mathrm{err}}_{(\mu,\Sigma)} \to 0$ in (\ref{intro:highd_contiguity_cond}) (or the current variation defined in (\ref{def:ratio_generic}) ahead). In the LRT setting, a much stronger estimate can be proved in that $\overline{\mathrm{err}}_{(\mu,\Sigma)} \leq Cp^{-1/3}$ holds for some absolute constant $C > 0$. This key estimate follows from a series of algebraic manipulations, upon calculating that $V^2_{(\mu,\Sigma)} = (n-1)\pnorm{\Sigma - I}{F}^2$, $m_{(\mu,\Sigma)} - m_{H_0} = [(n-1)/2]\mathcal{L}_S(\Sigma, I)$, and $\sigma^2_{H_0}\geq cp^2$ for some absolute constant $c > 0$. See Proposition \ref{prop:ratio_covariance_test} and its proof for more details.

The second application is the sphericity test $\Sigma = \lambda I$ for some unspecified $\lambda > 0$. In the growing $p$ setting, this problem has previously been studied in \cite{ledoit2002some,srivastava2005some,chen2010tests,jiang2012likelihood,jiang2013central,jiang2015likelihood}. We study in this paper the following two widely-used tests: LRT for sphericity (Section \ref{section:LRT_spherical}), John's test \cite{john1971some} (Section \ref{section:john}), both invariant under $H_0$. Similar to the previous case, our results on the power behavior of these tests do not pose any assumption on the alternative $\Sigma$. As an example, the LRT for sphericity, denoted by $\Psi_{\lrt,s}(X)$, is shown to admit the following asymptotic power formula (see Theorem \ref{thm:power_spherical}): under $\min\{n,p\}\rightarrow\infty$ with $\limsup(p/n) < 1$,
\begin{align}\label{ineq:power_lrts_intro}
\E_{(\mu,\Sigma)} \Psi_{\lrt,s}(X) \sim 1- \Phi\bigg(z_\alpha-\frac{-\log \det (\Sigma\cdot b^{-1}(\Sigma))}{ \sqrt{2\big(-\frac{p}{n-1}-\log\big(1-\frac{p}{n-1}\big)\big)}  }\bigg).
\end{align}
Here $\det(\cdot)$ is the matrix determinant and $b(\Sigma) \equiv \tr(\Sigma)/p$ with $\tr(\cdot)$ denoting the trace. To the best of our knowledge, the above power formula for the LRT in the sphericity is new in the literature.

The common feature of the power formulae obtained in this paper is that they require no assumptions on the alternative $\Sigma$ and only mild conditions on the growth of $(n,p)$, which goes much beyond the realm of previous available techniques. Roughly speaking, these available techniques either (i) directly establish a CLT under the alternative that crucially exploits the exact form of the test (cf. \cite{wang2013sphericity,cai2013optimal,chen2018study,jiang2019determinant}) and usually requires additional restrictions on the growth of $(n,p)$ and the alternative covariance, or (ii) resort to the classical Le Cam's third/contiguity lemma in (\ref{intro:lecam1})-(\ref{intro:lecam2}) (cf. \cite{onatski2013asymptotic,onatski2014signal}) which, as mentioned above, necessarily fails in a broad regime of alternative covariance for which the log likelihood ratio cannot stabilize.

It is also worth mentioning that the precise power formulae we obtain for the aforementioned tests also have interesting implications compared to previous results in the literature that target at spiked covariance alternatives \cite{onatski2013asymptotic,wang2013sphericity,onatski2014signal}. In particular, as will be clear in Section \ref{sec:spike}, although \cite{onatski2013asymptotic,wang2013sphericity,onatski2014signal} showed that some of the aforementioned tests have asymptotically equivalent power behavior under the spiked covariance alternative with a fixed number of spikes, our new power characterizations indicate that such equivalence in general fails when many spikes exist.

An interesting question untouched in this paper concerns what the information-theoretic optimal power curve for (\ref{intro:test}) looks like, and whether the power formulae (\ref{ineq:power_lrt_intro})-(\ref{ineq:power_lrts_intro}) (or power formulae for other tests) achieve such curves. In general, a successful study of this optimality problem requires two necessary elements: (i) an identification of the limits of the statistical experiments, and (ii) the solvability of the optimality problem in the identified limiting experiments. In classical low-dimensional statistical models (cf. \cite{van2000asymptotic}), (i) is achieved by the LAN property of the log likelihood ratio in these models with the help of the classical Le Cam's contiguity/third lemma (\ref{intro:lecam1})-(\ref{intro:lecam2}), and (ii) is a consequence of the classical decision-theoretic optimality properties of the limiting Gaussian location shift model. In the context of covariance testing, significant progress has been made in \cite{onatski2013asymptotic,onatski2014signal}, in which the `limiting Gaussian experiment' is obtained for fixed-rank covariance alternatives. However, as the limiting experiment is not of the LAN type and its optimality properties remain unclear, the information-theoretic optimal power curve remains unknown. While we believe our contiguity method (\ref{intro:highd_contiguity_1})-(\ref{intro:highd_contiguity_2}) is highly relevant to the optimality problem, achieving fully this goal is far beyond the content and scope of the current paper, and will therefore be deferred to a future study elsewhere.

\subsection{Organization}\label{subsec:organization}
The rest of the paper is organized as follows. We formalize our contiguity results (\ref{intro:highd_contiguity_1})-(\ref{intro:highd_contiguity_2}) in Section \ref{section:testing_general}. Section \ref{section:application_highd_cov_test} is devoted to the application of our contiguity result to the problem of high-dimensional covariance testing. Some key spectral estimates that will be used in the proofs for the results in Section \ref{section:application_highd_cov_test} are presented in Section \ref{section:spectral_estimate} and may be of independent interest. Sections \ref{section:proof_identity}-\ref{section:proof_sphericity} contain the main proofs of results in Sections \ref{section:test_identity} and \ref{section:test_spherical}, with the rest of technical/auxiliary details deferred to the appendices.

\subsection{Notation}\label{section:notation}

For any positive integer $n$, let $[n]$ denote the set $\{1,\ldots,n\}$. For $a,b \in \R$, $a\vee b\equiv \max\{a,b\}$ and $a\wedge b\equiv\min\{a,b\}$. For $a \in \R$, let $a_+\equiv a\vee 0$ and $a_- \equiv (-a)\vee 0$. For $x \in \R^n$, let $\pnorm{x}{p}=\pnorm{x}{\ell_p(\R^n)}$ denote its $p$-norm $(0\leq p\leq \infty)$ with $\pnorm{x}{2}$ abbreviated as $\pnorm{x}{}$. Let $B_p(r;x)\equiv \{z \in \R^p: \pnorm{z-x}{}\leq r\}$ be the unit $\ell_2$ ball in $\R^{p}$. By $\bm{1}_n$ we denote the vector of all ones in $\R^n$. For a matrix $M \in \R^{n\times n}$, let $\pnorm{M}{\op}$ and $\pnorm{M}{F}$ denote the spectral and Frobenius norms of $M$ respectively. We use $\{e_j\}$ to denote the canonical basis, whose dimension should be self-clear from the context.

We use $C_{x}$ to denote a generic constant that depends only on $x$, whose numeric value may change from line to line unless otherwise specified. Notations $a\lesssim_{x} b$ and $a\gtrsim_x b$ mean $a\leq C_x b$ and $a\geq C_x b$ respectively, and $a\asymp_x b$ means $a\lesssim_{x} b$ and $a\gtrsim_x b$. The symbol $a\lesssim b$ means $a\leq Cb$ for some absolute constant $C$. For two nonnegative sequences $\{a_n\}$ and $\{b_n\}$, we write $a_n\ll b_n$ (respectively~$a_n\gg b_n$) if $\lim_{n\rightarrow\infty} (a_n/b_n) = 0$ (respectively~$\lim_{n\rightarrow\infty} (a_n/b_n) = \infty$). We write $a_n\sim b_n$ if $\lim_{n\rightarrow\infty} (a_n/b_n) = 1$. We follow the convention that $0/0 = 0$. 

Let $\varphi,\Phi$ be the density and the cumulative distribution function of a standard normal random variable. For any $\alpha \in (0,1)$, let $z_\alpha$ be the normal quantile defined by $\Prob(\mathcal{N}(0,1)> z_\alpha) = \alpha$. For two random variables $X,Y$ on $\R$, we use $d_{\mathrm{TV}}(X,Y)$ and $d_{\mathrm{Kol}}(X,Y)$  to denote their total variation distance and Kolmogorov distance defined respectively by
\begin{align}\label{def:intro_distance}
\notag d_{\mathrm{TV}}(X,Y)&\equiv \sup_{B \in \mathcal{B}(\R)}\bigabs{\Prob\big(X \in B\big)-\Prob\big(Y \in B\big)},\\
d_{\mathrm{Kol}}(X,Y)&\equiv \sup_{t \in \R}\bigabs{\Prob\big(X \leq t\big)-\Prob\big(Y \leq t\big)}.
\end{align}
Here $\mathcal{B}(\R)$ denotes the Borel $\sigma$-algebra of $\R$. 

Let $\gamma_d$ be the standard Gaussian measure on $\R^d$, and for $r,p\geq 1$ let $W^{r,p}(\gamma_d)$ be the completion of $C_0^\infty(\R^d)$, the space of smooth and compactly supported functions in $\R^d$, with respect to the norm
\begin{align}\label{def:sobolev}
\pnorm{f}{r,p}\equiv \bigg[\sum_{\abs{\bm{\alpha}}\leq r} \int\big|\partial^{\bm{\alpha}} f(x)\big|^p\, \gamma_d(\d{x})\bigg]^{1/p}.
\end{align}
In other words, $W^{r,p}(\gamma_d)$ is the Sobolev space with respect to the Gaussian measure $\gamma_d$.

\section{Contiguity under high-dimensional Gaussianity}\label{section:testing_general}

\subsection{The formal description of (\ref{intro:highd_contiguity_1})-(\ref{intro:highd_contiguity_2})}

Let $T: \R^{n\times p}\to \R$ be a measurable map. For any $(\mu,\Sigma) \in \R^p\times \mathcal{M}_p$, let $\mathscr{T}_{(\mu,\Sigma)}:\R^{n\times p} \to \R^{n\times p}$ be defined by
\begin{align*}
\mathscr{T}_{(\mu,\Sigma)}(z)&\equiv \nabla T \big(z\Sigma^{1/2}+\bm{1}_n\mu^\top\big) \Sigma^{1/2},\quad z \in \R^{n\times p}.
\end{align*}
Here $\bm{1}_n$ is the $n$-vector of all ones, and $\nabla T:\R^{n\times p}\rightarrow\R^{n\times p}$ is the map with $\big(\nabla T(z)\big)_{ij}= \partial T(z)/\partial z_{ij}$. Let $Z_1,\ldots,Z_n$ be i.i.d. random variables with a standard $p$-variate normal distribution $\mathcal{N}(0,I_p)$. For any $(\mu,\Sigma) \in \R^p \times \mathcal{M}_p$, define the quantity
\begin{align}\label{def:V_generic}
V_{(\mu,\Sigma);(\mu_0,\Sigma_0)}^2\equiv  \E \bigpnorm{\mathscr{T}_{(\mu,\Sigma)}(Z)-\mathscr{T}_{(\mu_0,\Sigma_0)}(Z)}{F}^2.
\end{align}

Now we are in a position to give a formal description of our contiguity result (\ref{intro:highd_contiguity_1})-(\ref{intro:highd_contiguity_2}) under the condition (\ref{intro:highd_contiguity_cond}). Recall the quantities $m_{(\mu,\Sigma)},m_{H_0},\sigma_{(\mu,\Sigma)}^2, \sigma_{H_0}^2$ defined in (\ref{def:mean_variance_generic}) and (\ref{def:mean_variance_null_generic}) and that $\gamma_{n\times p}$ denotes the standard Gaussian measure in $\R^{n\times p}$.
\begin{theorem}\label{thm:general}
Suppose that $T:\R^{n\times p}\to \R$ is an element of $W^{1,2}(\gamma_{n\times p})$.
Then for any pair $(\mu_0,\Sigma_0),(\mu,\Sigma) \in \R^p\times \mathcal{M}_p$ and $t \in \R$,
\begin{align*}
&\biggabs{\Prob_{(\mu,\Sigma)}\bigg( \frac{T(X)-m_{(\mu_0,\Sigma_0)}}{\sigma_{(\mu_0,\Sigma_0)} }>t \bigg)-\Prob\bigg( \mathcal{N}\bigg(\frac{m_{(\mu,\Sigma)}-m_{(\mu_0,\Sigma_0)} }{\sigma_{(\mu_0,\Sigma_0)}},1\bigg)>t\bigg)}\\
&\qquad \leq \mathrm{err}_{(\mu_0,\Sigma_0)}+C\cdot \Big((1+\abs{t})\overline{\mathrm{err}}_{(\mu,\Sigma);(\mu_0,\Sigma_0)}\Big)^{2/3}.
\end{align*}
Here $C>0$ is a universal constant, $\overline{\mathrm{err}}_{(\mu,\Sigma);(\mu_0,\Sigma_0)}$ is defined in (\ref{intro:highd_contiguity_cond}), and
\begin{align}\label{def:null_approx_error}
\mathrm{err}_{(\mu_0,\Sigma_0)}\equiv d_{\mathrm{Kol}}\bigg(\frac{T\big(X\big)-m_{(\mu_0,\Sigma_0)}}{\sigma_{(\mu_0,\Sigma_0)} },\mathcal{N}(0,1)\bigg) \quad \textrm{under } (\mu_0,\Sigma_0)
\end{align}
is the normal approximation error of $T(X)$ under $(\mu_0,\Sigma_0)$ in Kolmogorov distance as defined in (\ref{def:intro_distance}).
\end{theorem}

At this point, Theorem \ref{thm:general} does not yet exactly guarantee the closeness of the distribution of the random variable $(T(X)-m_{(\mu_0,\Sigma_0)})/\sigma_{(\mu_0,\Sigma_0)}$ to a standard normal shifted by $(m_{(\mu,\Sigma)}-\mu_{(\mu_0,\Sigma_0)})/\sigma_{(\mu_0,\Sigma_0)}$ under $(\mu,\Sigma)$, as $t \in \R$ in the above theorem cannot be chosen arbitrarily large to yield an informative bound. This is however not a deficiency of our formulation in (\ref{intro:highd_contiguity_1})-(\ref{intro:highd_contiguity_2}). In fact, the range of admissible $t \in \R$ depends on the magnitude of the mean shift parameter $(m_{(\mu,\Sigma)}-\mu_{(\mu_0,\Sigma_0)})/\sigma_{(\mu_0,\Sigma_0)}$. As shown in the following corollary, when the mean shift parameter $(m_{(\mu,\Sigma)}-\mu_{(\mu_0,\Sigma_0)})/\sigma_{(\mu_0,\Sigma_0)}$ is bounded, the conclusion of Theorem \ref{thm:general} can indeed be strengthened to be uniform in $t \in \R$. This is in similar spirit to the classical Le Cam's formulation (\ref{intro:lecam1})-(\ref{intro:lecam2}), in which the mean shift parameter $\tau \in \R$ is treated as a fixed, finite real number in the asymptotics.

\begin{corollary}\label{cor:contiguity_dist}
Consider the same setting as in Theorem \ref{thm:general}. Suppose further that
\begin{align}\label{cond:mean_shift_finite}
\frac{ \abs{m_{(\mu,\Sigma)}-m_{(\mu_0,\Sigma_0)}} }{\sigma_{(\mu_0,\Sigma_0)}}\leq K
\end{align}
for some $K>0$. Then there exists some constant $C_K>0$,
\begin{align*}
&\sup_{t \in \R}\biggabs{\Prob_{(\mu,\Sigma)}\bigg( \frac{T(X)-m_{(\mu_0,\Sigma_0)}}{\sigma_{(\mu_0,\Sigma_0)} }>t \bigg)-\Prob\bigg( \mathcal{N}\bigg(\frac{m_{(\mu,\Sigma)}-m_{(\mu_0,\Sigma_0)} }{\sigma_{(\mu_0,\Sigma_0)}},1\bigg)>t\bigg)}\\
&\qquad \leq \mathrm{err}_{(\mu_0,\Sigma_0)}+C_K\cdot \overline{\mathrm{err}}_{(\mu,\Sigma);(\mu_0,\Sigma_0)}^{4/9}.
\end{align*}
\end{corollary}

There is no a priori reason to believe that the exponent $4/9$ is optimal, but this will have no impact on the qualitative distributional approximation under the contiguity condition  $\overline{\mathrm{err}}_{(\mu,\Sigma);(\mu_0,\Sigma_0)}\to 0$.

\begin{remark}\label{rmk:generalize_contiguity}
It is possible to formulate a version of Theorem \ref{thm:general} without assuming the regularity/integrability $T(X) \in W^{1,2}(\gamma_{n\times p})$ and a null CLT as follows. Suppose that sequences of $\{m_{(\mu,\Sigma)}\}, \{m_{(\mu_0,\Sigma_0)}\}\subset \R$, $\{\sigma_{(\mu_0,\Sigma_0)}\}\subset \R_{>0}$ are chosen such that the following hold:
\begin{itemize}
	\item (Null distribution) There exists some random variable $Y$ such that $(T(X)-m_{(\mu_0,\Sigma_0)})/\sigma_{(\mu_0,\Sigma_0)} \rightsquigarrow Y$ holds under the sequence of $\{(\mu_0,\Sigma_0)\}$. 
	\item (Finite mean shift) $(m_{(\mu,\Sigma)}-m_{(\mu_0,\Sigma_0)})/\sigma_{(\mu_0,\Sigma_0)} \to \tau$ for some $\tau \in \R$.
	\item (Generalized contiguity) With $X^{(\mu,\Sigma)}\equiv Z\Sigma^{1/2}+\bm{1}_n\mu^\top$, 
	\begin{align}\label{cond:generalized_contiguity}
	\frac{\abs{ (T(X^{(\mu,\Sigma)})-m_{(\mu,\Sigma)})-(T(X^{(\mu_0,\Sigma_0)})-m_{(\mu_0,\Sigma_0)})}}{\abs{m_{(\mu,\Sigma)}-m_{(\mu_0,\Sigma_0)}}\vee \sigma_{(\mu_0,\Sigma_0)} }=\smallop(1). 
	\end{align}
\end{itemize}
Then 
\begin{align*}
\frac{T(X)-m_{(\mu_0,\Sigma_0)}}{\sigma_{\mu_0,\Sigma_0}}\rightsquigarrow Y+\tau\quad \textrm{under the sequence of } \{(\mu,\Sigma)\}.
\end{align*}
The proof of this result is simpler than that of Theorem \ref{thm:general} so we omit the details. 

While the above result may be of some interest at an abstract level, currently we are not aware of a concrete example in which the above result leads to genuine reductions in the problem complexity beyond the scope of Theorem \ref{thm:general}. A leading test case here is to take $T(X)=\lambda_1(S)$ as the top eigenvalue of the sample covariance $S$ defined in (\ref{def:sample_covariance}) below, under the rank-one spiked alternative $\Sigma=I+hvv^\top$, $\pnorm{v}{}=1$. Then in the sub-critical regime $h \in (0,\sqrt{c})$ where $c\equiv \lim p/n$, it is not hard to see that verification of the generalized contiguity condition (\ref{cond:generalized_contiguity}) is equivalent to proving a (Type I) Tracy-Widom limit for $\lambda_1$ under the prescribed alternative. It remains an interesting open question, as to whether a general contiguity method without assuming $T(X) \in W^{1,2}(\gamma_{n\times p})$ and a null CLT can be formulated that leads a genuine reduction in deriving, say, at least the Tracy-Widom limit of $\lambda_1$ under the rank-one alternative in the prescribed sub-critical regime.
\end{remark}

\subsection{Power formula for tests with high-dimensional normal population}
Consider the general testing problem (\ref{intro:test}). Recall that $T(X)$ is a generic test statistic whose distribution remains invariant under $H_0$, and the generic test $\Psi(X)$ defined in (\ref{def:test_generic_intro}). To formalize the power formula in (\ref{intro:power}), we need a slight variation of the quantity $V_{(\mu,\Sigma);(\mu_0,\Sigma_0)}$ for non-singleton $\mathscr{H}_0$. Let
\begin{align}\label{def:V_generic_test}
V_{(\mu,\Sigma)}^2\equiv \inf_{(\mu_0,\Sigma_0) \in \mathscr{H}_0}\E \bigpnorm{\mathscr{T}_{(\mu,\Sigma)}(Z)-\mathscr{T}_{(\mu_0,\Sigma_0)}(Z)}{F}^2.
\end{align}
Now we formalize (\ref{intro:power}) which is an immediate consequence of Theorem \ref{thm:general}.

\begin{corollary}\label{cor:power_generic}
Suppose that $T:\R^{n\times p}\to \R$ is an element of $W^{1,2}(\gamma_{n\times p})$ and the law of $T(X)$ is invariant under $H_0$. For any $\alpha\in(0,1)$, there exists some $C_\alpha > 0$ such that
\begin{align*}
&\biggabs{\E_{(\mu,\Sigma)}\Psi(X)-\bigg[1-\Phi\bigg(z_\alpha-\frac{m_{(\mu,\Sigma)}-m_{H_0} }{\sigma_{H_0}}\bigg)\bigg] } \leq \mathrm{err}_{H_0}+C_\alpha \bigg(\frac{V_{(\mu,\Sigma)} }{\abs{m_{(\mu,\Sigma)}-m_{H_0}} \vee \sigma_{H_0} } \bigg)^{2/3}
\end{align*}
holds for any $(\mu,\Sigma)\in \R^p \times \mathcal{M}_p$. Here $\mathrm{err}_{H_0}$ is defined in (\ref{def:null_approx_error}) with $m_{(\mu_0,\Sigma_0)}, \sigma_{(\mu_0,\Sigma_0)}$ replaced by $m_{H_0}, \sigma_{H_0}$.
\end{corollary}

The above result reduces the analysis of the power behavior of $\Psi(X)$ into essentially the following two steps:
\begin{enumerate}
\item (\textit{Normal approximation under $H_0$}) Show that
\begin{align*}
\mathrm{err}_{H_0} = d_{\mathrm{Kol}}\bigg(\frac{T\big(X\big)-m_{H_0}}{\sigma_{H_0} },\mathcal{N}(0,1)\bigg)\rightarrow 0,\qquad \textrm{under }H_0.
\end{align*}
\item (\textit{Contiguity condition}) Show that
\begin{align}\label{def:ratio_generic}
\overline{\mathrm{err}}_{(\mu,\Sigma)}\equiv \frac{V_{(\mu,\Sigma)} }{\abs{m_{(\mu,\Sigma)}-m_{H_0}}\vee \sigma_{H_0}} \rightarrow 0.
\end{align}
\end{enumerate}
Normal approximation of $T(X)$ under $H_0$ can be established in different ways.
When $T(X)$ possesses further regularity, say $T(X) \in W^{2,4}(\gamma_{n\times p})$, a null CLT can usually be established efficiently via Chatterjee's second-order Poincar\'e inequality \cite{chatterjee2009fluctuations}. This approach is particularly compatible with the contiguity condition (\ref{def:ratio_generic}), as we only need to calculate derivatives of $T(X)$ and obtain good enough moment upper and lower bounds for these derivatives. In Section \ref{section:application_highd_cov_test} ahead, we implement this method to a variety of statistics in two concrete problems of high dimensional covariance testing. Some of the resulting exact power results are new, and some improve significantly over earlier results in the literature, both in terms of $(n,p)$-conditions and applicable alternatives.

\subsection{Proof of Theorem \ref{thm:general}}

\begin{lemma}\label{lem:normal_mean_multi}
	For any $t, u \in \R$ and $\eta \in \R$, 
	\begin{align*}
	\bigabs{ \Prob\big( \mathcal{N}(u,1)\leq t\big)-\Prob\big( \mathcal{N}((1+\eta)u,1)\leq t\big)}\leq 2(1+\abs{t})\cdot |\eta|.
	\end{align*}
\end{lemma}
\begin{proof}
	This result strengthens \cite[Lemma 5.4]{han2022high}. We assume without loss generality $\eta \in [-1/2,1/2]$ because otherwise the right hand side of the desired display is greater than or equal to $1$. Note that the left hand side is bounded by
	\begin{align*}
	\biggabs{\int_{t-u}^{t-(1+\eta)u} \varphi(z)\,\d{z}}\leq \abs{\eta}\cdot \bigg[\sup_{v \in [(t-u)- \abs{\eta u},(t-u)+\abs{\eta u}] }\varphi(v)  \abs{u}\bigg]\equiv \abs{\eta}\cdot M_t(u).
	\end{align*}
	Here $\varphi(\cdot)$ is the normal density. First consider the case $u\geq 0$. Then $M_t(u)\leq \sup_{v \in [t-3u/2,t-u/2]}\varphi(v)u$, which can be bounded further in different situations:
	\begin{itemize}
		\item If $t-u/2\leq 0$, then 
		\begin{align*}
		M_t(u)&\leq \varphi\Big(t-\frac{u}{2}\Big)u = \varphi\Big(t-\frac{u}{2}\Big)(u-2t)+2t \varphi\Big(t-\frac{u}{2}\Big)\\
		&\leq 2\sup_{x\in \R} \abs{x}\varphi(x)+\frac{2}{\sqrt{2\pi}}\abs{t} = \frac{2}{\sqrt{2\pi e}}+\frac{2}{\sqrt{2\pi}}\abs{t}.
		\end{align*}
		Here we used the readily verified fact that $\sup_{x\in \R} \abs{x}\varphi(x)=1/\sqrt{2\pi e}$.
		\item If $t-3u/2\geq 0$, then
		\begin{align*}
		M_t(u)&\leq \varphi\Big(t-\frac{3u}{2}\Big)u = \varphi\Big(t-\frac{3u}{2}\Big)\Big(u-\frac{2t}{3}\Big)+\frac{2}{3}t \varphi\Big(t-\frac{3u}{2}\Big)\\
		&\leq \frac{2}{3}\Big(\frac{1}{\sqrt{e}}+\frac{1}{\sqrt{2\pi}}\abs{t}\Big).
		\end{align*}
		\item Otherwise $(2/3)t\leq u\leq 2t$, so $M_t(u)\leq \abs{u}\leq 2\abs{t}$.
	\end{itemize}
    The case $u<0$ can be handled similarly, so we have $\sup_u M_t(u)\leq 2(1+\abs{t})$.
\end{proof}

\begin{proof}[Proof of Theorem \ref{thm:general}]
Let $Z\in \R^{n\times p}$ be a matrix generated by $n$ i.i.d. samples from $\mathcal{N}(0,I_p)$. Let $X^{(\mu,\Sigma)}\equiv Z\Sigma^{1/2}+\bm{1}_n \mu^\top$. 
Then, 
\begin{align}\label{eq:decomposition_general}
\notag&\frac{T(X)-m_{(\mu_0,\Sigma_0)}}{\sigma_{(\mu_0,\Sigma_0)}}\stackrel{d}{=} \frac{ T\big(X^{(\mu,\Sigma)}\big)- T\big(X^{(\mu_0,\Sigma_0)}\big) }{\sigma_{(\mu_0,\Sigma_0)}}+\frac{ T\big(X^{(\mu_0,\Sigma_0)}\big)-m_{(\mu_0,\Sigma_0)} }{\sigma_{(\mu_0,\Sigma_0)}}\\
& = \frac{m_{(\mu,\Sigma)}-m_{(\mu_0,\Sigma_0)}}{\sigma_{(\mu_0,\Sigma_0)}}+ \frac{W(Z)}{\sigma_{(\mu_0,\Sigma_0)}}+ \frac{T\big(X^{(\mu_0,\Sigma_0)}\big)-m_{(\mu_0,\Sigma_0)}}{\sigma_{(\mu_0,\Sigma_0)}}.
\end{align}
Here $W(Z)$ is the centered variable defined by
\begin{align*}
W(Z)\equiv T\big(Z\Sigma^{1/2}+\bm{1}_n\mu^\top\big)- T\big(Z\Sigma_0^{1/2}+\bm{1}_n\mu_0^\top\big) - \big(m_{(\mu,\Sigma)}-m_{(\mu_0,\Sigma_0)}\big).
\end{align*}
Using the chain rule, 
\begin{align*}
\partial_{(ij)} W(Z) 
& = \Big(\nabla T\big(X^{(\mu,\Sigma)}\big) \Sigma^{1/2}-\nabla T\big(X^{(\mu_0,\Sigma_0)}\big) \Sigma_0^{1/2}\Big)_{ij}\\
& = \big(\mathscr{T}_{(\mu,\Sigma)}(Z)-\mathscr{T}_{(\mu_0,\Sigma_0)}(Z)\big)_{ij}.
\end{align*}
By the Gaussian-Poincar\'e inequality \cite[Theorem 3.20]{boucheron2013concentration},
\begin{align*}
\var\big(W(Z)\big)&\leq \E \bigg[\sum_{(ij)} \big(\partial_{(ij)} W(Z)\big)^2\bigg] = \E \bigpnorm{\mathscr{T}_{(\mu,\Sigma)}(Z)-\mathscr{T}_{(\mu_0,\Sigma_0)}(Z)}{F}^2=V_{(\mu,\Sigma);(\mu_0,\Sigma_0)}^2.
\end{align*}
This means for any $u>0$, on an event $E$ with probability at least $1-u^{-2}$,
\begin{align*}
\bigabs{W(Z)}\leq u \cdot V_{(\mu,\Sigma);(\mu_0,\Sigma_0)}.
\end{align*}
Hence for any $t \in \R$, the decomposition (\ref{eq:decomposition_general}) entails that [recall the definition of $\mathrm{err}_{(\mu_0,\Sigma_0)}$ in (\ref{def:null_approx_error})]
\begin{align*}
&\Prob\bigg( \frac{T\big(X^{(\mu,\Sigma)}\big)-m_{(\mu_0,\Sigma_0)}}{\sigma_{(\mu_0,\Sigma_0)} }>t \bigg)\\
& =  \Prob\bigg(  \frac{m_{(\mu,\Sigma)}-m_{(\mu_0,\Sigma_0)} + W(Z)}{\sigma_{(\mu_0,\Sigma_0)}}+ \frac{T\big(X^{(\mu_0,\Sigma_0)}\big)-m_{(\mu_0,\Sigma_0)}}{\sigma_{(\mu_0,\Sigma_0)}}>t \bigg)\\
& \leq \Prob\bigg(  \frac{m_{(\mu,\Sigma)}-m_{(\mu_0,\Sigma_0)} + u\cdot V_{(\mu,\Sigma);(\mu_0,\Sigma_0)}}{\sigma_{(\mu_0,\Sigma_0)}}+ \frac{T\big(X^{(\mu_0,\Sigma_0)}\big)-m_{(\mu_0,\Sigma_0)}}{\sigma_{(\mu_0,\Sigma_0)}}>t \bigg) + \frac{1}{u^2}\\
&\leq \Prob\bigg(\frac{m_{(\mu,\Sigma)}-m_{(\mu_0,\Sigma_0)}+ u\cdot V_{(\mu,\Sigma);(\mu_0,\Sigma_0)}}{\sigma_{(\mu_0,\Sigma_0)}}+ \mathcal{N}(0,1)>t\bigg) + \frac{1}{u^2} + \mathrm{err}_{(\mu_0,\Sigma_0)}\\
&\equiv\mathfrak{p}(u) + \mathrm{err}_{(\mu_0,\Sigma_0)}.
\end{align*}
Next we bound $\mathfrak{p}(\cdot)$ using two different ways. First by Lemma \ref{lem:normal_mean_multi}, we have
\begin{align*}
\inf_{u>0} \mathfrak{p}(u)&\leq \Prob\bigg(\frac{m_{(\mu,\Sigma)}-m_{(\mu_0,\Sigma_0)}}{\sigma_{(\mu_0,\Sigma_0)}}+ \mathcal{N}(0,1)>t\bigg) \\
&\qquad + \inf_{u>0}\bigg[2(1+\abs{t}) u\cdot \frac{V_{(\mu,\Sigma);(\mu_0,\Sigma_0)}}{\abs{m_{(\mu,\Sigma)}-m_{(\mu_0,\Sigma_0)}}} + \frac{1}{u^2}\bigg]\\
&\leq  \Prob\bigg(\frac{m_{(\mu,\Sigma)}-m_{(\mu_0,\Sigma_0)}}{\sigma_{(\mu_0,\Sigma_0)}}+ \mathcal{N}(0,1)>t\bigg)+C \bigg(\frac{(1+\abs{t})V_{(\mu,\Sigma);(\mu_0,\Sigma_0)} }{\abs{m_{(\mu,\Sigma)}-m_{(\mu_0,\Sigma_0)}}}\bigg)^{2/3}.
\end{align*}
On the other hand, by anti-concentration of the standard normal distribution, i.e., $|\Prob\big(\mathcal{N}(0,1)\leq a\big) - \Prob\big(\mathcal{N}(0,1)\leq b\big)|\leq |a-b|$ for any $a,b\in\R$,
\begin{align*}
\inf_{u>0} \mathfrak{p}(u)&\leq \Prob\bigg(\frac{m_{(\mu,\Sigma)}-m_{(\mu_0,\Sigma_0)}}{\sigma_{(\mu_0,\Sigma_0)}}+ \mathcal{N}(0,1)>t\bigg)+\inf_{u>0}\bigg[\frac{u\cdot V_{(\mu,\Sigma);(\mu_0,\Sigma_0)}}{\sigma_{(\mu_0,\Sigma_0)}}+\frac{1}{u^2}\bigg]\\
&\leq \Prob\bigg(\frac{m_{(\mu,\Sigma)}-m_{(\mu_0,\Sigma_0)}}{\sigma_{(\mu_0,\Sigma_0)}}+ \mathcal{N}(0,1)>t\bigg)+C\bigg(\frac{V_{(\mu,\Sigma);(\mu_0,\Sigma_0)}}{\sigma_{(\mu_0,\Sigma_0)}}\bigg)^{2/3}.
\end{align*}
Collecting the bounds completes the proof for one direction. For the other direction, we have
\begin{align*}
&\Prob\bigg( \frac{T\big(X^{(\mu,\Sigma)}\big)-m_{(\mu_0,\Sigma_0)}}{\sigma_{(\mu_0,\Sigma_0)} }>t \bigg)\\
& =  \Prob\bigg(  \frac{m_{(\mu,\Sigma)}-m_{(\mu_0,\Sigma_0)} + W(Z)}{\sigma_{(\mu_0,\Sigma_0)}}+ \frac{T\big(X^{(\mu_0,\Sigma_0)}\big)-m_{(\mu_0,\Sigma_0)}}{\sigma_{(\mu_0,\Sigma_0)}}>t \bigg)\\
&\geq\Prob\bigg(  \frac{m_{(\mu,\Sigma)}-m_{(\mu_0,\Sigma_0)} - u\cdot V_{(\mu,\Sigma);(\mu_0,\Sigma_0)}}{\sigma_{(\mu_0,\Sigma_0)}}+ \frac{T\big(X^{(\mu_0,\Sigma_0)}\big)-m_{(\mu_0,\Sigma_0)} }{\sigma_{(\mu_0,\Sigma_0)}}>t \bigg) - \frac{1}{u^2}\\
&\geq\Prob\bigg(\frac{m_{(\mu,\Sigma)}-m_{(\mu_0,\Sigma_0)}- u\cdot V_{(\mu,\Sigma);(\mu_0,\Sigma_0)}}{\sigma_{(\mu_0,\Sigma_0)}}+ \mathcal{N}(0,1)>t\bigg) - \frac{1}{u^2} - \mathrm{err}_{(\mu_0,\Sigma_0)}.
\end{align*}
Using similar arguments as in the previous direction by invoking the two different bounds concludes the inequality.
\end{proof}

\subsection{Proof of Corollary \ref{cor:contiguity_dist}}

Using the decomposition (\ref{eq:decomposition_general}), and the assumed boundedness condition $\abs{ {(m_{(\mu,\Sigma)}-m_{(\mu_0,\Sigma_0)})}/{\sigma_{(\mu_0,\Sigma_0)}} }\leq K$ (which entails that $\sigma_{(\mu_0,\Sigma_0)}\geq (K+1)^{-1}(\abs{m_{(\mu,\Sigma)}-m_{(\mu_0,\Sigma_0)}}+\sigma_{(\mu_0,\Sigma_0)})$), we have
\begin{align*}
\var_{(\mu,\Sigma)}\bigg(\frac{T(X)-m_{(\mu_0,\Sigma_0)}}{\sigma_{(\mu_0,\Sigma_0)}}\bigg)&\lesssim K^2+\bigg(\frac{V_{(\mu,\Sigma);(\mu_0,\Sigma_0)}}{\sigma_{(\mu_0,\Sigma_0)}}\bigg)^{2}+1\\
& \lesssim (1\vee K)^2\big(1+\overline{\mathrm{err}}_{(\mu,\Sigma);(\mu_0,\Sigma_0)}\big)^2.
\end{align*}
Fix $t>0$. Then the above variance bound leads to
\begin{align*}
\Prob\bigg( \frac{T\big(X^{(\mu,\Sigma)}\big)-m_{(\mu_0,\Sigma_0)}}{\sigma_{(\mu_0,\Sigma_0)} }>t \bigg)\lesssim (t-K)_+^{-2} (1\vee K)^2\big(1+\overline{\mathrm{err}}_{(\mu,\Sigma);(\mu_0,\Sigma_0)}\big)^2.
\end{align*}
On the other hand, 
\begin{align*}
\Prob\bigg(\frac{m_{(\mu,\Sigma)}-m_{(\mu_0,\Sigma_0)}}{\sigma_{(\mu_0,\Sigma_0)}}+ \mathcal{N}(0,1)>t\bigg)\leq e^{-(t-K)_+^2/2}.
\end{align*}
Combined with Theorem \ref{thm:general}, we have for any $t>0$,
\begin{align*}
&\biggabs{\Prob_{(\mu,\Sigma)}\bigg( \frac{T\big(X\big)-m_{(\mu_0,\Sigma_0)}}{\sigma_{(\mu_0,\Sigma_0)} }>t \bigg)-\Prob\bigg( \mathcal{N}\bigg(\frac{m_{(\mu,\Sigma)}-m_{(\mu_0,\Sigma_0)} }{\sigma_{(\mu_0,\Sigma_0)}},1\bigg)>t\bigg)}\\
&\leq \mathrm{err}_{(\mu_0,\Sigma_0)}+C_K \cdot \min \bigg\{\big((1+\abs{t}) \overline{\mathrm{err}}_{(\mu,\Sigma);(\mu_0,\Sigma_0)}\big)^{2/3}, t^{-2} \big(1+\overline{\mathrm{err}}_{(\mu,\Sigma);(\mu_0,\Sigma_0)}\big)^2 \bigg\}.
\end{align*}
If $\overline{\mathrm{err}}_{(\mu,\Sigma);(\mu_0,\Sigma_0)}\geq 1$, then a trivial bound works; otherwise if $\overline{\mathrm{err}}_{(\mu,\Sigma);(\mu_0,\Sigma_0)}\leq 1$, then the right hand side of the above display can be further bounded by
\begin{align*}
&\mathrm{err}_{(\mu_0,\Sigma_0)}+C_K \cdot \min \bigg\{\big((1+\abs{t}) \overline{\mathrm{err}}_{(\mu,\Sigma);(\mu_0,\Sigma_0)}\big)^{2/3}, t^{-2} \bigg\}\\
&\leq \mathrm{err}_{(\mu_0,\Sigma_0)}+ C_K'\cdot  \overline{\mathrm{err}}_{(\mu,\Sigma);(\mu_0,\Sigma_0)}^{4/9}.
\end{align*}
A similar bound holds when $t\leq 0$. \qed

\section{Applications to high dimensional covariance testing}\label{section:application_highd_cov_test}
\subsection{Testing identity $\Sigma =I$}\label{section:test_identity}
Consider the testing problem:
\begin{align}\label{test:cov}
H_0: \Sigma = I\quad \textrm{versus}\quad H_1: \hbox{$H_0$ does not hold}.
\end{align}
This is a special case of (\ref{intro:test}) by taking $\mathscr{H}_0=\R^p \times \{I\}$, and has been extensively studied in the literature; see \cite{ledoit2002some,srivastava2005some,bai2009corrections,chen2010tests,jiang2012likelihood,cai2013optimal,jiang2013central,zheng2015substitution,chen2018study} for an incomplete list.

We introduce some additional notation. Based on i.i.d. samples $X_1,\ldots,X_n$ from $\mathcal{N}(\mu,\Sigma)$, the sample covariance matrix and its unbiased modification are given by 
\begin{align}\label{def:sample_covariance}
S_\ast &\equiv n^{-1}\sum_{k=1}^n \big(X_k-\bar{X}\big) \big(X_k-\bar{X}\big)^\top \quad\textrm{with}\quad  \bar{X}\equiv n^{-1}\sum_{i=1}^n X_i,\nonumber\\
S &\equiv \frac{n}{N}S_\ast  \stackrel{d}{=} \frac{1}{N}\sum_{k=1}^N (X_k-\mu)^\top (X_k-\mu).
\end{align}
Here
\begin{align}
N = n-1
\end{align}
and the equal in distribution in (\ref{def:sample_covariance}) follows from \cite[Theorem 3.1.2]{muirhead1982aspects}. Throughout the rest of the paper, we will mainly work with $S$ for mathematical simplicity (unless otherwise specified), and adopt the right most expression of (\ref{def:sample_covariance}) as its definition whenever no confusion could arise.

\subsubsection{LRT}\label{section:LRT_covariance}

This subsection studies the behavior of the LRT for testing (\ref{test:cov}). The modified log likelihood ratio statistic $T_{\lrt}: \R^{N\times p}\to \R$ (cf. \cite[Theorem 8.4.2]{muirhead1982aspects}) is defined as 
\begin{align}\label{def:lrs}
T_{\lrt}(X) \equiv\frac{N}{2}\big[\tr (S) - \log \det S - p\big].
\end{align}
Clearly the law of $T_{\lrt}(X)$ is invariant under $H_0$. Corollary \ref{cor:power_generic} applies in view of the regularity of $T_{\lrt}$ (see Appendix \ref{section:regularity}). 
We use $\big(m_{\Sigma;\lrt},\sigma_{\Sigma;\lrt},V_{\Sigma;\lrt}\big)$ to represent their generic versions defined in (\ref{def:mean_variance_generic}) and (\ref{def:V_generic_test}).

Following the discussion after Corollary \ref{cor:power_generic}, we start by establishing a quantitative CLT for $T_{\lrt}(X)$ under $H_0$; its proof is presented in Section \ref{subsec:LRT_covariance_clt}. 

\begin{proposition}\label{prop:clt_covariance_test}
Suppose $p/N\leq 1-\epsilon$ for some $\epsilon \in (0,1)$. Then there exists some constant $C=C(\epsilon)>0$, such that under $H_0$,
\begin{align*}
d_{\mathrm{TV}}\bigg(\frac{T_{\lrt}(X)-m_{I;\lrt}}{\sigma_{I;\lrt}},\,\mathcal{N}(0,1)\bigg)\leq \frac{C }{p}.
\end{align*}
\end{proposition}

The CLT for the log likelihood ratio statistic $T_{\lrt}(X)$ under $H_0$ was first derived in \cite{bai2009corrections} using random matrix theory under the assumption that $p/n\rightarrow y$ for some $y\in(0,1)$. This result was then improved in \cite{jiang2012likelihood} and \cite{chen2018study} to hold under the condition $n > p + 1$ and $p\rightarrow \infty$, and in \cite{zheng2015substitution} to relax the Gaussian assumption. The condition $p/N\leq 1-\epsilon$ in Proposition \ref{prop:clt_covariance_test} is used to derive the stable estimate $\E\pnorm{S^{-1}}{\op} \leq C$ for some constant $C = C(\epsilon) > 0$; see Lemma \ref{lem:moment_S_inverse} for details. To our best knowledge, the above result is the first quantitative CLT for $T_{\lrt}(X)$ under $H_0$ in the literature.

The following result establishes the contiguity condition (\ref{def:ratio_generic}) for the log likelihood ratio statistic $T_{\lrt}(X)$; its proof is presented in Section \ref{subsec:LRT_covariance_ratio}. For p.s.d. $\Sigma_1$ and p.d. $\Sigma_2$, let 
\begin{align}\label{def:stein_loss}
\mathcal{L}_S(\Sigma_1,\Sigma_2)\equiv \tr(\Sigma_1\Sigma_2^{-1})-\log \det\big(\Sigma_1 \Sigma_2^{-1})-p
\end{align}
be the Stein loss with the convention that $\mathcal{L}_S(\Sigma_1,\Sigma_2)\equiv \infty$ if $\Sigma_1$ is singular.

\begin{proposition}\label{prop:ratio_covariance_test}
Suppose $\Sigma$ is non-singular. The following hold:
\begin{enumerate}
	\item $V_{\Sigma;\lrt}^2= N\pnorm{\Sigma-I}{F}^2$.
	\item $m_{\Sigma;\lrt}-m_{I;\lrt} = (N/2)\mathcal{L}_S(\Sigma,I)$.
	\item In the asymptotic regime $N\geq p+1$ with $p \to \infty$,
	\begin{align*}
	\sigma_{I;\lrt}^2 \sim \frac{N^2}{2}\bigg[-\frac{p}{N}-\log\bigg(1-\frac{p}{N}\bigg)\bigg].
	\end{align*}
	In particular, $\sigma_{I;\lrt}^2\geq cp^2$ for some universal constant $c>0$.
	\item There exists some universal constant $C>0$ such that
	\begin{align*}
	\frac{V_{\Sigma;\lrt}}{\abs{m_{\Sigma;\lrt}-m_{I;\lrt}}\vee\sigma_{I;\lrt} }\leq \frac{C}{p^{1/2}}.
	\end{align*}	
\end{enumerate}

\end{proposition}

The above proposition gives a prototypical example of how to proceed with the contiguity condition (\ref{def:ratio_generic}). For the log likelihood ratio statistic $T_{\lrt}(X)$ defined in (\ref{def:lrs}), both $V_{\Sigma;\lrt}$ and the mean difference $m_{\Sigma;\lrt}-m_{I;\lrt}$ admit easy-to-handle closed-form formulae. To give some insights for the bound obtained in Proposition \ref{prop:ratio_covariance_test}-(4), let us consider  the `local regime' of alternatives in which $\mathcal{L}_S(\Sigma,I)\approx \pnorm{\Sigma-I}{F}^2$. Then (\ref{def:ratio_generic}) can be bounded, up to a constant, by
\begin{align*}
\frac{\sqrt{N\pnorm{\Sigma-I}{F}^2} }{ N\pnorm{\Sigma-I}{F}^2 \vee \sigma_{I;\lrt}}\leq \sup_{x\geq 0} \frac{x}{x^2 \vee \sigma_{I;\lrt}} = \frac{1}{\inf_{x \geq 0} \big(x \vee \frac{\sigma_{I;\lrt}}{x}\big)} = \frac{1}{\sigma_{I;\lrt}^{1/2}},
\end{align*}
in the prescribed local regime of alternatives. The above simple reasoning exemplifies the essential reason why (\ref{def:ratio_generic}) must be small in high dimensions: if $\Sigma$ is sufficiently away from $I$, then the mean difference $m_{\Sigma;\lrt}-m_{I;\lrt}$ is substantially larger than $V_{\Sigma;\lrt}$, but would otherwise be compensated by a diverging  $\sigma_{I;\lrt}$.

Let $\Psi_{\lrt}(X)$ be the LRT built from the generic test (\ref{intro:test}) and the log likelihood ratio statistic $T_{\lrt}(X)$. Now Corollary \ref{cor:power_generic} yields the following.

\begin{theorem}\label{thm:power_covariance_test}
	Suppose $p/N\leq 1-\epsilon$ for some $\epsilon \in (0,1)$. Then there exists some constant $C=C(\epsilon,\alpha)>0$ such that 
	\begin{align}\label{ineq:power_covariance_lrt}
	&\biggabs{ \E_\Sigma \Psi_{\lrt}(X) -   \Prob\bigg( \mathcal{N}\bigg(\frac{N \cdot \mathcal{L}_S(\Sigma,I) }{2\sigma_I},1\bigg)>z_\alpha\bigg)   } \leq C \cdot p^{-1/3}.
	\end{align}
	Consequently, in the asymptotic regime $N \wedge p \to \infty$ with $\limsup (p/N)<1$, 
	\begin{align*}
     \E_\Sigma \Psi_{\lrt}(X) \sim  1- \Phi\bigg(z_\alpha-\frac{\mathcal{L}_S(\Sigma,I)}{ \sqrt{2\big(-\frac{p}{N}-\log\big(1-\frac{p}{N}\big)\big)}  }\bigg).
	\end{align*} 
\end{theorem}

Compared to results in \cite{chen2018study} on the power behavior of the LRT $\Psi_{\lrt}(X)$, we remove the condition  $\sup_n \pnorm{\Sigma}{\op} < \infty$ completely. This unnecessary condition arises as a technical deficiency in the approach of \cite{chen2018study} that attempts at directly establishing a CLT for $\Psi_{\lrt}(X)$ under general alternatives.

\subsubsection{Ledoit-Nagao-Wolf's test}\label{section:nagao}
This subsection studies testing (\ref{test:cov}) using the (rescaled) modified Nagao's trace statistic \cite{nagao1973some} by Ledoit and Wolf \cite{ledoit2002some}:
\begin{align}\label{def:lrs_nagao}
T_{\na}(X)\equiv \frac{N}{4} \bigg[\tr\big(S-I\big)^2-\frac{1}{N} \tr^2(S)\bigg].
\end{align}
An asymptotically equivalent statistic as an unbiased estimator of $\pnorm{\Sigma -I}{F}^2$ has also been studied in \cite{srivastava2005some}. One advantage of using (\ref{def:lrs_nagao}) is that it applies to the case $p > n$ where the LRT in Section \ref{section:LRT_covariance} becomes degenerate.

We will use $\big(m_{\Sigma;\na},\sigma_{\Sigma;\na},V_{\Sigma;\na}\big)$ to represent their generic versions defined in (\ref{def:mean_variance_generic}) and (\ref{def:V_generic_test}). 

\begin{proposition}\label{prop:clt_covariance_test_nagao}
	There exists an absolute constant $C>0$ such that under $H_0$,
	\begin{align*}
	d_{\mathrm{TV}}\bigg(\frac{T_{\na}(X)-m_{I;\na}}{\sigma_{I;\na}},\,\mathcal{N}(0,1)\bigg)\leq \frac{C}{N\wedge p}.
	\end{align*}
\end{proposition}

The proof is presented in Section \ref{subsec:nagao_clt}. The CLT for $T_{\na}(X)$ was first derived in \cite[Proposition 7]{ledoit2002some} under the condition that $p/N\rightarrow y\in(0,\infty)$, which was later improved in \cite[Theorem 3.6]{birke2005note} to include the case $y\in\{0,\infty\}$. Here we give explicit error bounds in the normal approximation.

The following result establishes the contiguity condition (\ref{def:ratio_generic}) for $T_{\na}$; its proof is presented in Section \ref{subsec:nagao_ratio}.
\begin{proposition}\label{prop:ratio_nagao}
	Suppose $p/N \leq M$ for some $M>0$. Then the following hold:
	\begin{enumerate}
		\item $V_{\Sigma;\na}^2\leq C_1 N\big(\pnorm{\Sigma}{\op}^2\vee 1\big)\pnorm{\Sigma-I}{F}^2$ for some constant $C_1=C_1(M)>0$.
		\item With $Q_{\na}(\Sigma) \equiv (N^{-1}-2N^{-2})\tr(\Sigma^2-I)$, 
		\begin{align*}
		m_{\Sigma;\na}-m_{(0,I)} =  \frac{N}{4}\big[\pnorm{\Sigma-I}{F}^2+Q_{\na}(\Sigma)\big].
		\end{align*}
		\item In the asymptotic regime $N \wedge p \to \infty$,
		\begin{align*}
		\sigma_{I;\na}^2 \sim \frac{p^2}{4}.
		\end{align*}
		\item There exists some constant $C_2=C_2(M)>0$ such that
		\begin{align*}
		\frac{V_{\Sigma;\na}}{\abs{m_{\Sigma;\na}-m_{I;\na}}\vee\sigma_{I;\na} }\leq \frac{C_2}{p^{1/2}}.
		\end{align*}
		\end{enumerate}

\end{proposition}

Compared to Proposition \ref{prop:ratio_covariance_test}, although a closed-form formula is available for $m_{\Sigma;\na}$, a somewhat undesirable `residual term'  $Q_{\na}(\Sigma)$ exists. Removing the effect of these terms in the final step (4) requires significant additional technicalities, as will be detailed in Section \ref{subsec:nagao_ratio}.

Let $\Psi_{\na}(X)$ be the test built from (\ref{intro:test}) and the statistic in (\ref{def:lrs_nagao}). Combining the above results with Corollary \ref{cor:power_generic} and some additional efforts to remove the residual term $Q_{\na}(\Sigma)$ in the mean difference formula (2) in the above proposition, we have the following asymptotic power formula for $\Psi_{\na}(X)$; see Section \ref{subsec:nagao_power} for its proof.

\begin{theorem}\label{thm:power_nagao}
	Suppose $p/N\leq M$ for some $M>0$. Then there exists some constant $C=C(\alpha,M)>0$ such that 
	\begin{align*}
	&\biggabs{ \E_{\Sigma} \Psi_{\na}(X)-   \Prob\bigg( \mathcal{N}\bigg(\frac{N \cdot \pnorm{\Sigma-I}{F}^2 }{4\sigma_{I;\na}},1\bigg)>z_\alpha\bigg)   } \leq C\cdot p^{-1/3}.
	\end{align*}
	Consequently, in the asymptotic regime $N \wedge p \to \infty$ with $\limsup (p/N)<\infty$, 
	\begin{align*}
    \E_\Sigma \Psi_{\na}(X) \sim 1- \Phi\bigg(z_\alpha-\frac{\pnorm{\Sigma-I}{F}^2}{ 2(p/N)  }\bigg).
	\end{align*} 
\end{theorem}

The asymptotic behavior of $T_{\na}$ under the alternative is previously only known in \cite[Theorem 4.1]{srivastava2005some} under rather restrictive conditions on both $\Sigma$ and growth of $p$. Theorem \ref{thm:power_nagao} only requires $p/N$ to be bounded and makes no assumptions on $\Sigma$.

\subsection{Testing sphericity $\Sigma=\lambda I$}\label{section:test_spherical}
Consider the testing problem:
\begin{align}\label{test:spherical}
H_0: \Sigma = \lambda I\quad \textrm{versus}\quad H_1: \hbox{$H_0$ does not hold}
\end{align}
for some un-specified $\lambda>0$. This is a special case of (\ref{intro:test}) by taking $\mathscr{H}_0=\R^p\times \{\lambda I:\lambda>0\}$, and has been extensively studied previously in \cite{ledoit2002some,srivastava2005some,chen2010tests,jiang2012likelihood,jiang2013central}.

\subsubsection{LRT}\label{section:LRT_spherical}
This subsection studies the LRT for (\ref{test:spherical}). The (re-scaled) log-likelihood ratio statistic for (\ref{test:spherical}) is defined by (cf. \cite[Theorem 8.3.2]{muirhead1982aspects}):
\begin{align}\label{def:lrs_spherical}
T_{\lrt,s}(X) &\equiv \frac{N}{2}\big(p\log\tr(S) - \log\det S - p\log p\big).
\end{align} 
Evidently, the law of $T_{\lrt,s}(X)$ does not depend on the $\lambda$ in (\ref{test:spherical}) and hence is invariant under $H_0$. Thus the general principle in Theorem \ref{thm:general} applies due to regularity of $T_{\lrt;s}$ (see Appendix \ref{section:regularity}). 
We will use $\big(m_{\Sigma;\lrt,s},\sigma_{\Sigma;\lrt,s},V_{\Sigma;\lrt,s}\big)$ to represent their generic versions defined in (\ref{def:mean_variance_generic}) and (\ref{def:V_generic_test}). 

For a symmetric $p\times p$ matrix $M$, let
\begin{align}\label{def:trace_b}
b_\ell(M)&\equiv p^{-1}\tr(M^\ell),\quad b(M)\equiv b_1(M).
\end{align}
The next proposition establishes a quantitative CLT for $T_{\lrt,s}(X)$; its proof is presented in Section \ref{subsec:spherical_lrt_clt}. Recall that $T_{\lrt,s}$ is non-degenerate only if $p\leq n-1 = N$.

\begin{proposition}\label{thm:clt_spherical}
	Suppose $p/N\leq 1-\epsilon$ for some $\epsilon \in (0,1)$. Then there exists some $C=C(\epsilon)>0$ such that under $H_0$,
	\begin{align*}
	d_{\mathrm{TV}}\bigg(\frac{T_{\lrt,s}(X)-m_{I;\lrt,s}}{\sigma_{I;\lrt,s}},\,\mathcal{N}(0,1)\bigg)\leq \frac{C}{p}.
	\end{align*}
\end{proposition}
The CLT for $T_{\lrt,s}(X)$ was previously derived in \cite[Theorem 1]{jiang2013central} under the asymptotics $y\in(0,1]$. The quantitative CLT above does not require $p$ to grow proportionally to $N$ but excludes the boundary case $y  = 1$.

The following result establishes the contiguity condition (\ref{def:ratio_generic}) for $T_{\lrt,s}$; see Section \ref{subsec:spherical_lrt_ratio} for its proof.

\begin{proposition}\label{prop:ratio_spherical}
	Suppose $\Sigma$ is non-singular. The following hold:
	\begin{enumerate}
		\item There exists some absolute constant $C_1>0$ such that
		\begin{align*}
		V_{\Sigma;\lrt,s}^2\leq C_1 N\pnorm{{\Sigma}\cdot{b^{-1}(\Sigma)}-I}{F}^2
		\end{align*}
		holds for $N,p$ large enough.
		\item The mean difference is given by
		\begin{align*}
		m_{\Sigma;\lrt,s}-m_{I;\lrt,s} =\frac{N}{2} \big[-\log \det({\Sigma}\cdot{b^{-1}(\Sigma)}) +Q_{\lrt,s}({\Sigma}\cdot{b^{-1}(\Sigma)}) \big].
		\end{align*}
		Here 
		\begin{align}\label{ineq:Q_lrt}
		\bigabs{Q_{\lrt,s}\big({\Sigma}\cdot{b^{-1}(\Sigma)}\big)} &\leq C_2 N^{-1} b\big[ \big({\Sigma}\cdot{b^{-1}(\Sigma)}\big)^2\big]
		\end{align}
		for some absolute constant $C_2>0$.
		\item In the asymptotic regime $N\wedge p \to \infty$ with $\limsup (p/N)<1$, 
		\begin{align*}
		\sigma_{I;\lrt,s}^2\sim  \frac{N^2}{2}\bigg[-\frac{p}{N}-\log\bigg(1-\frac{p}{N}\bigg)\bigg].
		\end{align*}
		\item There exists some absolute constant $C_3>0$ such that
		\begin{align*}
		\frac{V_{\Sigma;\lrt,s}}{\abs{m_{\Sigma;\lrt,s}-m_{I;\lrt,s}}\vee\sigma_{I;\lrt,s} }\leq \frac{C_3 }{(\sigma_{I;\lrt,s}\wedge N)^{1/2}}.
		\end{align*}	
	\end{enumerate}
\end{proposition}

There is a genuine difference between the above contiguity result and the previous ones studied in Section \ref{section:LRT_covariance}, in that a closed-form formula for the mean difference $m_{\Sigma;\lrt,s}-m_{I;\lrt,s}$ is no longer available. One therefore has to work with strong enough upper bounds for the `residual term' $Q_{\lrt,s}(\Sigma\cdot b^{-1}(\Sigma))$, the removal of which constitutes the main technicalities in the proofs; see Section \ref{subsec:spherical_lrt_ratio} for details.

Let $\Psi_{\lrt,s}(X)$ be the test built from (\ref{intro:test}) and the statistic in (\ref{def:lrs_spherical}). Combining the above results with Theorem \ref{thm:general} and some additional efforts to remove the residual term $Q_{\lrt,s}(\Sigma\cdot b^{-1}(\Sigma))$, we have the following asymptotic power formula for $\Psi_{\lrt,s}(X)$; see Section \ref{subsec:spherical_clt_power} for its proof.

\begin{theorem}\label{thm:power_spherical}
	Suppose $p/N\leq 1-\epsilon$ for some $\epsilon \in (0,1)$. Then there exists some constant $C=C(\epsilon,\alpha)>0$ such that 
	\begin{align*}
	&\biggabs{ \E_{\Sigma} \Psi_{\lrt,s}(X)-   \Prob\bigg( \mathcal{N}\bigg(-\frac{ N \log \det\big({\Sigma}\cdot{b^{-1}(\Sigma)}\big) }{2\sigma_{I;s}},1\bigg)>z_\alpha\bigg)   }\leq C\cdot p^{-1/3}.
	\end{align*}
	Consequently, in the asymptotic regime $N \wedge p \to \infty$ with $\limsup (p/N)<1$,
	\begin{align*}
	\E_{\Sigma} \Psi_{\lrt,s}(X) \sim 1- \Phi\bigg(z_\alpha-\frac{-\log \det \big({\Sigma}\cdot{b^{-1}(\Sigma)}\big)}{ \sqrt{2\big(-\frac{p}{N}-\log\big(1-\frac{p}{N}\big)\big)}  }\bigg).
	\end{align*} 
\end{theorem}

To the best of our knowledge, in the high dimensional regime $N\wedge p \rightarrow \infty$, the LRT for (\ref{test:spherical}) was only studied in \cite{jiang2013central,jiang2015likelihood}, where formal theory was missing on the power behavior of $\Psi_{\lrt,s}$. Theorem \ref{thm:power_spherical} fills this gap.

\subsubsection{John's test}\label{section:john}
Consider testing (\ref{test:spherical}) using the (rescaled) John's trace statistic \cite{john1971some}:
\begin{align}\label{def:lrs_john}
T_{\jo}(X)\equiv \frac{N}{4} \tr\bigg[\bigg(\frac{S}{p^{-1}\tr(S)}-I\bigg)^2\bigg].
\end{align}
Clearly the law of $T_{J}(X)$ is invariant under $H_0$, and the above statistic  is non-degenerate for all configurations of $(n,p)$. The general principle in Theorem \ref{thm:general} thereby applies in view of the regularity of $T_{\jo}$ (see Appendix \ref{section:regularity}). We will use $\big(m_{\Sigma;\jo},\sigma_{\Sigma;\jo},V_{\Sigma;\jo}\big)$ to represent their generic versions defined in (\ref{def:mean_variance_generic}) and (\ref{def:V_generic}). 

The next proposition establishes a quantitative CLT for $T_{\jo}(X)$ under $H_0$; its proof is given in Section \ref{subsec:john_clt}. 

\begin{proposition}\label{prop:clt_covariance_test_john}
	There exists some absolute constant $C>0$, such that under $H_0$,
	\begin{align*}
	d_{\mathrm{TV}}\bigg(\frac{T_{\jo}(X)-m_{I;\jo}}{\sigma_{I;\jo}},\,\mathcal{N}(0,1)\bigg)\leq \frac{C}{N\wedge p}.
	\end{align*}
\end{proposition}

CLTs for $T_{\jo}(X)$ under $H_0$ in high dimensions are first obtained in \cite{ledoit2002some}. We improve these results both in terms of non-asymptotic normal approximation bound and the removal of the condition $0<\liminf (p/N)\leq \limsup (p/N) < \infty$.

The following result establishes the contiguity condition (\ref{def:ratio_generic}) for $T_{\jo}$; its proof is presented in Section \ref{subsec:john_ratio}. Recall the definition of $b(\Sigma)$ in (\ref{def:trace_b}).

\begin{proposition}\label{prop:ratio_john}
	Suppose $p/N\leq M$ for some $M>1$. Then the following hold for $N$ larger than a big enough absolute constant:
	\begin{enumerate}
		\item There exists some constant $C_1=C_1(M)>0$  such that
		\begin{align*}
		V_{\Sigma;\jo}^2\leq C_1 \cdot N \big(\pnorm{{\Sigma}\cdot{b^{-1}(\Sigma)}}{\op}^2 \vee 1\big)\pnorm{{\Sigma}\cdot{b^{-1}(\Sigma)}-I}{F}^2.
		\end{align*}
		\item The mean difference is given by
		\begin{align*}
		m_{\Sigma;\jo}-m_{I;\jo} =\frac{N}{4}\big[ \pnorm{{\Sigma}\cdot{b^{-1}(\Sigma)}-I}{F}^2+Q_{\jo}\big({\Sigma}\cdot{b^{-1}(\Sigma)}\big) \big].
		\end{align*}
		Here 
		\begin{align*}
		\bigabs{Q_{\jo}\big({\Sigma}\cdot{b^{-1}(\Sigma)}\big)}\leq C_2 \cdot N^{-1/2} \big(p^{-1}\pnorm{{\Sigma}\cdot{b^{-1}(\Sigma)}}{F}^2+ 1\big) \pnorm{{\Sigma}\cdot{b^{-1}(\Sigma)}-I}{F}
		\end{align*}
		for some $C_2=C_2(M)>0$.
		\item In the asymptotic regime $N \wedge p \to \infty$,
		\begin{align*}
		\sigma_{I;\jo}^2\sim \frac{p^2}{4}.
		\end{align*}
		\item There exists some $C_3=C_3(M)>0$ such that
		\begin{align*}
		\frac{V_{\Sigma;\jo}}{\abs{m_{\Sigma;\jo}-m_{I;\jo}}\vee\sigma_{I;\jo} }\leq \frac{C_3}{p^{1/2}}.
		\end{align*}
	\end{enumerate}
\end{proposition}

The proof of the above contiguity is the most complicated among the examples studied in this paper. The main complication is due to the existence of the $\tr(S)$ term in the denominator in (\ref{def:lrs_john}), which leads to the complications both in the control of $V_{\Sigma;J}^2$ and the `residual term' $Q_{\jo}(\Sigma\cdot b^{-1}(\Sigma))$.

Let $\Psi_{\jo}(X)$ be the test built from (\ref{intro:test}) and the statistic in (\ref{def:lrs_john}). Now we have

\begin{theorem}\label{thm:power_john}
	Suppose $ p/N \leq M$ for some $M>1$. Then there exists some constant $C=C(\alpha,M)>0$ such that 
	\begin{align*}
	&\biggabs{ \E_{\Sigma} \Psi_{\jo}-   \Prob\bigg( \mathcal{N}\bigg(\frac{N \cdot \pnorm{{\Sigma}\cdot{b^{-1}(\Sigma)}-I}{F}^2 }{4\sigma_{I;\jo}},1\bigg)>z_\alpha\bigg)   }\leq C \cdot p^{-1/3}.
	\end{align*}
	Consequently, in the asymptotic regime $N\wedge p \to \infty$ with $ \limsup (p/N)<\infty$,
	\begin{align*}
    \E_\Sigma \Psi_{\jo} \sim 1- \Phi\bigg(z_\alpha-\frac{\pnorm{{\Sigma}\cdot{b^{-1}(\Sigma)}-I}{F}^2}{2(p/N)}\bigg).
	\end{align*} 
\end{theorem}
See Section \ref{subsec:john_power} for the proof. The power behavior for John's test is previous studied in \cite{onatski2013asymptotic,onatski2014signal,wang2013sphericity} for a special class of alternatives under the spiked covariance model with a fixed number of spikes; see Section \ref{sec:spike} ahead for a detailed discussion. To the best of our knowledge, the theorem above gives the first complete characterization of the power behavior for John's test for arbitrary alternatives in the high-dimensional regime $N\wedge p\to \infty$ with $\limsup(p/N)<\infty$.

\subsection{Case study: spiked covariance models}\label{sec:spike}

In this subsection, we consider a special class of alternatives known as the spiked covariance model \cite{johnstone2001distribution}:
\begin{align}\label{def:spike_model}
\Sigma(a) \equiv \mathrm{diag}\big(1+a_1,\ldots,1+a_p\big),
\end{align}
where $a=(a_1,\ldots,a_p) \in (-1,\infty)^p$. Write $\bar{a}=\sum_{j=1}^p a_j/p$. Specializing the results obtained in Sections \ref{section:test_identity} and \ref{section:test_spherical}, we have the following. 
\begin{corollary}\label{cor:spike_power}
The following hold. 
\begin{enumerate}
	\item The power for the likelihood ratio test of $\Sigma=I$ satisfies
	\begin{align*}
	\E_{\Sigma(a)} \Psi_{\lrt} \sim 1-\Phi\bigg(z_\alpha-\frac{\sum_{j=1}^p \big(a_j-\log(1+a_j)\big)}{ \sqrt{2\big(-\frac{p}{N}-\log\big(1-\frac{p}{N}\big)\big)} }\bigg)\equiv \beta_{\lrt}(a),
	\end{align*}
	under $N \wedge p \to \infty$ with $\limsup (p/N)< 1$.
	\item The power for Ledoit-Nagao-Wolf test of $\Sigma=I$ satisfies
	\begin{align*}
	\E_{\Sigma(a)}  \Psi_{\na}\sim  1-\Phi\bigg(z_\alpha- \frac{\sum_{j=1}^p a_j^2}{2(p/N)}\bigg)\equiv \beta_{\na}(a),
	\end{align*} 
	under $N \wedge p \to \infty$ with $\limsup (p/N)< \infty$. 
	\item The power for the likelihood ratio test of $\Sigma=\lambda I$ satisfies
	\begin{align*}
	\E_{\Sigma(a)} \Psi_{\lrt;s}  \sim 1-\Phi\bigg(z_\alpha-\frac{ \sum_{j=1}^p \log \frac{1+\bar{a}}{1+a_j} }{ \sqrt{2\big(-\frac{p}{N}-\log\big(1-\frac{p}{N}\big)\big)} }\bigg)\equiv \beta_{\lrt;s}(a),
	\end{align*}
	under $N \wedge p \to \infty$ with $\limsup (p/N)< 1$.
	\item The power for John's test of $\Sigma=\lambda I$ satisfies
	\begin{align*}
	\E_{\Sigma(a)}  \Psi_{\jo} \sim 1-\Phi\bigg(z_\alpha-\frac{\sum_{j=1}^p (a_j-\bar{a})^2/(1+\bar{a})^2}{2(p/N)}\bigg)\equiv \beta_{\jo}(a),
	\end{align*}
	under $N \wedge p \to \infty$ with $\limsup (p/N)< \infty$.
\end{enumerate}
\end{corollary}

(1), (2) and (4) above recover \cite[Proposition 8 (i)-(ii)]{onatski2014signal}, while (3)-(4) above recover \cite[Equations (4.5) and (4.8)]{wang2013sphericity}. Both \cite{onatski2013asymptotic,wang2013sphericity} considered the case where $r\equiv \pnorm{a}{0}$ and the non-zero elements of $a$ are fixed. The techniques in \cite{onatski2014signal} work with a further restriction $\pnorm{a}{\infty}< \sqrt{y}$ where $y$ is the limiting value of the ratio $p/N$. As mentioned in the introduction, this restriction coincides with the Baik-Ben Arous-P\'{e}ch\'{e} (BBP) phase transition \cite{baik2005phase}, and is essential for the techniques of \cite{onatski2014signal}, due to the singular nature of the likelihood ratio process when $\pnorm{a}{\infty}>\sqrt{y}$ already in the case $r=1$, see \cite[Theorem 8]{onatski2013asymptotic}. The restriction $\pnorm{a}{\infty}<\sqrt{y}$ is removed in \cite{wang2013sphericity} for the likelihood ratio test $\Psi_{\lrt;s}$ and John's test $\Psi_{\jo}$ for sphericity, by variations of Bai-Silverstein techniques developed in \cite{bai2004clt,bai2009corrections}. 

It is easy to see that in the setting of \cite{onatski2013asymptotic,wang2013sphericity} with a fixed number of spikes as described above, the asymptotic powers are the same for the following two group of tests:
\begin{enumerate}
	\item Likelihood ratio tests $\Psi_{\lrt},\Psi_{\lrt;s}$: $\beta_{\lrt}=\beta_{\lrt;s}$.
	\item Ledoit-Nagao-Wolf and John's tests: $\beta_{\na}=\beta_{\jo}$.
\end{enumerate}
Clearly, neither group of tests universally dominates the other in terms of the power behavior. For instance, the power of tests in (1) dominates that of (2) when some of $a_j$'s are close to $-1$ (i.e., $\Sigma$ is near singular), while the reversed phenomenon occurs when some of $a_j$'s are close to $\infty$.

In general, the asymptotic power equivalence of the above two groups may not hold when the number of spikes are no longer fixed. Instead, we have the following power ordering within each group.

\begin{corollary}\label{cor:spike_compare}
\begin{enumerate}
	\item Likelihood ratio tests $\Psi_{\lrt},\Psi_{\lrt;s}$ have the power ordering:
	\begin{align*}
	\beta_{\lrt} (a) \geq \beta_{\lrt;s}(a).
	\end{align*}
	\item Ledoit-Nagao-Wolf and John's tests $\Psi_{\na},\Psi_{\cm},\Psi_{\jo}$ have the power ordering: 
	\begin{align*}
	\beta_{\na}(a)
	\begin{cases}
	\geq \beta_{\jo}(a), &\overline{a^2}\big(1-(1+\bar{a})^2\big)\leq \bar{a}^2;\\
	< \beta_{\jo}(a), & \overline{a^2}\big(1-(1+\bar{a})^2\big)> \bar{a}^2.
	\end{cases}
	\end{align*} 
	Here $\overline{a^2}\equiv \sum_{j=1}^p a_j^2/p$.
\end{enumerate}
\end{corollary}
\begin{proof}
(1) follows from the inequality $\sum_{j=1}^p \log (1+\bar{a}) \leq \sum_{j=1}^p \bar{a}=\sum_{j=1}^p a_j$.  (2) follows by the following calculation:
\begin{align*}
\frac{\sum_{j=1}^p (a_j-\bar{a})^2}{(1+\bar{a})^2} = \frac{\sum_{j=1}^p (a_j^2-\bar{a}^2)}{(1+\bar{a})^2} = \sum_{j=1}^p a_j^2 +\frac{\sum_{j=1}^p a_j^2\cdot \big(1-(1+\bar{a})^2\big)-p\bar{a}^2}{(1+\bar{a})^2}.
\end{align*}
The proof is complete. 
\end{proof}

Note that $\{\bar{a}\geq 0\}\subsetneq \{ \overline{a^2}\big(1-(1+\bar{a})^2\big)\leq \bar{a}^2\}$ (the inclusion is in fact proper), so if $\bar{a}\geq 0$, John's test $\Psi_{\jo}$ will be less powerful than Ledoit-Nagao-Wolf's $\Psi_{\na}$. Furthermore, both inequalities in the above corollary can be strict asymptotically, and similar to the discussion above, there are no universal power dominance relationships between the tests in the two groups.

\section{Some spectral estimates}\label{section:spectral_estimate}

In this section, we will prove some spectral estimate for a class of high-dimensional random matrices that will be useful for the proofs of the results in Section \ref{section:application_highd_cov_test}.

First we introduce some convention on notation: Let $\mathcal{I}_1$, $\mathcal{I}_2$ be finite index sets. For $A=(A_{\iota_1,\iota_2})_{\iota_1 \in \mathcal{I}_1, \iota_2 \in \mathcal{I}_2} \in \R^{\mathcal{I}_1\times \mathcal{I}_2}$, its operator norm is defined as
\begin{align}\label{def:operator_norm_abstract}
\pnorm{A}{\op}\equiv \sup_{v \in B_{\mathcal{I}_2}(1)} \pnorm{A v}{\ell_2(\R^{\mathcal{I}_2} )  }.
\end{align}	
It can be readily verified that $\pnorm{A}{\op}=\sup_{u \in B_{\mathcal{I}_1}, v \in B_{\mathcal{I}_2}}\iprod{u}{Av}_{\mathcal{I}_1}$, and for a symmetric matrix $A \in \R^{\mathcal{I}_1\times \mathcal{I}_1}$, $\pnorm{A}{\op}=\sup_{u \in B_{\mathcal{I}_1} } \abs{\iprod{u}{Au}_{\mathcal{I}_1}}$. Here $\iprod{\cdot}{\cdot}_{\mathcal{I}_1}$ is the standard inner product on $\R^{\mathcal{I}_1}$. Clearly, the definition of the operator norm does not depend on the choice of the ordering of the index sets. 

Under this notational convention, with the index set $\Lambda \equiv \{(ij): i \in [N], j \in [p]\}$, we present below two results on the spectral norm of some special $\Lambda \times \Lambda$ matrices that are crucial to the proof of the quantitative CLTs. We do not specify a particular ordering on $\Lambda$ as we will be only interested in the operator norm as defined above. In the following we use $\mathbb{N}$ to denote the set of natural numbers. Recall the data matrix $X = [X_1,\ldots,X_N]^\top\in\R^{N\times p}$ and the definition of $S$ in (\ref{def:sample_covariance}).

\begin{proposition}\label{prop:U_spec_norm}
	\begin{enumerate}
		\item Suppose $p/N\leq 1-\epsilon$ for some $\epsilon>0$. For $\ell,m\in\mathbb{N}$ such that $\ell+m\geq 1$, let $U_{\ell,m}\in\R^{\Lambda \times \Lambda }$ be defined by
		\begin{align}\label{def:U_ell_m}
		(U_{\ell,m})_{(ij),(i'j')} \equiv N^{-1}X_i^\top S^{-\ell}X_{i'}(S^{-m})_{jj'}.
		\end{align}
		Then for any $q\in\mathbb{N}$, there exists some $C=C(\epsilon,\ell,m,q)>0$ such that $\E \pnorm{U_{\ell,m}}{\op}^q \leq C$ for $p\geq C$.
		\item When $X_i$, $S$ and $N$ is replaced by $X_i-\bar{X}$, $S_\ast$ and $n$, the conclusion of (1) still holds.
	\end{enumerate}
	
\end{proposition}

When the inverse $S^{-1}$ in (\ref{def:U_ell_m}) is replaced by $S$, the condition on $p/N$ can be substantially relaxed. 

\begin{proposition}\label{prop:U_spec_norm_+}
	Let $y\equiv p/N$. For $\ell,m\in\mathbb{N}$, let $U_{\ell,m;+}\in\R^{\Lambda \times \Lambda }$ be defined by
	\begin{align}\label{def:U_ell_m_+}
	(U_{\ell,m;+})_{(ij),(i'j')} \equiv N^{-1}X_i^\top S^\ell X_{i'}(S^{m})_{jj'}.
	\end{align}
	Then for any $q\in\mathbb{N}$, there exists some $C=C(\ell,m,q)>0$ such that $\E \pnorm{U_{\ell,m;+}}{\op}^q \leq C(\sqrt{y}\vee y)^{q(\ell+m+1)}$.
\end{proposition}

The proof of Proposition \ref{prop:U_spec_norm} relies crucially on the following stable moment estimate for $\pnorm{S^{-1}}{\op}$. Its proof utilizes two main technical tools: (i) rigidity estimates on the eigenvalues of the sample covariance matrix (cf. \cite{pillai2014universality}); (ii) closed form distributional formula of sample eigenvalues via zonal polynomials \cite[Chapter 9.7]{muirhead1982aspects}.

\begin{lemma}\label{lem:moment_S_inverse}
	Let $S_Z\equiv N^{-1}\sum_{i=1}^N Z_iZ_i^\top$ where $Z_i$'s are i.i.d. $\mathcal{N}(0,I)$ in $\R^p$.  Suppose $p/N\leq 1-\epsilon$ for some fixed $\epsilon>0$ and every $N,p\geq 2$. Then for any positive integer $q\leq (N-p-1)/8$, we have $\E \pnorm{S^{-1}_Z}{\op}^q\leq C$ for some positive $C = C(\epsilon,q)$.
\end{lemma}

\begin{proof}
	Write $S_Z$ for $S$ in the proof for simplicity. Let $\lambda$ be the smallest eigenvalue of $S$, and $y\equiv(p-1)/N < 1-\epsilon$. By \cite[Theorem 1.1]{rudelson2009smallest}, on an event $E$ with probability at least $1- e^{-cN(1-y)}$, $\lambda \geq c(1-\sqrt{y})^2$ for some absolute constant $c>0$. A similar estimate can be obtained using rigidity estimate for the eigenvalues of the sample covariance matrix, e.g., \cite[Theorem 3.1(iii)]{pillai2014universality}. Hence
	\begin{align}\label{ineq:moment_S_inv_1}
	\E \pnorm{S^{-1}}{\op}^q& = \E \pnorm{S^{-1}}{\op}^q \bm{1}_E + \E \pnorm{S^{-1}}{\op}^q \bm{1}_{E^c} \nonumber \\
	&\leq c^{-q} \big(1-\sqrt{y}\big)^{-2q}+ \E^{1/2} \pnorm{S^{-1}}{\op}^{2q}\cdot e^{-cN(1-y)/2}.
	\end{align}
	Now we give an upper bound for $\E\pnorm{S^{-1}}{\op}^{2q}$.  Let $r \equiv (N-p-1)/2$ assumed to be a positive integer. For any non-negative integer $k$, we write $\kappa\vdash k$ if $\kappa = (k_1,k_2,\ldots)$, with convention $k_1 \geq k_2\geq \ldots$, is a partition of $k$, i.e., $\sum_i k_i =k$. Let $C_\kappa$ denote the zonal polynomial (cf. \cite[Chapter 7]{muirhead1982aspects}) with respect to the partition $\kappa$. Then it follows from \cite[Corollary 9.7.4]{muirhead1982aspects} that, for any $x > 0$,
	\begin{align*}
	&\Prob(\pnorm{S^{-1}}{\op} > x)= 1 - \Prob(\lambda > 1/x)\\
	&= 1 -e^{-\frac{Np}{2x}}\sum_{k=0}^{pr}\sum_{\kappa\vdash k:k_1\leq r} \frac{C_\kappa\big(NI/(2x)\big)}{k!}\\
	&= e^{-\frac{Np}{2x}}\Big[\sum_{k=0}^\infty \frac{\big(Np/(2x)\big)^k}{k!} - \sum_{k=0}^{pr}\sum_{\kappa\vdash k:k_1\leq r} \frac{C_\kappa\big(NI/(2x)\big)}{k!}\Big]\\
	&= e^{-\frac{Np}{2x}}\Big[\sum_{k=pr+1}^\infty \frac{\big(Np/(2x)\big)^k}{k!} + \sum_{k=0}^{pr}\frac{1}{k!}\Big\{\bigg(\frac{Np}{2x}\bigg)^k - \sum_{\kappa\vdash k:k_1\leq r} C_\kappa\Big(\frac{NI}{2x}\Big)\Big\}\Big]\\
	&\stackrel{(*)}{=} e^{-\frac{Np}{2x}}\Big[\sum_{k=pr+1}^\infty \frac{\big(Np/(2x)\big)^k}{k!} + \sum_{k=0}^{pr}\frac{1}{k!}\sum_{\kappa\vdash k:k_1> r} C_\kappa\Big(\frac{NI}{2x}\Big) \Big]\\
	&\stackrel{(**)}{=} e^{-\frac{Np}{2x}}\Big[\sum_{k=pr+1}^\infty \frac{\big(Np/(2x)\big)^k}{k!} + \sum_{k=r+1}^{pr}\frac{\big(N/(2x)\big)^k}{k!}\sum_{\kappa\vdash k:k_1> r} C_\kappa\big(I\big)\Big)\Big]\\
	&\stackrel{(***)}{\leq }e^{-\frac{Np}{2x}}\cdot \sum_{k=r+1}^\infty \frac{\big(Np/(2x)\big)^k}{k!}.
	\end{align*}
	Here $(*)$ follows from \cite[Definition 7.2.1, (iii)]{muirhead1982aspects}: for any $k\geq 0$ and $t > 0$,
	\begin{align}\label{eq:zonal_identity}
	\sum_{\kappa\vdash k}C_\kappa(t\cdot I) = \big[\tr(t\cdot I)\big]^k = (tp)^k;
	\end{align}
	$(**)$ follows from the fact that for each $k$ and partition $\kappa$ of $k$, $C_\kappa$ is a homogeneous polynomial of order $k$; $(***)$ follows from the non-negativity of zonal polynomial for $I$ (cf. \cite[Corollary 7.2.4]{muirhead1982aspects}) and an application of (\ref{eq:zonal_identity}) with $t = 1$:
	\begin{align*}
	\sum_{\kappa\vdash k:k_1> r} C_\kappa\big(I\big)\leq \sum_{\kappa\vdash k} C_\kappa\big(I\big) = p^k.
	\end{align*}
	Hence by using the fact that for any $k\geq 2q+1$,
	\begin{align*}
	\int_0^\infty e^{-\frac{Np}{2x}}\bigg(\frac{Np}{2x}\bigg)^k\cdot x^{2q-1}\,\d x &=  \bigg(\frac{Np}{2}\bigg)^{2q}\int_0^\infty e^{-y}y^{k-2q-1}\,\d y \\
	&= \bigg(\frac{Np}{2}\bigg)^{2q}(k-2q-1)!,
	\end{align*}
	we have for every $r\geq 4q$
	\begin{align}\label{ineq:moment_S_inv_2}
	\E\pnorm{S^{-1}}{\op}^{2q}&= 2q\int_0^\infty x^{2q-1} \Prob(\pnorm{S^{-1}}{\op} > x)\, \d x \nonumber\\
	&\leq  \bigg(\frac{Np}{2}\bigg)^{2q}\sum_{k = r+1}^\infty \frac{1}{k(k-1)\cdots (k-2q)} \nonumber\\
	&= \bigg(\frac{Np}{2}\bigg)^{2q}\sum_{k = r+1}^\infty \frac{1}{2q}\cdot \Big\{\frac{1}{(k-1)\cdots (k-2q)}-\frac{1}{k\cdots (k-2q-1)}\Big\} \nonumber\\
	&= \bigg(\frac{Np}{2}\bigg)^{2q}\frac{1}{2q}\frac{1}{r(r-1)\cdots (r-2q+1)} \lesssim_q \frac{(Np)^{2q}}{r^{2q}}. 
	\end{align}
	Combining (\ref{ineq:moment_S_inv_1}) and (\ref{ineq:moment_S_inv_2}), as $p/N\leq 1-\epsilon$,
	\begin{align}\label{ineq:moment_S_inv_3}
	\E \pnorm{S^{-1}}{\op}^q \leq C_{\epsilon}^q+ C_\epsilon N^{q} e^{-c_\epsilon N} \lesssim_{q,\epsilon}1 
	\end{align}
	with $r=(N-p-1)/2$ being a positive integer. If $r$ is not an integer, write $S = \frac{N-1}{N}S'+\frac{1}{N}X_NX_N^\top$, where $S' \equiv \frac{1}{N-1}\sum_{i=1}^{N-1} X_iX_i^\top$. Then using Sherman-Morrison formula,
	\begin{align}\label{ineq:moment_S_inv_4}
	S^{-1} &= \frac{N}{N-1} (S')^{-1}-\frac{N}{(N-1)^2}\cdot \frac{(S')^{-1}X_NX_N^\top (S')^{-1}}{1 + \frac{1}{N-1}X_N^\top (S')^{-1} X_N}\nonumber\\
	& \equiv \frac{N}{N-1} (S')^{-1}- R.
	\end{align}
	As $X_N^\top (S')^{-1} X_N \geq \pnorm{X_N}{}^2/\lambda_{\max}(S')$, we have
	\begin{align*}
	\E \pnorm{R}{\op}^q &\leq \Big(\frac{N}{N-1}\Big)^q \cdot \E \bigg[\pnorm{ (S')^{-1}X_NX_N^\top (S')^{-1}}{\op}^{q}\frac{\lambda_{\max}(S')^{2q}}{\pnorm{X_N}{}^{2q}} \bigg]\\
	&\lesssim_q \E \Big( \frac{\lambda_{\max}(S')^{2q}}{\lambda_{\min}(S')^{2q}}\Big) \leq  \E^{1/2} \pnorm{S'}{\op}^{4q}\cdot\E^{1/2} \pnorm{(S')^{-1}}{\op}^{4q} \lesssim_{q,\epsilon} 1.
	\end{align*}
	The claim for  $r$ not being an integer follows from the decomposition (\ref{ineq:moment_S_inv_3}) and the estimate above.
\end{proof}

The following corollary of the Koltchinskii-Lounici theorem \cite{koltchinskii2017concentration} will also be repeatedly used.
\begin{lemma}\label{lem:kl}
	Let $S_Z\equiv N^{-1}\sum_{i=1}^N Z_iZ_i^\top$ where $Z_i$'s are i.i.d. $\mathcal{N}(0,I)$ in $\R^p$. Then for any positive integer $q$, there exists some positive $C = C(q)$ such that
	\begin{align*}
		\E\pnorm{S_Z - I}{\op}^q	\leq C\cdot\Big(\sqrt{\frac{p}{N}}\vee \frac{p}{N}\Big)^q.
	\end{align*}
\end{lemma}
\begin{proof}
This is a direct consequence of  \cite[Corollary 2]{koltchinskii2017concentration}.
\end{proof}

Now we prove Propositions \ref{prop:U_spec_norm} and \ref{prop:U_spec_norm_+}.

\begin{proof}[Proof of Proposition \ref{prop:U_spec_norm}]
	We only prove (1); claim (2) follows from completely same arguments by noting that (\ref{ineq:U_spec_norm_1}) and (\ref{ineq:U_spec_norm_2}) below still hold with the prescribed substitution. Note that $U_{\ell,m}$ is symmetric in that $(U_{\ell,m})_{(ij)(i^\prime j^\prime)} = (U_{\ell,m})_{(i^\prime j^\prime)(ij)}$, and satisfies that for any non-negative integers $(\ell_1,m_1)$ and $(\ell_2,m_2)$ such that $(\ell_1+m_1)\wedge (\ell_2+m_2)\geq 1$,
	\begin{align}\label{ineq:U_spec_norm_1}
	&(U_{\ell_1,m_1}\cdot U_{\ell_2,m_2})_{(ij),(i'j')}\nonumber\\
	& = N^{-2}\sum_{(\bar{i}\bar{j})} X_i^\top S^{-\ell_1} X_{\bar{i}} (S^{-m_1})_{j\bar{j}} X_{\bar{i}}^\top S^{-\ell_2} X_{i'} (S^{-m_2})_{\bar{j}j'} \nonumber \\
	& = N^{-2} X_i^\top S^{-\ell_1} \bigg(\sum_{\bar{i}} X_{\bar{i}}X_{\bar{i}}^\top\bigg) S^{-\ell_2} X_{i'}\cdot \bigg(\sum_{\bar{j}} (S^{-m_1})_{j\bar{j}}(S^{-m_2})_{\bar{j}j'}\bigg) \nonumber \\
	& = N^{-1} X_i^\top S^{-(\ell_1+\ell_2-1)} X_{i'} (S^{-(m_1+m_2)})_{jj'} \nonumber \\
	& = (U_{\ell_1+\ell_2-1,m_1+m_2})_{(ij),(i'j')}.
	\end{align}
    Consequently the above argument entails that for any $q\in \mathbb{N}$, $(U_{\ell,m})^q = U_{\ell^\prime, m'}$ with 
	\begin{align}\label{def:l_m_new}
	\ell^\prime\equiv \ell^\prime(q)\equiv q(\ell-1)+1, \quad m^\prime \equiv m^\prime(q)\equiv qm.
	\end{align}
	Using that $\pnorm{U_{\ell,m}}{\op}=\sup_{u \in B_{N\times p}(1)}\abs{\sum_{(ij)(i'j')} u_{ij} (U_{\ell,m})_{(ij),(i'j')} u_{i'j'} }$, we have 
	\begin{align*}
	\pnorm{U_{\ell,m}}{\op} &= \sup_{u\in B_{N\times p}(1)}N^{-1} \biggabs{ \sum_{i,i^\prime,j,j^\prime}X_{i}^\top S^{-\ell} X_{i^\prime} (S^{-m})_{jj'}u_{ij}u_{i^\prime j^\prime} } \nonumber\\
	&= N^{-1}\sup_{u\in B_{N\times p}(1)} \biggabs{ \sum_{i,i^\prime}X_{i}^\top S^{-\ell} X_{i^\prime}\cdot \sum_{jj'} u_{ij} \cdot (S^{-m})_{jj'} u_{i'j'}} \nonumber\\
	&= N^{-1}\sup_{u\in B_{N\times p}(1)} \biggabs{\sum_{i,i^\prime}X_{i}^\top S^{-\ell} X_{i^\prime}\cdot \Big[u \cdot S^{-m}\cdot u^\top\Big]_{ii'}}.
	\end{align*}
	As the $(i,i')$-th entry of $X^\top S^{-\ell} X$ is $X_i^\top S^{-\ell}X_{i'}$ and that $\tr(AB) = \sum_{i,i'=1}^N A_{ii'}B_{ii'}$ for two symmetric matrices in $\R^{N\times N}$, we have
	\begin{align*}
	\pnorm{U_{\ell,m}}{\op}&= N^{-1}\sup_{u\in B_{N\times p}(1)}\bigabs{ \tr\big[XS^{-\ell}X^\top \cdot uS^{-m}u^\top\big] }\\
	&= N^{-1}\sup_{u\in B_{N\times p}(1)}\bigabs{\tr\big[(u^\top X) S^{-\ell}(X^\top  u)\cdot S^{-m}\big]}.
	\end{align*}
	Further using twice the fact that $\tr(AB)\leq \tr(A)\pnorm{B}{\op}$ for any two p.s.d. and symmetric matrices $A,B$, we arrive at
	\begin{align}\label{ineq:U_spec_norm_2}
	\pnorm{U_{\ell,m}}{\op}&\leq N^{-1}\pnorm{S^{-m}}{\op}\pnorm{S^{-\ell}}{\op}\cdot \sup_{u\in B_{N\times p}(1)}\bigabs{ \tr\big(XX^\top uu^\top\big)}\nonumber\\
	&\leq \pnorm{S^{-(\ell+m)}}{\op} \cdot N^{-1}\pnorm{XX^\top}{\op}\cdot \sup_{u\in B_{N\times p}(1)}\tr(uu^\top)\nonumber\\
	& = \pnorm{S^{-(\ell+m)}}{\op} \cdot \pnorm{S}{\op} = \pnorm{S^{-1}}{\op}^{\ell+m-1}.
	\end{align}
	Hence for any $q\in\mathbb{N}$, equations (\ref{ineq:U_spec_norm_1})-(\ref{ineq:U_spec_norm_2}) and Lemma \ref{lem:moment_S_inverse} entail that
	\begin{align*}
	\E\pnorm{U_{\ell,m}}{\op}^{q} &= \E\pnorm{U_{\ell,m}^{q}}{\op} = \E\pnorm{U_{\ell',m'}}{\op} \leq \E\pnorm{S^{-1}}{\op}^{\ell'+m'-1}\\
	&=\E\pnorm{S^{-1}}{\op}^{q(\ell+m-1)} \leq C_{\ell,m,q},
	\end{align*}
	completing the proof.
\end{proof}

\begin{proof}[Proof of Proposition \ref{prop:U_spec_norm_+}]
	The proof largely follows that of Proposition \ref{prop:U_spec_norm} with modifications. We sketch the difference below. Using the same calculations as in (\ref{ineq:U_spec_norm_1}), we have
	\begin{align*}
	&(U_{\ell_1,m_1;+}\cdot U_{\ell_2,m_2;+})_{(ij),(i'j')}\\
	& = N^{-2}\sum_{(\bar{i}\bar{j})} X_i^\top S^{\ell_1} X_{\bar{i}} (S^{m_1})_{j\bar{j}} X_{\bar{i}}^\top S^{\ell_2} X_{i'} (S^{m_2})_{\bar{j}j'}\\
	& = N^{-2} X_i^\top S^{\ell_1} \bigg(\sum_{\bar{i}} X_{\bar{i}}X_{\bar{i}}^\top\bigg) S^{\ell_2} X_{i'}\cdot \bigg(\sum_{\bar{j}} (S^{m_1})_{j\bar{j}}(S^{m_2})_{\bar{j}j'}\bigg)\\
	& = N^{-1} X_i^\top S^{\ell_1+\ell_2+1} X_{i'} (S^{m_1+m_2})_{jj'} = (U_{\ell_1+\ell_2+1,m_1+m_2;+})_{(ij),(i'j')}.
	\end{align*}
	Hence for any $q\in \mathbb{N}$, $(U_{\ell,m;+})^q = U_{\ell', m';+}$ with $\ell'$ now defined by $\ell^\prime\equiv \ell^\prime(q)\equiv \ell + (q-1)(\ell+1)$ and $m'=qm$ remains the same as in (\ref{def:l_m_new}).
	
	Then using the same arguments as in (\ref{ineq:U_spec_norm_2}), we have $\pnorm{U_{\ell,m;+}}{\op} \leq \pnorm{S}{\op}^{\ell+m+1}$, hence for any $q \in \mathbb{N}$,
	\begin{align*}
	\E\pnorm{U_{\ell,m;+}}{\op}^{q} \leq \E\pnorm{S}{\op}^{\ell'+m'+1} \lesssim_{\ell,m,q} (\sqrt{y}\vee y)^{q(\ell+m+1)},
	\end{align*}
	where the last inequality follows by Lemma \ref{lem:kl} and the fact that $\ell'+m'+1 = \ell+(q-1)(\ell+1)+qm+1=q(\ell+m+1)$.
\end{proof}

\section{Proofs for Section \ref{section:test_identity} (testing identity)}\label{section:proof_identity}

\subsection{Proofs for Section \ref{section:LRT_covariance} (LRT)}\label{section:proof_LRT_covariance}

Recall $\Lambda = \{(ij): i \in [N], j \in [p]\}$. In the following sections, for a sufficiently smooth function $T:\R^{\Lambda} \rightarrow\R$, its gradient $\nabla T: \R^{\Lambda}\to \R^{\Lambda}$ and Hessian $\nabla^2 T: \R^{\Lambda}\to \R^{\Lambda \times \Lambda}$ are defined respectively by
\begin{align*}
\big(\nabla T(x)\big)_{(ij)} \equiv \frac{\partial T}{\partial x_{(ij)}}(x) \quad \text{ and } \quad \big(\nabla^2 T(x)\big)_{(ij),(i'j')}\equiv \frac{\partial^2 T}{\partial x_{(ij)}\partial x_{(i'j')}}(x),
\end{align*}
with $x = (x_{(ij)})\in\R^{\Lambda}$. Slightly abusing notation, we using $\pnorm{\nabla T(x)}{F}\equiv \pnorm{\nabla T(x)}{\ell_2(\R^\Lambda)}$. The operator norm $\pnorm{\nabla^2 T(x)}{\op}$ is defined in (\ref{def:operator_norm_abstract}).

\subsubsection{Evaluation of derivatives}

In the following, we use $\{e_j\}_{j=1}^p$ to represent the canonical basis in $\R^p$. Let $\delta_{ij}$ be the Kronecker delta.

\begin{lemma}\label{lem:derivatives_T}
	Recall the form of $T_{\lrt}(X)$ in (\ref{def:lrs}). We assume without loss of generality that $\mu=0$. Then for any $(i,j),(i^\prime,j^\prime)\in [N]\times [p]$,
	\begin{enumerate}
		\item $\big(\nabla T_{\lrt}(X)\big)_{(i j)} = \big(X(I - S^{-1})\big)_{(ij)} = e_j^\top(I-S^{-1})X_i$. 
		\item $\big(\nabla^2T_{\lrt}(X)\big)_{(i j),(i^\prime j^\prime)} = N^{-1}X_i^\top S^{-1}(e_{j^\prime}X_{i^\prime}^\top + X_{i^\prime}e_{j^\prime}^\top)S^{-1}e_j + \delta_{ii^\prime}e_j^\top(I-  S^{-1})e_{j^\prime}$. 
	\end{enumerate}
\end{lemma}
\begin{proof}
	We use $T$ as a shorthand for $T_{\lrt}$. 
	
	\noindent (1). By definition, we have
	\begin{align*}
	\partial_{(i j)} T(X) = \frac{N}{2}\big(\partial_{(i j)} \tr(S) - \partial_{(ij)} \log\det S\big).
	\end{align*}
	For the first partial derivative, using $\partial_{(i j)}X_k = \delta_{ik}e_j$, we have
	\begin{align}\label{eq:first_order_1}
	\notag&\partial_{(ij)} \tr(S(X)) = N^{-1} \sum_k \frac{\partial}{\partial X_{ij}}\tr(X_kX_k^\top)\\
	&= N^{-1} \sum_k \delta_{ik} \tr\big[ e_j X_k^\top+ X_k e_j^\top \big] = N^{-1} \sum_k \delta_{ik} \cdot 2X_{kj} = 2N^{-1}X_{ij}.
	\end{align}
	For the second partial derivative, using the well-known fact that $\nabla \log\det A = A^{-1}$ for any invertible and symmetric matrix $A$ (see e.g., \cite[Section A.4.1]{boyd2004convex}), we have
	\begin{align}\label{eq:first_order_2}
	\notag\partial_{(ij)} \log\det S &= \sum_{k,\ell}\frac{\partial \log\det S}{\partial S_{k\ell}}\frac{\partial S_{k\ell}}{\partial X_{ij}}\stackrel{(*)}{=} \sum_{k,\ell} (S^{-1})_{k\ell}\cdot \frac{1}{N}(\delta_{jk}X_{i\ell} + \delta_{j\ell}X_{ik})\\
	& = \frac{1}{N}\sum_\ell (S^{-1})_{j\ell}X_{i\ell} + \sum_k (S^{-1})_{kj}X_{ik} = \frac{2}{N}(XS^{-1})_{ij},
	\end{align}
	where in $(*)$, we use
	\begin{align}\label{eq:partial_S}
	\frac{\partial S_{k\ell}}{\partial X_{ij}} = \frac{1}{N}\frac{\partial}{\partial X_{ij}}\iprod{Xe_k}{Xe_\ell} =\frac{1}{N}(\delta_{jk}X_{i\ell} + \delta_{j\ell}X_{ik}).
	\end{align}
	Combining (\ref{eq:first_order_1}) and (\ref{eq:first_order_2}) yields the first claim.
	
	\noindent (2). Again by definition, we have
	\begin{align*}
	\partial_{(i j)(i^\prime j^\prime)} T(X) = \frac{N}{2}\big(\partial_{(i j)(i^\prime j^\prime)} \tr(S) - \partial_{(i j)(i^\prime j^\prime)} \log\det S\big).
	\end{align*}
	For the first derivative, it follows from (\ref{eq:first_order_1}) that
	\begin{align}\label{eq:sec_order_1}
	\partial_{(i j)(i^\prime j^\prime)} \tr(S) = 2N^{-1}\partial_{(i^\prime j^\prime)} X_{ij} = 2N^{-1}\delta_{ii^\prime}\delta_{jj^\prime}.
	\end{align}
	For the second derivative, it follows from (\ref{eq:first_order_2}) that
	\begin{align}\label{eq:sec_order_2}
	&\notag\partial_{(i j)(i^\prime j^\prime)} \log\det S\\
	&= \frac{2}{N}\frac{\partial }{\partial {X_{i^\prime j^\prime}}} X_i^\top S^{-1}e_j= \frac{2}{N}\Big(\frac{\partial X_i}{\partial X_{i^\prime j^\prime}}^\top S^{-1}e_j + X_i^\top\frac{\partial S^{-1}}{\partial X_{i^\prime j^\prime}}e_j\Big) \nonumber\\
	&\stackrel{(**)}{=} \frac{2}{N}\Big(\delta_{ii^\prime}e_{j^\prime}^\top S^{-1}e_j - N^{-1} X_i^\top S^{-1}\big(e_{j^\prime}X^\top_{i^\prime} +  X_{i^\prime}e_{j^\prime}^\top\big)S^{-1}e_j\Big),
	\end{align}
	where in $(**)$ we use the following calculation with the help of (\ref{eq:partial_S}):
	\begin{align}\label{eq:sec_order_3}
	\frac{\partial S^{-1}}{\partial X_{i^\prime j^\prime}} &= -S^{-1}\frac{\partial S}{\partial X_{i^\prime j^\prime}}S^{-1} = -S^{-1}\Big(\sum_{k\ell}e_ke_\ell^\top\frac{\partial S_{k\ell}}{\partial X_{i^\prime j^\prime}}\Big) S^{-1} \nonumber \\
	&= -\frac{1}{N}S^{-1}\cdot\Big(\sum_{k\ell}e_ke_\ell^\top(\delta_{j^\prime k}X_{i^\prime \ell} + \delta_{j^\prime \ell}X_{i^\prime k})\Big)S^{-1}\nonumber\\
	&= -\frac{1}{N} S^{-1}\big(e_{j^\prime}X_{i^\prime}^\top + X_{i^\prime}e_{j^\prime}^\top\big) S^{-1}.
	\end{align}
	We obtain the second claim by combining (\ref{eq:sec_order_1}) and (\ref{eq:sec_order_2}).
\end{proof}

\subsubsection{Normal approximation}\label{subsec:LRT_covariance_clt}

\begin{proof}[Proof of Proposition \ref{prop:clt_covariance_test}]
	We again shorthand $T_{\lrt;\Sigma}$ by $T$. By Lemma \ref{lem:derivatives_T}, 
	\begin{align*}
	\pnorm{\nabla T(X)}{F}^2 & = \sum_{i,j} \big(\partial_{(ij)}T(X) \big)^2 = \sum_i \pnorm{(I-S^{-1})X_i}{}^2\leq \pnorm{S^{-1}}{\op}^2 \pnorm{I-S}{\op}^2 \sum_i \pnorm{X_i}{}^2.
	\end{align*}
	Using Lemma \ref{lem:moment_S_inverse} and Lemma \ref{lem:kl},
	\begin{align}\label{ineq:clt_covariance_test_1}
	\E \pnorm{\nabla T(X)}{F}^4& \lesssim \E \bigg(\pnorm{S^{-1}}{\op}^2 \pnorm{I-S}{\op}^2 \sum_i \pnorm{X_i}{}^2\bigg)^2 \nonumber\\
	& = \sum_{i,i'} \E\bigg[\pnorm{S^{-1}}{\op}^4 \pnorm{I-S}{\op}^4 \pnorm{X_i}{}^2 \pnorm{X_{i'}}{}^2\bigg] \nonumber \\
	& \leq \sum_{i,i'} \E^{1/4}\pnorm{S^{-1}}{\op}^{16}\cdot \E^{1/4} \pnorm{I-S}{\op}^{16}\cdot \E^{1/4}  \pnorm{X_i}{}^8\cdot  \E^{1/4}  \pnorm{X_{i'}}{}^8 \nonumber\\
	& \lesssim N^2\cdot \Big(\sqrt{\frac{p}{N}}\Big)^4\cdot p\cdot p = p^4.
	\end{align}
	Again by Lemma \ref{lem:derivatives_T}, the second derivatives are 
	\begin{align*}
	\partial_{(ij),(i'j')}T(X)
	& = N^{-1}X_i^\top S^{-1}e_{j'} X_{i'}^\top S^{-1} e_j  + N^{-1}X_i^\top S^{-1} X_{i'} (S^{-1})_{jj'}+ \delta_{ii'} \big(I-S^{-1}\big)_{jj'}\\
	& \equiv  \big(T_1+T_2+T_3\big)_{(ij),(i'j')}.
	\end{align*}
	Recall the definition of $U_{\ell,m}$ in Proposition \ref{prop:U_spec_norm}-(1). Then
	\begin{align}\label{ineq:clt_covariance_test_2}
	(T_1^2)_{(ij),(i'j')}& = N^{-2} \sum_{(\bar{i}\bar{j})} X_i^\top S^{-1} e_{\bar{j}} \cdot X_{\bar{i}}^\top S^{-1} e_j \cdot X_{\bar{i}}^\top S^{-1} e_{j'}\cdot X_{i'}^\top S^{-1} e_{\bar{j}} \nonumber\\
	& = N^{-2} \bigg(\sum_{\bar{i}} e_j^\top S^{-1} X_{\bar{i}} X_{\bar{i}}^\top S^{-1} e_{j'}\bigg)\cdot \bigg(\sum_{\bar{j}} X_i^\top S^{-1} e_{\bar{j}}e_{\bar{j}}^\top S^{-1} X_{i'}\bigg) \nonumber\\
	& = N^{-1} (S^{-1})_{jj'}  \cdot X_i^\top S^{-2} X_{i'}= (U_{2,1})_{(ij),(i'j')},
	\end{align}
	and $T_2 = U_{1,1}$.  Proposition \ref{prop:U_spec_norm}-(1) entails that
	\begin{align}\label{ineq:clt_covariance_test_2.1}
	\E \pnorm{T_1}{\op}^4 \vee \E\pnorm{T_2}{\op}^4 = \mathcal{O}(1).
	\end{align} 
	On the other hand, $T_3$ has a block diagonal structure with respect to the index $(i,i')$, so its spectral norm equals that of $(I-S^{-1}) \in \R^{p\times p}$, and hence
	\begin{align}\label{ineq:clt_covariance_test_3}
	\E \pnorm{T_3}{\op}^4 = \E \pnorm{I-S^{-1}}{\op}^4\leq \E \big[\pnorm{S^{-1}}{\op}^4 \pnorm{I-S}{\op}^4\big] = \mathcal{O}(1).
	\end{align}
	Combining all the estimates above, we find that
	\begin{align}\label{ineq:clt_covariance_test_4}
	\E \pnorm{\nabla^2 T(X)}{\op}^4 \lesssim \E \pnorm{T_1}{\op}^4+  \E \pnorm{T_2}{\op}^4+ \E \pnorm{T_3}{\op}^4=  \mathcal{O}(1).
	\end{align}
	Let $X'$ be an independent copy of $X$ and let $X_t' \equiv \sqrt{t}X+\sqrt{1-t}X' \in \R^{N\times p}$. Let $\E'$ denote expectation only with respect to $X'$ and 
	\begin{align*}
	\bar{T}(X)&\equiv \int_0^1 \frac{1}{2\sqrt{t}} \iprod{\nabla T(X)}{\E' \nabla T(X_t')}\ \d{t}.
	\end{align*}
	Then by the Gaussian-Poincar\'e inequality
	\begin{align*}
	\var\big(\bar{T}(X)\big)&\leq \E \pnorm{\nabla \bar{T}(X)}{F}^2\lesssim \sqrt{\E \pnorm{\nabla^2 T (X)}{\op}^4} \sqrt{\E \pnorm{\nabla T(X)}{F}^4}\lesssim p^{2}.
	\end{align*}
	The claim now follows from the second-order Poincar\'e inequality in Lemma \ref{lem:sec_poincare} and Proposition \ref{prop:ratio_covariance_test}-(4).
\end{proof}

\subsubsection{Contiguity}\label{subsec:LRT_covariance_ratio}
\begin{proof}[Proof of Proposition \ref{prop:ratio_covariance_test}]
We shorthand $(T_{\lrt}, m_{\Sigma;\lrt},\sigma_{\Sigma;\lrt}, V_{\Sigma;\lrt})$ by $(T, m_{\Sigma},\sigma_{\Sigma}, V_{\Sigma})$.

\noindent (1). Recall that $Z_1,\ldots,Z_n$ are i.i.d. samples from $\mathcal{N}(0,I_p)$. By Lemma \ref{lem:derivatives_T}, with $S_Z\equiv N^{-1}\sum_{i=1}^N Z_iZ_i^\top$, we have
\begin{align*}
\mathscr{T}_\Sigma(Z)=Z\Sigma^{1/2}\big(I-\Sigma^{-1/2}S_Z^{-1}\Sigma^{-1/2}\big)\Sigma^{1/2} = Z(\Sigma-S_Z^{-1}).
\end{align*} 
Hence with $\{\lambda_j\}_{j=1}^p$ denoting the eigenvalues of $\Sigma$, we have
\begin{align*}
V_\Sigma^2& = \E \pnorm{ \mathscr{T}_\Sigma(Z)-\mathscr{T}_I(Z)}{F}^2=\E \pnorm{Z(\Sigma-I)}{F}^2\\
& = \E \tr \big((\Sigma-I)Z^\top Z (\Sigma-I)\big)= \tr \big(\E Z^\top Z (\Sigma-I)^2\big)\\
&= N \pnorm{\Sigma-I}{F}^2 = N \sum_j (\lambda_j-1)^2.
\end{align*}

\noindent (2). Note that
\begin{align*}
m_\Sigma&= (N/2)\E \big[\tr \big(\Sigma^{1/2}S_Z\Sigma^{1/2}\big)-\log \det (\Sigma^{1/2}S_Z\Sigma^{1/2})-p\big]\\
& = (N/2)\big[\tr(\Sigma)-\log \det \Sigma-p\big]-(N/2)\E \log \det S_Z,
\end{align*}
so
\begin{align*}
m_\Sigma-m_I &= (N/2)\mathcal{L}_S(\Sigma,I)\\
& = (N/2) \sum_{j=1}^p \big(\lambda_j -\log \lambda_j -1\big) \gtrsim N \sum_{j=1}^p \big[\abs{\lambda_j-1}\wedge (\lambda_j-1)^2\big].
\end{align*}

\noindent (3). It is shown by the proof of \cite[Theorem 1]{chen2018study} that with $\mu_{n,0},\sigma_{n,0}^2$ defined in \cite[Corollary 1]{chen2018study}, and $Y_n\equiv \big(T(X)-\mu_{n,0}\big)/(n \sigma_{n,0})$, for $s \in (-s_0,s_0)$ for some $s_0>0$, 
\begin{align*}
\lim_{n\wedge p\to \infty, n\geq p+2} M_{Y_n}(s) = M_{\mathcal{N}(0,1)}(s)= e^{s^2/2},
\end{align*}
where $M_Y(s)\equiv\E e^{sY}$ denotes the moment generating function of a generic random variable $Y$. Now using that for any $s \in (0,s_0)$,
\begin{align*}
\E Y_n^4 =4 \int_0^\infty t^3\Prob\big(Y_n>t\big)\,\d{t} \leq 4\int_0^\infty t^3 e^{-st} M_{Y_n}(s)\,\d{t} = (6/s^3) M_{Y_n}(s),
\end{align*}
it follows that  $\sup_n \E Y_n^4<\infty$, and hence convergence of moments yields that $\E Y_n \to 0, \E Y_n^2 \to 1$. This implies $
{\sigma_I^2}/(n^2 \sigma_{n,0}^2 )= \var(Y_n) = \E Y_n^2 - (\E Y_n)^2 \to 1$. Hence the asymptotic formula for $\sigma_I^2$ holds. In particular, this means that there exists some sufficiently large $M$ such that for $n\wedge p\geq M, n\geq p+2$, 
\begin{align*}
\sigma_I^2 &\geq M^{-1}\cdot n^2 \sigma_{n,0}^2 = \frac{1}{2M}\cdot n^2 \bigg[-\frac{p}{N}-\log\bigg(1-\frac{p}{N}\bigg)\bigg]\stackrel{(\ast)}{\geq} \frac{1}{4M} n^2\cdot \frac{p^2}{N^2} \geq \frac{p^2}{4M},
\end{align*}
where in $(\ast)$ we used the inequality $-x-\log(1-x)\geq x^2/2$ that holds for $x \in (0,1)$.	

\noindent (4). Recall that $\{\lambda_j\}_{j=1}^p$ are eigenvalues of $\Sigma$. By (1)-(3), we only need to show that for some universal $C>0$,
\begin{align}\label{ineq:power_reg_cond}
\frac{\sqrt{N \sum_j (\lambda_j-1)^2}}{\big(N \sum_j \big(\abs{\lambda_j-1}\wedge (\lambda_j-1)^2\big)\big)\vee \sigma_I}\leq \frac{C}{ \big(\sigma_I \wedge N\big)^{1/2}}.
\end{align}
To see this, let $\nu_j \equiv \abs{\lambda_j-1}$, and $J\equiv\{j\in[p]: \nu_j\leq 1\}$, it suffices to prove
\begin{align}\label{ineq:power_reg_cond_1}
\frac{\sqrt{N \sum_{j \in J} \nu_j^2} \vee \sqrt{N \sum_{j \in J^c} \nu_j^2}}{\big(N\sum_{j \in J} \nu_j^2\big) \vee \big(N \sum_{j \in J^c} \nu_j \big) \vee \sigma_I}\leq \frac{C}{ \big(\sigma_I \wedge N\big)^{1/2}}.
\end{align}
This follows as
\begin{align*}
\hbox{LHS of (\ref{ineq:power_reg_cond_1})}&\leq \frac{\sqrt{N \sum_{j \in J} \nu_j^2}}{ \big(N\sum_{j \in J} \nu_j^2\big)  \vee \sigma_I} + \frac{\sqrt{N } \sum_{j \in J^c} \nu_j}{ \big(N \sum_{j \in J^c} \nu_j \big) \vee \sigma_I}\\
& \leq \frac{1}{\inf_{x\geq 0} \big(x \vee \frac{\sigma_I}{x}\big) } + N^{-1/2} \lesssim \big(\sigma_I\wedge N\big)^{-1/2}.
\end{align*}
The proof is complete.
\end{proof}

\subsection{Proofs for Section \ref{section:nagao} (Ledoit-Nagao-Wolf's test)}

\subsubsection{Evaluation of derivatives}

\begin{lemma}\label{lem:derivatives_T_nagao}
	Recall the form of $T_{\na}(X)$ in (\ref{def:lrs_nagao}). We assume without loss of generality that $\mu=0$. Then for any $(i,j),(i^\prime,j^\prime)\in [N]\times [p]$,
	\begin{enumerate}
		\item $\big(\nabla T_{\na}(X)\big)_{(i j)} = \big(X(S-I) - (\tr(S)/N)X\big)_{ij} = e_j^\top(S-I)X_i-(\tr(S)/N)X_{ij}$. 
		\item $\big(\nabla^2T_{\na}(X)\big)_{(i j)(i^\prime j^\prime)} = N^{-1} \delta_{jj'} X_i^\top X_{i'}+ N^{-1} X_{i'j} X_{ij'} +\delta_{ii'} (S-I)_{jj'}-(2/N^2)X_{ij}X_{i'j'}- (\tr(S)/N)\delta_{ii'}\delta_{jj'}$. 
	\end{enumerate}
	Furthermore, for any $(i_\ell,j_\ell) \in [N]\times [p], \ell=1,2,3,4$,
	\begin{align*}
&\partial_{(i_1j_1)(i_2j_2)(i_3j_3)(i_4j_4)}T_{\na}(X)\\
& = N^{-1}\big(\delta_{i_1i_3}\delta_{i_2i_4}\delta_{j_1j_2}\delta_{j_3j_4}+\delta_{i_1i_4}\delta_{i_2i_3}\delta_{j_1j_2}\delta_{j_3j_4}\\
&\qquad + \delta_{i_1i_4}\delta_{i_2i_3}\delta_{j_1j_3}\delta_{j_2j_4}+\delta_{i_1i_3}\delta_{i_2i_4}\delta_{j_1j_4}\delta_{j_2j_3}\\
&\qquad + \delta_{i_1i_2}\delta_{i_3i_4}\delta_{j_1j_3}\delta_{j_2j_4}+\delta_{i_1i_2}\delta_{i_3i_4}\delta_{j_1j_4}\delta_{j_2j_3}\big)\\
&\qquad -2N^{-2}\big(\delta_{i_1i_3}\delta_{i_2i_4}\delta_{j_1j_3}\delta_{j_2j_4}+\delta_{i_1i_4}\delta_{i_2i_3}\delta_{j_1j_4}\delta_{j_2j_3}+\delta_{i_1i_2}\delta_{i_3i_4}\delta_{j_1j_2}\delta_{j_3j_4}\big).
\end{align*}
\end{lemma}
\begin{proof}
	We shorthand $T_{\na}$ as $T$. As $\partial_{ij}S(X) = N^{-1}(e_jX_i^\top+X_i e_j^\top)$, for the first-order derivatives we have
	\begin{align*}
	\partial_{(ij)} T(X)& = \frac{N}{4}\bigg(\tr \big[\partial_{(ij)}(S-I)^2\big] -\frac{1}{N}\cdot 2\tr(S) \tr\big[\partial_{(ij)} S\big]\bigg)\\
	& =  \frac{1}{2}\tr\big[(S-I)(e_jX_i^\top+X_i e_j^\top)\big]-\frac{\tr(S)X_{ij}}{N}\\
	& = \big(X(S-I)\big)_{ij} -\frac{\tr(S)}{N}X_{ij}= e_j^\top(S-I)X_i-\frac{\tr(S)}{N}X_{ij}.
	\end{align*}
	For the second-order derivatives we have
	\begin{align*}
	&\partial_{(ij),(i'j')}T(X) = \partial_{(i'j')} \big(e_j^\top(S-I)X_i\big)-N^{-1} \partial_{(i'j')}\big(\tr(S)X_{ij}\big) \\
	& = N^{-1} e_j^\top (e_{j'}X_{i'}^\top+X_{i'} e_{j'}^\top)\cdot X_i + \delta_{ii'} e_j^\top (S-I) e_{j'}  - N^{-1}\big((2/N)X_{ij}X_{i'j'}+\delta_{ii'}\delta_{jj'}\tr(S)\big) \\
	& = N^{-1} \delta_{jj'} X_i^\top X_{i'}+ N^{-1} X_{i'j} X_{ij'} +\delta_{ii'} (S-I)_{jj'}  -2N^{-2}X_{ij}X_{i'j'}- N^{-1}\tr(S)\delta_{ii'}\delta_{jj'}.
	\end{align*}
	For the third-order derivatives we have
	\begin{align*}
	&\partial_{(i_1j_1)(i_2j_2)(i_3j_3)}T(X)\\
	& = N^{-1}\delta_{j_1j_2} \partial_{(i_3j_3)}(X_{i_1}^\top X_{i_2})+N^{-1} \partial_{(i_3j_3)}(X_{i_2j_1}X_{i_1j_2})\\
	&\qquad +N^{-1}\delta_{i_1i_2} e_{j_1}^\top \big(e_{j_3}X_{i_3}^\top+X_{i_3} e_{j_3}^\top\big) e_{j_2}  -2N^{-2} \partial_{(i_3j_3)}\big(X_{i_1j_1}X_{i_2j_2}\big)-2N^{-2} \delta_{i_1i_2}\delta_{j_1j_2} X_{i_3j_3}\\
	& = N^{-1} \big(\delta_{i_1i_3} \delta_{j_1j_2}X_{i_2j_3}+\delta_{i_2i_3} \delta_{j_1j_2}X_{i_1j_3} \big) \\
	& \qquad + N^{-1} \big(\delta_{i_2i_3}\delta_{j_1j_3} X_{i_1j_2}+ \delta_{i_1i_3}\delta_{j_2j_3} X_{i_2j_1}\big) + N^{-1} \big(\delta_{i_1i_2}\delta_{j_1j_3} X_{i_3j_2}+ \delta_{i_1i_2}\delta_{j_2j_3}X_{i_3j_1}\big)\\
	&\qquad -2N^{-2}\big(\delta_{i_1i_3}\delta_{j_1j_3}X_{i_2j_2}+\delta_{i_2i_3}\delta_{j_2j_3}X_{i_1j_1}+ \delta_{i_1i_2}\delta_{j_1j_2} X_{i_3j_3} \big).
	\end{align*}
	For the fourth-order derivatives we have
	\begin{align*}
	&\partial_{(i_1j_1)(i_2j_2)(i_3j_3)(i_4j_4)}T(X)\\
	& = N^{-1}\big(\delta_{i_1i_3}\delta_{i_2i_4}\delta_{j_1j_2}\delta_{j_3j_4}+\delta_{i_1i_4}\delta_{i_2i_3}\delta_{j_1j_2}\delta_{j_3j_4}\\
	&\qquad + \delta_{i_1i_4}\delta_{i_2i_3}\delta_{j_1j_3}\delta_{j_2j_4}+\delta_{i_1i_3}\delta_{i_2i_4}\delta_{j_1j_4}\delta_{j_2j_3}\\
	&\qquad + \delta_{i_1i_2}\delta_{i_3i_4}\delta_{j_1j_3}\delta_{j_2j_4}+\delta_{i_1i_2}\delta_{i_3i_4}\delta_{j_1j_4}\delta_{j_2j_3}\big)\\
	&\qquad -2N^{-2}\big(\delta_{i_1i_3}\delta_{i_2i_4}\delta_{j_1j_3}\delta_{j_2j_4}+\delta_{i_1i_4}\delta_{i_2i_3}\delta_{j_1j_4}\delta_{j_2j_3}+\delta_{i_1i_2}\delta_{i_3i_4}\delta_{j_1j_2}\delta_{j_3j_4}\big).
	\end{align*}
	The proof is complete.
\end{proof}

\subsubsection{Normal approximation}\label{subsec:nagao_clt}	

	\begin{proof}[Proof of Proposition \ref{prop:clt_covariance_test_nagao}]
	Let $y\equiv p/N$. We start by showing that 
	\begin{align}\label{ineq:nagao_clt_1}
	\E\pnorm{\nabla^2 T(X)}{\op}^4 \leq C(1\vee y)^4
	\end{align}
	for some absolute constant $C>0$. Reorganizing the terms in Lemma \ref{lem:derivatives_T_nagao}, we have
	\begin{align*}
	\big(\nabla^2 T(X)\big)_{(ij),(i'j')} &= N^{-1}X_i^\top X_{i'}\delta_{jj'} + N^{-1}X_{ij'}X_{i'j} - 2N^{-2}X_{i'j'}X_{ij}\\
	&\quad\quad\quad + \delta_{ii'}e_{j'}^\top (S - I - N^{-1}\tr(S)I)e_j\\
	&\equiv (T_{2,1} + T_{2,2} - T_{2,3} + T_{2,4})_{(ij),(i'j')}.
  	\end{align*} 
	Recall the definition of $U_{\ell,m;+}$ from Proposition \ref{prop:U_spec_norm_+}. As $T_{2,1} = U_{0,0;+}$ and
	\begin{align*}
	(T_{2,2}^2)_{(ij),(i'j')}& = N^{-2}\sum_{(\bar{i}\bar{j})} X_{\bar{i}j}X_{i\bar{j}}X_{i'\bar{j}}X_{\bar{i}j'}= N^{-2} \bigg(\sum_{\bar{i}} X_{\bar{i}j} X_{\bar{i}j'} \bigg)\bigg(\sum_{\bar{j}} X_{i\bar{j}}X_{i'\bar{j}}\bigg)\\
	& = N^{-1} S_{jj'} X_i^\top X_{i'} = (U_{0,1;+})_{(ij),(i'j')},
	\end{align*}
	Proposition \ref{prop:U_spec_norm_+} entails that $\E \pnorm{T_{2,1}}{\op}^4 \vee \E \pnorm{T_{2,2}}{\op}^4=\mathcal{O}((1\vee y)^4)$. For $T_{2,3}$, as
	\begin{align*}
	\pnorm{T_{2,3}}{\op}& = (2/N^2)\sup_{u,v \in B_{N\times p}} \biggabs{\sum_{(ij),(i'j')} u_{ij} X_{ij}X_{i'j'} v_{i'j'}} = (2/N^2) \pnorm{X}{F}^2.
	\end{align*}
	Hence $\E \pnorm{T_{2,3}}{\op}^4 =\mathcal{O}(y^4) = \mathcal{O}((1\vee y)^4)$. For $T_{2,4}$, it holds by the block diagonal structure that
	\begin{align*}
	\pnorm{T_{2,4}}{\op} = \pnorm{S - I - N^{-1}\tr(S)I}{\op} \leq \pnorm{S-I}{\op} + N^{-1}\tr(S).
	\end{align*}
	Hence it holds by Lemma \ref{lem:kl} that
	\begin{align*}
	\E\pnorm{T_{2,4}}{\op}^4 \lesssim \big(y\vee\sqrt{y}\big)^4 + N^{-4}\cdot N^{-4}\E\pnorm{X}{F}^8 \lesssim (1\vee y)^4.
	\end{align*}
	By collecting the estimates of $T_{2,1}$-$T_{2,4}$, we complete the proof of (\ref{ineq:nagao_clt_1}).
	
	Next we show that $\E\pnorm{\nabla T(X)}{F}^4 \lesssim p^4$. This will be done by two estimates below.

	\noindent (\textbf{Estimate 1}) By Lemma \ref{lem:derivatives_T_nagao}-(1),
	\begin{align*}
	\pnorm{\nabla T(X)}{F}^2 & \lesssim \sum_i \pnorm{(S-I)X_i}{}^2+N^{-2}\tr^2(S)\pnorm{X}{F}^2\\
	&\leq \big(\pnorm{S-I}{\op}^2+ N^{-2}\tr^2(S)\big)\sum_i \pnorm{X_i}{}^2,
	\end{align*}
	so by Lemma \ref{lem:kl} and Proposition \ref{prop:U_spec_norm_+},
	\begin{align*}
	\E \pnorm{\nabla T(X)}{F}^4 &\lesssim \E \Big[ \big(\pnorm{S-I}{\op}^2+N^{-2}\tr^2(S)\big) \sum_i \pnorm{X_i}{}^2 \Big]^2\\
	& \lesssim \sum_{i,i'}  \E \big[ \big(\pnorm{S-I}{\op}^4+N^{-4}\tr^4(S)\big) \pnorm{X_i}{}^2\pnorm{X_{i'}}{}^2 \big]\\
	&\leq \sum_{i,i'} \big(\E^{1/2} \pnorm{S-I}{\op}^8+ N^{-4}\E^{1/2}\tr^8(S) \Big)\cdot \E^{1/4}\pnorm{X_i}{}^8\cdot \E^{1/4}\pnorm{X_{i'}}{}^8\\
	&\lesssim N^2\cdot \Big[\Big(\frac{p}{N}\Big)^2+\Big(\frac{p}{N}\Big)^4\cdot \E^{1/2} \pnorm{S}{\op}^8\Big]\cdot p\cdot p \lesssim p^4(1+y^6).
	\end{align*}
	\noindent (\textbf{Estimate 2})
	Note that
	\begin{align*}
	\nabla T(X) = X\big(S - N^{-1}\tr(S)I\big) - X \equiv T_{1,1} + T_{1,2}.
	\end{align*}
	It is clear that $\E\pnorm{T_{1,2}}{F}^2 \lesssim Np $. To handle $T_{1,1}$, note that
	\begin{align*}
	\pnorm{T_{1,1}}{F}^2 &= N\tr\Big(\big(S - N^{-1}\tr(S)I\big)^2 S\Big)\\
	& = N\tr\Big(S^3 + N^{-2}\tr^2(S)S - 2N^{-1}\tr(S)S^2\Big)\\
	&= N\Big[\tr(S^3) + N^{-2}\tr^3(S) - 2N^{-1}\tr(S)\tr(S^2)\Big].
	\end{align*}
	Then using Lemma \ref{lem:trace_asymptotics}-(1)(2)(3), we have under the prescribed asymptotics that
	\begin{align*}
	\E\pnorm{T_{1,1}}{F}^2 & = N \bigg[py^2+3py+p+3y^2+3y+4N^{-1}y+N^{-2} \big(p^3+6py+ 8N^{-1}y\big)\\
	&\qquad - 2N^{-1}\big(p^2 y+ p^2 + py+ 4(y^2+y)+4N^{-1}y \big)\bigg]\\
	& = p^2 \bigg(1+\frac{N}{p}+\frac{1}{N}+\frac{3}{p}-\frac{4}{Np}-\frac{2}{N^2}\bigg)\\
	&= p^2\big[1+\mathcal{O}((N\wedge p)^{-1})\big]+pN.
	\end{align*}
	Hence we have
	\begin{align}\label{eq:nagao_gradient}
	\notag\E\pnorm{\nabla T(X)}{F}^4 &= \big(\E\pnorm{\nabla T(X)}{F}^2\big)^2 + \var\big(\pnorm{\nabla T(X)}{F}^2\big)\\
	&= \mathcal{O}(p^4(1+y^{-2})) + \var\big(\pnorm{\nabla T(X)}{F}^2\big).
	\end{align}
	By the Gaussian-Poincar\'e inequality, we have
	\begin{align*}
	\var\big(\pnorm{\nabla T(X)}{F}^2\big) &\leq \E\pnorm{\nabla \pnorm{\nabla T(X)}{F}^2}{F}^2 = 4 \E \pnorm{\big(\nabla^2 T(X)\big)^\top \nabla T(X)}{F}^2\\
	&\leq 4 \E^{1/2}\pnorm{\nabla^2 T(X)}{\op}^4\cdot \E^{1/2}\pnorm{\nabla T(X)}{F}^4.
	\end{align*}
	Combining the above display with (\ref{eq:nagao_gradient}) yields that
	\begin{align*}
	\E\pnorm{\nabla T(X)}{F}^4 \leq \mathcal{O}(p^4(1+y^{-2})) + 4\E^{1/2}\pnorm{\nabla^2 T(X)}{\op}^4\cdot \E^{1/2}\pnorm{\nabla T(X)}{F}^4.
	\end{align*}
	Solving the quadratic inequality above and using (\ref{ineq:nagao_clt_1}), we arrive at
	\begin{align*}
	\E\pnorm{\nabla T(X)}{F}^4 = \mathcal{O}\big(p^4(1+y^{-2}) \vee \E\pnorm{\nabla^2 T(X)}{\op}^4\big) = \mathcal{O}(p^4(1+y^{-2})).
	\end{align*}
	Combining the above two estimates, we have
	\begin{align*}
	\E\pnorm{\nabla T(X)}{F}^4 \lesssim p^4 \max_{y\geq 0}\min\big\{(1+y^6),(1+y^{-2})\}\asymp p^4.
	\end{align*}
    The rest of the proof proceeds along the lines in the proof of Proposition \ref{prop:clt_covariance_test}, with the help of the variance formula in Proposition \ref{prop:ratio_nagao}-(3).
\end{proof}

\subsubsection{Contiguity}\label{subsec:nagao_ratio}

\begin{proof}[Proof of Proposition \ref{prop:ratio_nagao}]
\noindent (1). Recall that $Z_1,\ldots,Z_n$ are i.i.d. samples from $\mathcal{N}(0,I_p)$. By Lemma \ref{lem:derivatives_T_nagao}, with $S_Z\equiv N^{-1}\sum_{i=1}^N Z_iZ_i^\top$,
\begin{align*}
\mathscr{T}_{\Sigma;\na}(Z)&= \bigg[Z\Sigma^{1/2}\big(\Sigma^{1/2}S_Z\Sigma^{1/2}-I\big)-N^{-1}\tr(\Sigma S_Z)Z\Sigma^{1/2}\bigg]\Sigma^{1/2} \\
&= Z\Sigma S_Z\Sigma - Z\Sigma-N^{-1}\tr(\Sigma S_Z)Z\Sigma,
\end{align*}
so
\begin{align*}
&\mathscr{T}_{\Sigma;\na}(Z)-\mathscr{T}_{I;\na}(Z)\\
&=\Big[Z\Sigma(S_Z\Sigma-I)-Z(S_Z-I)\Big]-\frac{1}{N}\Big[\tr(\Sigma S_Z)Z\Sigma-\tr(S_Z)Z\Big]\\
&=\Big[Z\Sigma(S_Z\Sigma-I)-Z(S_Z\Sigma-I)+Z(S_Z\Sigma-I)-Z(S_Z-I)\Big]\\
&\qquad -\frac{1}{N}\Big[\tr(\Sigma S_Z)Z\Sigma-\tr(\Sigma S_Z)Z+\tr(\Sigma S_Z)Z-\tr(S_Z)Z\Big]\\
& = Z(\Sigma-I)(S_Z\Sigma-I)+Z S_Z(\Sigma-I) -\frac{1}{N} \tr(\Sigma S_Z)Z(\Sigma-I)-\frac{1}{N}\tr\big((\Sigma-I)S_Z\big)Z\\
& \equiv V_1(Z)+V_2(Z)+V_3(Z)+V_4(Z).
\end{align*}
Note that
\begin{align*}
\pnorm{V_1(Z)}{F}^2&\leq \pnorm{S_Z\Sigma-I}{\op}^2\pnorm{Z(\Sigma-I)}{F}^2 \leq \pnorm{S_Z\Sigma-I}{\op}^2\pnorm{Z}{\op}^2\pnorm{\Sigma-I}{F}^2,\\
\pnorm{V_2(Z)}{F}^2&\leq \pnorm{ZS_Z}{\op}^2\pnorm{\Sigma-I}{F}^2\leq  \pnorm{Z}{\op}^2 \pnorm{S_Z}{\op}^2\pnorm{\Sigma-I}{F}^2,\\
\pnorm{V_3(Z)}{F}^2&\leq N^{-2} \tr^2(\Sigma S_Z)\pnorm{Z(\Sigma-I)}{F}^2\leq p^2N^{-2}\pnorm{\Sigma}{\op}^2 \pnorm{S_Z}{\op}^2 \pnorm{Z(\Sigma-I)}{F}^2,\\
\pnorm{V_4(Z)}{F}^2&\leq N^{-2} \tr^2\big((\Sigma-I)S_Z\big)\pnorm{Z}{F}^2\\
&\leq N^{-2}\pnorm{S_Z}{F}^2\pnorm{Z}{F}^2\pnorm{\Sigma-I}{F}^2\leq pN^{-2}\pnorm{S_Z}{\op}^2\pnorm{Z}{F}^2\pnorm{\Sigma-I}{F}^2.
\end{align*}
Under $p/N\leq M$, we have
\begin{align*}
V_{\Sigma;\na}^2&\lesssim_M N\big(\pnorm{\Sigma}{\op}^2 \vee 1\big)\pnorm{\Sigma-I}{F}^2.
\end{align*}
\noindent (2). By Lemma \ref{lem:trace_moment}, with $\delta_N\equiv N^{-1}-2N^{-2}$,
\begin{align*}
m_\Sigma&=\frac{N}{4}\bigg[\E \tr(S-I)^2-\frac{1}{N}\E\tr^2(S)\bigg] = \frac{N}{4}\bigg[\E \tr (S^2)-2\E \tr (S)+p-\frac{1}{N} \E \tr^2(S)\bigg]\\
& = \frac{N}{4}\bigg[\big(1+N^{-1}\big)\tr(\Sigma^2)+N^{-1}\tr^2(\Sigma) - 2\tr(\Sigma)+p- N^{-1}\tr^2(\Sigma)-2N^{-2}\tr(\Sigma^2)  \bigg]\\
& = \frac{N}{4}\big[(1+\delta_N)\tr(\Sigma^2)-2\tr(\Sigma)+p\big].
\end{align*}
Hence
\begin{align*}
m_{\Sigma}-m_I&=\frac{N}{4}\Big[(1+\delta_N)\tr(\Sigma^2-I)-2\tr(\Sigma-I)\Big]\\
& = \frac{N}{4}\big[\pnorm{\Sigma-I}{F}^2+\delta_N \tr(\Sigma^2-I)\big].
\end{align*}

\noindent (3). By the Plancherel's theorem (i.e., \cite[formula (6.2)]{chatterjee2014supreconcentration}), we have
\begin{align*}
\sigma_{I;\na}^2 &= \sum_{(ij)}\big[\E\partial_{(ij)}T(X)\big]^2 + \frac{1}{2!}\sum_{(i_1j_1)(i_2j_2)}\big[\E\partial_{(i_1j_1)(i_2j_2)}T(X)\big]^2\\
&+\frac{1}{3!}\sum_{(i_1j_1)(i_2j_2)(i_3j_3)}\big[\E\partial_{(i_1j_1)(i_2j_2)(i_3j_3)}T(X)\big]^2\\
&+ \frac{1}{4!}\sum_{(i_1j_1)(i_2j_2)(i_3j_3)(i_4j_4)}\big[\E\partial_{(i_1j_1)(i_2j_2)(i_3j_3)(i_4j_4)}T(X)\big]^2\\
&\equiv (I) + (II) + (III) + (IV).
\end{align*}
Terms $(I)$ - $(IV)$ are handled as follows:
\begin{itemize}
	\item To handle $(I)$, note that
	\begin{align*}
	\E\partial_{(ij)}T(X) = \E e_j^\top(S - I)X_i - \E\big[(\tr(S)/N)X_{ij}\big].
	\end{align*}
	The first term satisfies
	\begin{align*}
	&\E e_j^\top(S - I)X_i = \E e_j^\top \bigg(\frac{1}{N}\sum_{k=1}^N X_kX_k^\top\bigg)X_i = N^{-1}e_j^\top \E (X_i\cdot\pnorm{X_i}{}^2)\\
	&= N^{-1}e_j^\top \E \bigg(\frac{X_i}{\pnorm{X_i}{}}\cdot\pnorm{X_i}{}^3\bigg) = N^{-1}e_j^\top \E \bigg(\frac{X_i}{\pnorm{X_i}{}}\bigg)\cdot\E\pnorm{X_i}{}^3 = 0.
	\end{align*}
	A similar identity holds for the second term, so $(I) = 0$.
	\item $(II) \lesssim p/N = \mathfrak{o}(p^2)$ by noting that $\E\partial_{(i_1j_1)(i_2j_2)}T(X) = (N^{-1}-2N^{-2})\cdot \delta_{i_1i_2}\delta_{j_1j_2}$.
	\item $(III) = 0$ by direct calculation.
	\item $(IV) = 6p^2\big(1+\mathfrak{o}(1)\big)$ by direct calculation.
\end{itemize}
The proof is now complete by collecting all of the estimates.

\noindent (4). By (1)-(3), $\pnorm{\Sigma}{\op}\leq \pnorm{\Sigma-I}{F}+1$ and the condition $p/N\leq M$, we only need to show that 
\begin{align}\label{ineq:power_nagao_2}
\frac{ \sqrt{N} \pnorm{\Sigma-I}{F}^2\vee \sqrt{N \pnorm{\Sigma-I}{F}^2 } }{ \big(N \pnorm{\Sigma-I}{F}^2-N\delta_N |\tr(\Sigma^2-I)|\big)_+\vee \sigma_{I;\na} }\leq \frac{C_M}{(\sigma_{I;\na}\wedge N)^{1/2}}.
\end{align}
Note that with $\{\lambda_j\}_{j=1}^p$ denoting the eigenvalues of $\Sigma$,
\begin{align}\label{ineq:Q}
&\abs{\tr(\Sigma^2 - I)} = \biggabs{\sum_{j=1}^p (\lambda_j^2 - 1)}\leq \max_j (\lambda_j + 1)\cdot\sum_{j=1}^p |\lambda_j - 1| \nonumber \\
&\leq \sqrt{p}(\pnorm{\Sigma}{\op}+1)\pnorm{\Sigma- I}{F} \lesssim_M \sqrt{N}(\pnorm{\Sigma-I}{F}\vee1)\pnorm{\Sigma- I}{F},
\end{align}
so for $N$ large enough, (\ref{ineq:power_nagao_2}) is satisfied provided that
\begin{align}\label{ineq:power_nagao_3}
\frac{ \sqrt{N} \pnorm{\Sigma-I}{F}^2\vee \sqrt{N \pnorm{\Sigma-I}{F}^2 } }{ \big(N \pnorm{\Sigma-I}{F}^2-C_M'\sqrt{N}\pnorm{\Sigma-I}{F}\big)_+\vee \sigma_{I;\na} }\leq \frac{C_M}{(\sigma_{I;\na}\wedge N)^{1/2}}.
\end{align}
To see this, note that the left hand side of the above display is bounded, up to a constant that may depend on $M$, by
\begin{align*}
 &\bm{1}_{\sqrt{N}\pnorm{\Sigma-I}{F}\leq  2C_M' } \frac{1}{\sigma_{I;\na}}+ \bm{1}_{\sqrt{N}\pnorm{\Sigma-I}{F}>  2C_M' }\frac{\sqrt{N} \pnorm{\Sigma-I}{F}^2 \vee \sqrt{N \pnorm{\Sigma-I}{F}^2} }{N \pnorm{\Sigma-I}{F}^2\vee \sigma_{I;\na} }\\
 &\lesssim \frac{1}{\sigma_{I;\na}}+ \frac{\sqrt{N} \pnorm{\Sigma-I}{F}^2 }{N \pnorm{\Sigma-I}{F}^2\vee \sigma_{I;\na} }+ \frac{ \sqrt{N \pnorm{\Sigma-I}{F}^2} }{N \pnorm{\Sigma-I}{F}^2\vee \sigma_{I;\na} }\\
 &\leq \frac{1}{\sigma_{I;\na}}+ \frac{1}{N^{1/2}}+ \frac{1}{\inf_{x\geq 0} \big(x \vee \frac{ \sigma_{I;\na}}{x}\big)} \leq \hbox{RHS of (\ref{ineq:power_nagao_3})}.
\end{align*}
This completes the proof.
\end{proof}

\subsubsection{Completing the proof for power expansion}\label{subsec:nagao_power}

\begin{proof}[Proof of Theorem \ref{thm:power_nagao}]
Abbreviate $\Psi_{\na}$ by $\Psi$. By Proposition \ref{prop:clt_covariance_test_nagao} and Proposition \ref{prop:ratio_nagao}, we have
	\begin{align*}
	&\biggabs{ \E_{\Sigma} \Psi(X)-   \Prob\bigg( \mathcal{N}\bigg(\frac{N \cdot \big(\pnorm{\Sigma-I}{F}^2+Q_{\na}(\Sigma)\big) }{4\sigma_{I;\na}},1\bigg)>z_\alpha\bigg)   } \leq C\cdot p^{-1/3}.
	\end{align*}
We only need to remove the residual term $Q_{\na}(\Sigma)$. To see this, note that by (\ref{ineq:Q}),
\begin{align*}
\abs{Q_{\na}(\Sigma)}\leq C_M N^{-1/2}(\pnorm{\Sigma - I}{F}\vee 1) \pnorm{\Sigma-I}{F}.
\end{align*}
So using Lemma \ref{lem:normal_mean_multi} we have
\begin{align*}\
\Delta P \leq \frac{C_{\alpha,M}(\pnorm{\Sigma-I}{F}\vee 1)}{N^{1/2}\pnorm{\Sigma-I}{F}},
\end{align*}
where
\begin{align*}
\Delta P& \equiv \Prob\bigg( \mathcal{N}\bigg(\frac{N \cdot \big(\pnorm{\Sigma-I}{F}^2+Q_{\na}(\Sigma)\big) }{4\sigma_{I;\na}},1\bigg)>z_\alpha\bigg)- \Prob\bigg( \mathcal{N}\bigg(\frac{N \cdot \pnorm{\Sigma-I}{F}^2 }{4\sigma_{I;\na}},1\bigg)>z_\alpha\bigg).
\end{align*}
On the other hand, by anti-concentration of normal random variable,
\begin{align*}
\Delta P\leq  C_M\frac{N^{1/2} (\pnorm{\Sigma - I}{F}\vee 1)\pnorm{\Sigma-I}{F}}{\sigma_{I;\na}}.
\end{align*}
Hence
\begin{align*}
\Delta P &\lesssim_{\alpha,M} \frac{(\pnorm{\Sigma-I}{F}\vee 1)}{N^{1/2}\pnorm{\Sigma-I}{F}} \wedge \frac{N^{1/2} (\pnorm{\Sigma - I}{F}\vee 1)\pnorm{\Sigma-I}{F}}{\sigma_{I;\na}}\\
&\leq \bm{1}_{\pnorm{\Sigma-I}{F}> 1} \frac{1}{N^{1/2}}+\bm{1}_{\pnorm{\Sigma-I}{F}\leq 1} \bigg[ \frac{1}{N^{1/2}\pnorm{\Sigma-I}{F}} \wedge \frac{N^{1/2} \pnorm{\Sigma-I}{F}}{\sigma_{I;\na}} \bigg]\\
&\leq \frac{1}{N^{1/2}}+ \frac{1}{\inf_{\alpha\geq 0}\big(x\vee \frac{\sigma_{I;\na}}{x}\big)}\asymp \frac{1}{(\sigma_{I;\na}\wedge N)^{1/2}}.
\end{align*}
Similarly we may get a lower bound for $\Delta P$. The proof is complete. 
\end{proof}

\section{Proofs for Section \ref{section:test_spherical} (testing sphericity)}\label{section:proof_sphericity}

\subsection{Proofs for Subsection \ref{section:LRT_spherical} (LRT)}

\subsubsection{Evaluation of derivatives}

\begin{lemma}\label{lem:derivatives_T_sphere}
	Recall the form of $T_{\lrt,s}(X)$ in (\ref{def:lrs_spherical}) and the definition of $b(S)$ in (\ref{def:trace_b}). Then for any $(i,j),(i^\prime,j^\prime)\in [N]\times [p]$,
	\begin{enumerate}
		\item $\partial_{(ij)}T_{\lrt,s}(X) = \big(X(I - S^{-1})\big)_{(ij)} +\big(1/b(S)-1\big)X_{ij} = e_j^\top\big[(I-S^{-1})X_i+\big(1/b(S)-1\big)X_i\big]$.
		\item $\partial_{(ij),(i'j')}T_{\lrt,s}(X) = N^{-1}X_i^\top S^{-1}(e_{j^\prime}X_{i^\prime}^\top + X_{i^\prime}e_{j^\prime}^\top)S^{-1}e_j + \delta_{ii^\prime}e_j^\top(I-  S^{-1})e_{j^\prime} + \big(1/b(S)-1\big)\delta_{ii'}\delta_{jj'}- (2/Np)X_{ij}X_{i'j'}/b^2(S)$.
	\end{enumerate}
\end{lemma}
\begin{proof}
	(1). We shorthand $T_{\lrt,s}(X)$ as $T$. By definition, (\ref{eq:first_order_1}) and (\ref{eq:partial_S}), we have
	\begin{align*}
	\partial_{(ij)} T(X) &= \frac{N}{2}\big(p\cdot\partial_{(ij)}\log\tr(S) - \partial_{ij} \log\det S\big)\\
	&= \frac{N}{2}\bigg(p\frac{\partial_{(ij)}\tr(S)}{\tr(S)} - \sum_{k,\ell=1}^p \frac{\partial \log\det S}{\partial S_{k\ell}}\frac{\partial S_{k\ell}}{\partial X_{ij}}\bigg)\\
	&= \frac{N}{2}\bigg[\frac{2p}{N}\frac{X_{ij}}{\tr(S)} - \sum_{k,\ell=1}^p (S^{-1})_{k\ell}\cdot \frac{1}{N}\big(\delta_{kj}X_{i\ell} + \delta_{\ell j}X_{ik} \big)\bigg]\\
	&= \frac{p}{\tr(S)} X_{ij} - \sum_{k=1}^p (S^{-1})_{kj} X_{ik}  =  \big(X (I-S^{-1})\big)_{ij}+\bigg(\frac{p}{\tr(S)}-1\bigg) X_{ij}.
	\end{align*} 
	(2). By the previous part, we have
	\begin{align*}
	\partial_{(ij),(i'j')} T(X)& = \partial_{(i'j')} \big(X (I-S^{-1})\big)_{ij}+\partial_{(i'j')} \bigg(\frac{p}{\tr(S)}-1\bigg) X_{ij}\equiv (I)+(II).
	\end{align*}
	The first term above is already calculated in Lemma \ref{lem:derivatives_T}-(2):
	\begin{align*}
	(I) = N^{-1}X_i^\top S^{-1}(e_{j^\prime}X_{i^\prime}^\top + X_{i^\prime}e_{j^\prime}^\top)S^{-1}e_j + \delta_{ii^\prime}e_j^\top(I-  S^{-1})e_{j^\prime}.
	\end{align*}
	So we only need to evaluate the second term:
	\begin{align*}
	(II)& = p\cdot \partial_{(i'j')}\tr^{-1}(S)\cdot X_{ij}+ \Big(\frac{p}{\tr(S)}-1\Big)\partial_{(i'j')}X_{ij}\\
	& = -p \cdot  \partial_{(i'j')}\tr (S)\cdot X_{ij}\cdot\tr^{-2}(S)+ \Big(\frac{p}{\tr(S)}-1\Big)\delta_{ii'}\delta_{jj'}\\
	& = - \frac{2p}{N} X_{ij}X_{i'j'}\cdot\tr^{-2}(S)+\Big(\frac{p}{\tr(S)}-1\Big)\delta_{ii'}\delta_{jj'}.
	\end{align*}
	The proof is complete. 
\end{proof}

\subsubsection{Normal approximation}\label{subsec:spherical_lrt_clt}

\begin{proof}[Proof of Theorem \ref{thm:clt_spherical}]
	We abbreviate $T_{\lrt,s}(X)$ as $T$. First we bound the norm for the gradient. Comparing Lemmas \ref{lem:derivatives_T}-(1) and \ref{lem:derivatives_T_sphere}-(1),  we only need to control
	\begin{align*}
	&\E \pnorm{ \big(b^{-1}(S)-1\big)X}{F}^4 = \E\big(N  \big(b^{-1}(S)-1\big)^2 \tr(S)\big)^2\\
	&\leq N^2p^2\cdot \E^{1/2} b^4(S)\cdot \E^{1/2} \big(b^{-1}(S)-1\big)^{8}\lesssim N^2 p^2\cdot \Big(\frac{p}{N}\Big)^2 =p^4.
	\end{align*}
	The inequality in the final line of the above display follows as 
	\begin{align}\label{ineq:spherical_lrt_clt_1}
	\E b^4(S)&\leq  \E \pnorm{S}{\op}^4 \lesssim 1,\\
	\E\big(b^{-1}(S)-1\big)^8 &= \E^{1/2} b^{-16}(S)\cdot \E^{1/2}\big(b(S)-1\big)^{16}\stackrel{(*)}{\lesssim}  (pN^{-1})^4.
	\end{align}
	Here $(*)$ follows from Lemma \ref{lem:trace_moment}-(3). Now by combining with (\ref{ineq:clt_covariance_test_1}) derived in the proof of Proposition \ref{prop:clt_covariance_test}, we see that $\E \pnorm{\nabla T(X)}{F}^4 \lesssim p^4$.

	Next we bound the spectral norm of the Hessian. Comparing Lemmas \ref{lem:derivatives_T}-(1) and \ref{lem:derivatives_T_sphere}-(1),  we only need to control the spectral norms of $T_4$ and $T_5$, where 
	\begin{align*}
	(T_4)_{(ij),(i'j')} &\equiv \big(b^{-1}(S)-1\big)\delta_{ii'}\delta_{jj'},\, (T_5)_{(ij),(i'j')} \equiv -\frac{2}{Np}X_{ij}X_{i'j'}\cdot b^{-2}(S).
	\end{align*}
	For $T_4$, clearly $\pnorm{T_4}{\op} = \abs{1/b(S)-1}$, so $\E \pnorm{T_4}{\op}^4 = \E \big(1/b(S)-1\big)^4 \lesssim (p/N)^2$ by (\ref{ineq:spherical_lrt_clt_1}). For $T_5$,  note that
	\begin{align*}
	\pnorm{T_5}{\op} &= \frac{2}{Np\cdot b^2(S)}\sup_{u,v \in B_{N\times p}(1)} \biggabs{\sum_{(ij),(i'j')}  u_{ij}X_{ij} X_{i'j'} v_{i'j'}} = \frac{2}{Np\cdot b^2(S)} \pnorm{X}{F}^2  = \frac{2}{b(S)}.
	\end{align*}
	So $\E \pnorm{T_5}{\op}^4 \lesssim  \E b^{-4}(S)= \mathcal{O}(1)$ by Lemma \ref{lem:trace_moment}-(3). By combining with (\ref{ineq:clt_covariance_test_4}) derived in the proof of Proposition \ref{prop:clt_covariance_test}, we see that $\E \pnorm{\nabla^2 T(X)}{\op}^4 = \mathcal{O}(1)$. The rest of the proof proceeds along the lines in the proof of Proposition \ref{prop:clt_covariance_test}, with the help of the variance formula in Proposition \ref{prop:ratio_spherical}-(3).
\end{proof}

\subsubsection{Contiguity}\label{subsec:spherical_lrt_ratio}

\begin{proof}[Proof of Proposition \ref{prop:ratio_spherical}]
	We will abbreviate $(T_{\lrt,s}, m_{\Sigma;\lrt,s}, \sigma_{\Sigma;\lrt,s},V_{\Sigma;\lrt,s})$ as $(T,m_{\Sigma;s},\sigma_{\Sigma;s},V_{\Sigma;s})$, and assume without loss of generality that $b(\Sigma)=\tr(\Sigma)/p=1$ (otherwise we may replace $\Sigma$ by $\Sigma\cdot b^{-1}(\Sigma)$). 
	
	\noindent (1). By Lemma \ref{lem:derivatives_T_sphere}, with $S_Z\equiv N^{-1}\sum_{i=1}^N Z_i Z_i^\top$, we have
	\begin{align*}
	\mathscr{T}_{\Sigma;s}& = \bigg[Z\Sigma^{1/2}(I-\Sigma^{-1/2}S_Z^{-1}\Sigma^{-1/2})+\bigg(\frac{1}{b(\Sigma^{1/2} S_Z\Sigma^{1/2})}-1\bigg)Z\Sigma^{1/2}\bigg]\Sigma^{1/2}\\
	& = Z(\Sigma-S_Z^{-1})+ \bigg(\frac{1}{b(\Sigma^{1/2} S_Z\Sigma^{1/2})}-1\bigg)Z\Sigma=  \frac{Z\Sigma}{b(\Sigma^{1/2} S_Z\Sigma^{1/2})}-ZS_Z^{-1}.
	\end{align*}
	Hence
	\begin{align*}
	V_{\Sigma;s}^2&= \E \pnorm{\mathscr{T}_{\Sigma;s}-\mathscr{T}_{I;s}}{F}^2= \E \biggpnorm{ \frac{Z\Sigma}{b(\Sigma^{1/2} S_Z\Sigma^{1/2})}- \frac{Z}{b( S_Z)} }{F}^2\\
	&\leq 2  \bigg\{ \E\bigg[\bigg(\frac{1}{b(\Sigma^{1/2} S_Z\Sigma^{1/2})}-\frac{1}{b(S_Z)}\bigg)^2 \pnorm{Z\Sigma}{F}^2 \bigg]+ \E \big[b^{-2}(S_Z)\pnorm{Z(\Sigma-I)}{F}^2\big] \bigg\}\\
	&\equiv 2 \E\big((I)+(II)\big).
	\end{align*}
	We bound $(I)$ and $(II)$ separately: 
	\begin{align*}
	(I)&= b^{-2}(\Sigma^{1/2} S_Z\Sigma^{1/2})b^{-2}(S_Z)b^2\big((\Sigma-I)S_Z\big)\pnorm{Z\Sigma }{F}^2\\
	&\leq b^{-2}(\Sigma^{1/2} S_Z\Sigma^{1/2})b^{-2}(S_Z)  \pnorm{S_Z}{\op}^2 \cdot \big(\pnorm{\Sigma}{F}^2/p\big) \pnorm{Z}{\op}^2 \pnorm{\Sigma-I}{F}^2;\\
	(II)& \leq b^{-2}(S_Z)\cdot \pnorm{Z}{\op}^2\pnorm{\Sigma-I}{F}^2.
	\end{align*}
	Using Lemmas \ref{lem:concentration_trace} and \ref{lem:kl}, we have
	\begin{align*}
	V_{\Sigma;s}^2\lesssim (p^{-1}\pnorm{\Sigma - I}{F}^2 + 1)N\pnorm{(\Sigma-I)}{F}^2.
	\end{align*}
	On the other hand, a trivial bound for $V_{\Sigma;s}^2$ is
	\begin{align*}
	V_{\Sigma;s}^2&= \E \biggpnorm{ \frac{Z\Sigma}{b(\Sigma^{1/2} S_Z\Sigma^{1/2})}- \frac{Z}{b( S_Z)} }{F}^2\\
	&\lesssim  \E b^{-2}(\Sigma^{1/2} S_Z\Sigma^{1/2}) \pnorm{Z\Sigma}{F}^2+ \E b^{-2} (S_Z) \pnorm{Z}{F}^2\lesssim N \big(\pnorm{\Sigma-I}{F}^2\vee p\big).
	\end{align*}
	Collecting the two bounds, we have
	\begin{align*}
	V_{\Sigma;s}^2 &\lesssim \big[\big(p^{-1}\pnorm{\Sigma-I}{F}^2+ 1\big)N\pnorm{(\Sigma-I)}{F}^2\big] \wedge N \big(\pnorm{\Sigma-I}{F}^2\vee p\big)\asymp N\pnorm{(\Sigma-I)}{F}^2.
	\end{align*}
	
	\noindent (2). As 
	\begin{align*}
	m_{\Sigma;s}& = \frac{N}{2}\Big[p\cdot \E \log \tr (\Sigma S_Z)- \log \det (\Sigma)-p\log p-\E \log \det (S_Z)\Big],
	\end{align*}
	by Lemma \ref{lem:res_log_trace} we have
	\begin{align*}
	m_{\Sigma;s}-m_{I;s}=\frac{N}{2}\big[-\log \det(\Sigma)+Q_s(\Sigma)\big],
	\end{align*}
	where
	\begin{align*}
	\abs{Q_s(\Sigma)}&\equiv \abs{p\big(\E \log \tr (\Sigma S_Z) - \E\log\tr(S_Z)\big)}\\
	& \lesssim N^{-1}\Big\{1+ b(\Sigma^2)  + e^{-cN}\big[ 1+ b^{1/2}( \Sigma^2 ) \big] \Big\}\lesssim N^{-1}\Big[1+ b(\Sigma^2)  \Big] \lesssim N^{-1} b(\Sigma^2),
	\end{align*}
	where the last inequality follows as $b(\Sigma^2) = p^{-1}\sum_{j=1}^p \lambda_j^2 \geq p^{-2}(\sum_{j=1}^p \lambda_j)^2 =1$.
	
	\noindent (3). Recall $T_{\lrt}$ defined in (\ref{def:lrs}). Define
	\begin{align*}
	\Delta(X) &\equiv T_{\lrt}(X) - T_{\lrt,s}(X).
	\end{align*}
	Then for any $\epsilon > 0$, there exists some $C_\epsilon>0$ such that under the null (i.e., $X_1,\ldots,X_n$ are i.i.d. $\mathcal{N}(0,I_p)$),
	\begin{align}\label{ineq:variance_compare}
	\big[(1-\epsilon)\sigma^2_{I;\lrt} -C_\epsilon \var_I(\Delta)\big]_+ 
	 \leq \sigma^2_{I;\lrt,s}\leq (1+\epsilon)\sigma^2_{I;\lrt} + C_\epsilon \var_I(\Delta).
	\end{align}
 	We will now bound $\var_I(\Delta)$. By Lemmas \ref{lem:derivatives_T}-(1) and \ref{lem:derivatives_T_sphere}-(1), we have for any $i,j\in[N]\times [p]$
	\begin{align*}
	\partial_{(ij)}\Delta(X) =  \partial_{(ij)}T_{\lrt}(X) - \partial_{(ij)}T_{\lrt,s}(X) = X_{ij}\big[b^{-1}(S) - 1\big].
	\end{align*}
	By the Gaussian-Poincar\'e inequality \cite[Theorem 3.20]{boucheron2013concentration},
	\begin{align*}
	\var_I \Delta(X) &\leq \E\big[b^{-1}(S) - 1\big]^2\pnorm{X}{F}^2 = Np\E\big[b(S) - 1\big]^2b^{-1}(S)\\
	&\leq Np\cdot\E^{1/2}\big(b(S) - 1\big)^4\cdot \E^{1/2}b^{-2}(S) \stackrel{(*)}{\lesssim} Np\cdot (Np)^{-1} = 1.
	\end{align*}
	Here $(*)$ follows from Lemma \ref{lem:trace_moment}-(3). Hence by choosing $\epsilon$ in (\ref{ineq:variance_compare}) to be decaying to $0$ slowly enough, $\sigma^2_{I;\lrt}$ and $\sigma^2_{I;\lrt,s}$ share the same asymptotic formula in Proposition \ref{prop:ratio_covariance_test}-(3).

	\noindent (4). By (1)-(2), and using that $b(\Sigma^2)=\pnorm{\Sigma}{F}^2/p$,	we only need to prove that for a given constant $C_0>0$, there exists some constant $C=C(C_0)>0$ such that
	\begin{align*}
	&\frac{ \sqrt{N \pnorm{\Sigma-I}{F}^2} }{ \Big(-N \log \det (\Sigma)-C_0(1+\frac{\pnorm{\Sigma}{F}^2}{p})-C_0 e^{-cN}\big( \frac{\pnorm{\Sigma}{F}}{p^{1/2}}+ 1\big) \Big)_+\vee \sigma_{I;s} }\leq \frac{C}{\big(\sigma_{I;s}\wedge N\big)^{1/2}}.
	\end{align*}
	Equivalently, with $\lambda = (\lambda_1,\ldots,\lambda_p) \in (0,\infty)^p$ and $\bar{\lambda}\equiv p^{-1}\sum_j \lambda_j =1$, we only need to show
	\begin{align*}
	&\frac{ \sqrt{N \sum_j (\lambda_j-1)^2 } }{ \Big(N \sum_j -\log(1+(\lambda_j-1))-C_0-C_0 \big( \frac{\sum_j \lambda_j^2}{p}\big)  -C_0 e^{-cN} \frac{(\sum_j \lambda_j^2)^{1/2}}{p^{1/2}}  \Big)_+\vee \sigma_{I;s} }
	\end{align*}
	is at most a multiple of $\big(\sigma_{I;s}\wedge N\big)^{-1/2}$. 
	Let $J\equiv \{j: \abs{\lambda_j-1}\leq 1\}$ and $J^c\equiv \{j: \abs{\lambda_j-1}> 1\}$. As $\abs{\lambda_j-1}\lesssim p$, so the first term in the denominator becomes
	\begin{align*}
	&N \sum_j \big[-\log(1+(\lambda_j-1))+(\lambda_j-1)\big]-C_0-C_0 \frac{\sum_j \lambda_j^2}{p}   -C_0 e^{-cN} \frac{(\sum_j \lambda_j^2)^{1/2}}{p^{1/2}} \\
	&\gtrsim N \sum_j (\lambda_j-1)^2\wedge \abs{\lambda_j-1}-C_1 p^{-1} \sum_j (\lambda_j-1)^2-C_2.
	\end{align*}
	Next, by breaking the summation in $\sum_j (\lambda_j-1)^2$ into $J$ and $J^c$, the above display equals
	\begin{align*}
	&N \sum_{j \in J} (\lambda_j-1)^2+N\sum_{j \in J} \abs{\lambda_j-1} - C_1 \frac{\sum_{j \in J} (\lambda_j-1)^2 + \sum_{j \in J^c} (\lambda_j-1)^2}{p}-C_2\\
	&\geq (N-C_1p^{-1}) \sum_{j \in J} (\lambda_j-1)^2 + (N-\mathcal{O}(1)) \sum_{j \in J^c} \abs{\lambda_j-1}-C_2\\
	&\geq \frac{N}{2}\sum_j (\lambda_j-1)^2\wedge \abs{\lambda_j-1}-C_2
	\end{align*}
	for $N$ and $p$ large enough. Hence with $\nu_j\equiv \abs{\lambda_j-1}$, we only need to show that for given $C_0>0$, 
	\begin{align*}
	\frac{ \sqrt{N \sum_{j \in J} \nu_j^2 }\vee \sqrt{N \sum_{j \in J^c} \nu_j^2 }}{ \Big(N\sum_{j \in J} \nu_j^2+ N \sum_{j \in J^c} \nu_j-C_0 \Big)_+\vee \sigma_{I;s} } \leq \frac{C}{ \big(\sigma_{I;s}\wedge N\big)^{1/2}}. 
	\end{align*}
	Equivalently, we only need to show 
	\begin{align}
	&\frac{ \sqrt{N \sum_{j \in J} \nu_j^2 } }{ \big(N\sum_{j \in J} \nu_j^2-C_0 \big)_+\vee \sigma_{I;s} } \leq \frac{C}{\sigma_{I;s}^{1/2}},\label{ineq:power_reg_cond_spherical_1}\\
	& \frac{  \sqrt{N \sum_{j \in J^c} \nu_j^2 }}{ \big(N \sum_{j \in J^c} \nu_j-C_0 \big)_+\vee \sigma_{I;s} } \leq \frac{C}{N^{1/2}}.\label{ineq:power_reg_cond_spherical_2}
	\end{align}
	To see these inequalities, note that the left side of (\ref{ineq:power_reg_cond_spherical_1}) is bounded by
	\begin{align*}
	&\bm{1}_{ N\sum_{j \in J}\nu_j^2 \leq 2C_0 } \frac{ (2C_0)^{1/2}}{\sigma_{I;s}}+\bm{1}_{ N\sum_{j \in J}\nu_j^2 > 2C_0 } \frac{ \sqrt{N \sum_{j \in J} \nu_j^2 } }{ (N/2)\sum_{j \in J} \nu_j^2\vee \sigma_{I;s} }\\
	&\lesssim \frac{1}{\sigma_{I;s}}+ \frac{1}{\inf_{x\geq 0}\big(x \vee \frac{\sigma_{I;s}}{x}\big)} \lesssim \sigma_{I;s}^{-1/2}.
	\end{align*}
	Also, the left side of (\ref{ineq:power_reg_cond_spherical_2}) is bounded by
	\begin{align*}
	\frac{  \sqrt{N  } \sum_{j \in J^c} \nu_j}{ \big(N \sum_{j \in J^c} \nu_j-C_0 \big)_+\vee \sigma_{I;s} } &\leq \bm{1}_{ N \sum_{j \in J^c} \nu_j\leq 2C_0 } \frac{(2C_0)^{1/2}}{\sqrt{N}\sigma_{I;s} }+ \bm{1}_{N \sum_{j \in J^c} \nu_j>2C_0}\frac{  \sqrt{N  } \sum_{j \in J^c} \nu_j}{ N \sum_{j \in J^c} \nu_j\vee \sigma_{I;s} }\\
	&\lesssim \frac{1}{\sqrt{N}\sigma_{I;s}}+\frac{1}{\sqrt{N}}\lesssim \frac{1}{N^{1/2}},
	\end{align*}
	proving the claim.
\end{proof}

\subsubsection{Completing the proof for power expansion}\label{subsec:spherical_clt_power}

\begin{proof}[Proof of Theorem \ref{thm:power_spherical}]
The proof is similar to that of Theorem \ref{thm:power_nagao}, we provide some details for the convenience of the reader. Without loss of generality we assume $b(\Sigma)=1$. Abbreviate $\Psi_{\lrt,s}$ by $\Psi$ and $Q_{\lrt,s}(\Sigma)$ by $Q(\Sigma)$. By Theorem \ref{thm:clt_spherical} and Proposition \ref{prop:ratio_spherical}, we have
	\begin{align*}
	&\biggabs{ \E_{\Sigma} \Psi(X)-   \Prob\bigg( \mathcal{N}\bigg(\frac{N \cdot \big(-\log\det\big(\Sigma\big)+Q(\Sigma)\big) }{2\sigma_{I;s}},1\bigg)>z_\alpha\bigg)   } \leq C\cdot p^{-1/3}.
	\end{align*}
We only need to remove the residual term $Q(\Sigma)$. To this end, we claim that
\begin{align}\label{ineq:Delta_P_1}
|\Delta P| \leq C_\alpha\Big[\frac{NQ(\Sigma)}{\sigma_{I;s}}\wedge \frac{Q(\Sigma)}{|\log\det\big(\Sigma\big)|}\Big].
\end{align}
where
\begin{align*}
\Delta P &\equiv \Prob\bigg( \mathcal{N}\bigg(\frac{N \cdot \big(-\log\det\big(\Sigma\big)+Q(\Sigma)\big) }{2\sigma_{I;s}},1\bigg)>z_\alpha\bigg)- \Prob\bigg( \mathcal{N}\bigg(\frac{-N\log\det\big(\Sigma\big) }{2\sigma_{I;s}},1\bigg)>z_\alpha\bigg).
\end{align*}
Here the first bound in (\ref{ineq:Delta_P_1}) is by anti-concentration of the normal distribution, and the second bound in (\ref{ineq:Delta_P_1}) follows from Lemma \ref{lem:normal_mean_multi}.

Let $\{\lambda_j\}_{j=1}^p$ be the eigenvalues of $\Sigma$ so that $\sum_{j=1}^p \lambda_j = p$. Then by (\ref{ineq:Q_lrt}), $
Q(\Sigma) \lesssim 2(Np)^{-1}\sum_{j=1}^p \lambda_j^2$. Hence using the bound $\sigma_{I;s}\geq cp$, (\ref{ineq:Delta_P_1}) entails that
\begin{align}\label{ineq:Delta_P_2}
|\Delta P| \leq C'_\alpha\cdot\Big[\frac{p^{-1}\sum_{j=1}^p \lambda_j^2}{p}\wedge \frac{(Np)^{-1}\sum_{j=1}^p \lambda_j^2}{\sum_{j=1}^p \lambda_j - \log \lambda_j - 1}\Big].
\end{align}
If $\max_j \lambda_j \leq 10$, we use the first bound in (\ref{ineq:Delta_P_2}) to conclude that $\Delta P\lesssim_\alpha p^{-1}$. Otherwise, by writing $J\equiv \{j\in[p]: | \lambda_j - 1|\geq 1\}$ and $J^c\equiv [p]\backslash J$, the second bound in (\ref{ineq:Delta_P_2}) yields that
\begin{align*}
|\Delta P| &\lesssim_\alpha \frac{(Np)^{-1}\sum_{j=1}^p \lambda_j^2}{\sum_{j=1}^p | \lambda_j - 1|\wedge (\lambda_j - 1)^2} \lesssim \frac{(Np)^{-1}\sum_{j\in J} (\lambda_j-1)^2+(Np)^{-1}(|J|+|J^c|)}{\sum_{j\in J} |\lambda_j - 1|}\\
&\leq \frac{(Np)^{-1}\sum_{j\in J} (\lambda_j-1)^2}{\sum_{j\in J} |\lambda_j - 1|}+ \frac{N^{-1}}{\sum_{j\in J} |\lambda_j - 1|} \equiv (I)+ (II).
\end{align*}
Now $(II)\lesssim N^{-1}$ as $\max_j {\lambda}_j > 10$, and $(I)$ satisfies
\begin{align*}
(I)\leq (Np)^{-1}\max_{j\in J}| {\lambda}_j - 1| \lesssim N^{-1}
\end{align*}
by using the trivial bound that $\max_j {\lambda}_j \leq p$ due to the normalization $b(\Sigma)=1$. The proof is complete. 
\end{proof}

\subsection{Proofs for Section \ref{section:john} (John's test)}\label{section:proof_john}

\subsubsection{Evaluation of derivatives}

\begin{lemma}\label{lem:derivatives_john}
	Recall the form of $T_J(X)$ in (\ref{def:lrs_john}) and the definition of $b_\ell(S)$ in (\ref{def:trace_b}). Then the following hold:
	\begin{enumerate}
		\item For the first-order partial derivatives: for any $(i,j)\in [N]\times [p]$,
		\begin{align*}
		\partial_{(ij)} T_J(X)&= \Big(\frac{XS}{ b^2(S)}-X\cdot\frac{b_2(S)}{b^3(S)}\Big)_{ij} = \frac{X_i^\top Se_j}{b^2(S)} - X_{ij}\frac{b_2(S)}{b^3(S)}.
		\end{align*}
		\item For the second-order partial derivatives: for any $(i,j),(i',j')\in [N]\times [p]$,
		\begin{align*}
		&\partial_{(ij),(i'j')}T_J(X)\\
		& = b(S)^{-2}\big(N^{-1}\delta_{jj'} X_i^\top X_{i'}+N^{-1}X_{i'j}X_{ij'}+\delta_{ii'}S_{jj'}\big)-\delta_{ii'}\delta_{jj'} \frac{b_2(S)}{b^3(S)} \\
& \qquad +X_{ij}X_{i'j'} \frac{6b_2(S)}{b^4(S)Np}-\frac{4}{b^3(S)Np}\Big[ X_i^\top S e_j\cdot X_{i'j'}+X_{i'}^\top S e_{j'} \cdot X_{ij}\Big].
		\end{align*}
	\end{enumerate}
\end{lemma}
\begin{proof}
We abbreviate $T_J(X)$ by $T(X)$ and write $b=b(S)$ in the proof if no confusion could arise. 

\noindent (1). Note that $\partial_{ij}S(X) = N^{-1}(e_jX_i^\top+X_i e_j^\top)$, $\partial_{(ij)}\tr(S)=2N^{-1}X_{ij}$ and
\begin{align}\label{ineq:derivative_john_1}
\partial_{(ij)}b(S)=\frac{2}{Np}X_{ij}, \, \partial_{(ij)}b_2(S)& = \frac{2}{Np}\tr\big(S (e_jX_i^\top+X_ie_j^\top)\big) = \frac{4}{Np} X_i^\top S e_j.
\end{align}
For the first-order derivatives we have
\begin{align*}
\partial_{(ij)} T(X)& = \frac{N}{4} \tr\bigg[2\bigg(\frac{S}{b}-I\bigg) \partial_{(ij)} \bigg(\frac{S}{b}\bigg)\bigg]\\
& = \frac{N}{2} \tr \bigg[\bigg(\frac{S}{b}-I\bigg)\cdot \frac{N^{-1}(e_jX_i^\top+X_ie_j^\top) b-2(Np)^{-1}S X_{ij}}{b^2} \bigg]\\
& = \frac{1}{2 b^2} \tr\big[(S-bI)(e_jX_i^\top+X_ie_j^\top)\big] - \frac{X_{ij} }{b^3 p} \tr\big[(S-bI)S\big]\\
& = \frac{(XS)_{ij}}{ b^2}-\frac{X_{ij}}{b}-\bigg[\frac{X_{ij}b_2}{b^3}-\frac{X_{ij}}{b}\bigg]\\
& = \frac{(XS)_{ij}}{ b^2}-X_{ij}\cdot\frac{b_2}{b^3}  \equiv T_{1,(ij)}(X)-T_{2,(ij)}(X).
\end{align*}
\noindent (2). For the second-order derivatives, 
\begin{align*}
\partial_{(i'j')} T_{1,(ij)}(X)& = \frac{ \partial_{(i'j')} (X_i^\top Se_j) b^2- (X_i^\top Se_j) \partial_{(i'j')} b^2 }{ b^4 }\\
& = \frac{ \delta_{ii'} e_{j'}^\top S e_j+N^{-1}X_i^\top (e_{j'}X_{i'}^\top+X_{i'}e_{j'}^\top) e_j}{b^2}- \frac{4}{pN}\frac{ X_i^\top S e_j\cdot X_{i'j'}  }{b^3}\\
& = \frac{N^{-1}\delta_{jj'} X_i^\top X_{i'}+N^{-1}X_{i'j}X_{ij'}+\delta_{ii'}S_{jj'} }{b^2}- \frac{4}{pN}\frac{ X_i^\top S e_j\cdot X_{i'j'}  }{b^3}, \\
\partial_{(i'j')}T_{2,(ij)}(X)& = \delta_{ii'}\delta_{jj'} \frac{b_2}{b^3} + X_{ij} \cdot \partial_{(i'j')}\bigg[\frac{b_2}{b^3}\bigg]\\
& =  \delta_{ii'}\delta_{jj'} \frac{b_2}{b^3} + X_{ij}\cdot \bigg[ \frac{4X_{i'}^\top Se_{j'} }{b^3Np}- \frac{6b_2 X_{i'j'}}{b^4Np}\bigg]\\
& = \delta_{ii'}\delta_{jj'} \frac{b_2}{b^3} - X_{ij}X_{i'j'}\frac{6b_2}{b^4Np}+4X_{i'}^\top S e_{j'} \cdot X_{ij}\frac{1}{b^3 Np}.
\end{align*}
Combining the above two displays, we have
\begin{align*}
\partial_{(ij),(i'j')}T(X) & = b^{-2}\big(N^{-1}\delta_{jj'} X_i^\top X_{i'}+N^{-1}X_{i'j}X_{ij'}+\delta_{ii'}S_{jj'}\big)-\delta_{ii'}\delta_{jj'} \frac{b_2}{b^3} \\
& \qquad +X_{ij}X_{i'j'} \frac{6b_2}{b^4Np}-\frac{4}{b^3Np}\Big[ X_i^\top S e_j\cdot X_{i'j'}+X_{i'}^\top S e_{j'} \cdot X_{ij}\Big].
\end{align*}
The proof is complete.
\end{proof}

\subsubsection{Normal approximation}\label{subsec:john_clt}

\begin{proof}[Proof of Proposition \ref{prop:clt_covariance_test_john}]
We abbreviate $T_J$ by $T$ and write $b=b(S)$ in the proof if no confusion could arise.  First we bound the operator norm of the Hessian.  By Lemma \ref{lem:derivatives_john}-(2),
\begin{align*}
&\partial_{(ij),(i'j')}T(X)= b^{-2}\big(N^{-1}\delta_{jj'} X_i^\top X_{i'}+N^{-1}X_{i'j}X_{ij'}+\delta_{ii'}S_{jj'}\big)\\
&-\delta_{ii'}\delta_{jj'} \frac{b_2}{b^3} +X_{ij}X_{i'j'} \frac{6b_2}{b^4Np}-\frac{4}{b^3Np}\Big[ X_i^\top S e_j\cdot X_{i'j'}+X_{i'}^\top S e_{j'} \cdot X_{ij}\Big]\\
&\equiv (T_1 - T_2 + T_3 - T_4)_{(ij),(i'j')}.
\end{align*}
Following the proof of Proposition \ref{prop:clt_covariance_test_nagao} along with Lemma \ref{lem:concentration_trace}, we have $\E\pnorm{T_1}{\op}^4\lesssim (1\vee y)^4$. Next for $T_2$, we have by Lemma \ref{lem:concentration_trace} and Lemma \ref{lem:kl} that
\begin{align*}
\E\pnorm{T_2}{\op}^4 \lesssim \E (b_2^4\cdot b^{-12})\leq \E^{1/2} b_2^8 \cdot \E^{1/2} b^{-24}\lesssim \E^{1/2}\pnorm{S}{\op}^8\lesssim (1\vee y)^4.
\end{align*}
The operator norm of $T_3$ can be similarly bounded by
\begin{align*}
\E\pnorm{T_3}{\op}^4 &= \frac{6^4}{(Np)^4}\E\Big[\Big(\frac{b_2}{b^4}\Big)^4\pnorm{X}{F}^8\Big] \lesssim (Np)^{-4}\E^{1/2}b_2^8\cdot \E^{1/4}b^{-64}\E^{1/4}\pnorm{X}{F}^{32}\\
&\lesssim (Np)^{-4}\cdot \E^{1/2}b_2^8\cdot (Np)^4 \lesssim (1\vee y)^4.
\end{align*}
Lastly,
\begin{align*}
\pnorm{T_4}{\op} &\lesssim \frac{1}{b^3Np}\cdot \sup_{u,v\in B_{N\times p}(1)} \biggabs{\sum_{(ij),(i'j')} X_i^\top S e_j\cdot X_{i'j'}u_{ij}v_{i'j'}}\\
&= \frac{1}{b^3Np}\cdot \sup_{u,v\in B_{N\times p}(1)} \biggabs{\Big(\sum_{i,j} X_i^\top Se_j u_{ij}\Big)\Big(\sum_{i'j'}X_{i'j'}v_{i'j'}\Big)}\\
&\leq \frac{1}{b^3Np}\cdot \pnorm{XS}{F}\cdot\pnorm{X}{F}\leq \frac{1}{b^3Np}\cdot \pnorm{S}{\op}\pnorm{X}{F}^2.
\end{align*} 
Hence by Lemma \ref{lem:kl} and Lemma \ref{lem:concentration_trace}, $\E\pnorm{T_4(X)}{\op}^4 \lesssim (1\vee y)^4$. Putting together the bounds for $T_1$ - $T_4$ yields that $\E\pnorm{\nabla^2T(X)}{\op}^4\lesssim (1\vee y)^4$.

Next we bound the norm of the gradient. We will show that $\E\pnorm{\nabla T(X)}{F}^2 \lesssim p^2$ by considering the two cases $p/N\leq 1$ and $p/N> 1$ separately.

\noindent (\textbf{Case $p/N\leq 1$}) By Lemma \ref{lem:derivatives_john}-(1), we may write 
\begin{align*}
\nabla T(X) = b^{-1}X\big(b^{-1}S-I\big)-b^{-1}X\cdot b\big(b^{-1}S-I\big)^2,
\end{align*}
so
\begin{align*}
\pnorm{\nabla T(X)}{F}^4 &\lesssim b^{-8} \pnorm{X}{F}^4 \pnorm{S-bI}{\op}^4+ b^{-12} \pnorm{X}{F}^4 \pnorm{S-bI}{\op}^8\\
&\lesssim \pnorm{X}{F}^4 \big(b^{-8}\pnorm{S-I}{\op}^4+b^{-8}\abs{b-1}^4+b^{-12}\pnorm{S-I}{\op}^8+b^{-12}\abs{b-1}^8\big).
\end{align*}
By Lemma \ref{lem:kl} and Lemma \ref{lem:trace_moment}, it holds under the condition $ p/N\leq 1$ that
\begin{align*}
\E\pnorm{T(X)}{F}^4 \lesssim (Np)^2\big((N^{-1}p)^2+(N^{-1}p)^4\big)\lesssim p^4.
\end{align*}

\noindent (\textbf{Case $p/N> 1$}) By Lemma \ref{lem:derivatives_john}-(1), we have
\begin{align*}
\E\pnorm{\nabla T(X)}{F}^2 &= \E \bigg[\biggpnorm{\frac{XS}{b^2}}{F}^2 + \biggpnorm{\frac{Xb_2}{b^3}}{F}^2 - 2\iprod{\frac{XS}{b^2}}{\frac{Xb_2}{b^3}}\bigg]= Np\cdot \E\bigg[\frac{bb_3 - b_2^2}{b^5} \bigg]\\
&= Np\cdot \E (bb_3 - b_2^2) + Np\cdot\E (bb_3 - b_2^2)(b^{-5} - 1) \equiv (I) + (II).
\end{align*}
To handle $(I)$, it holds by Lemma \ref{lem:trace_asymptotics}-(4)(5) that under $p > N$,
\begin{align*}
(I) = \frac{N}{p}\E\big[\tr(S)\tr(S^3) - \tr^2(S^2)\big] = \frac{N}{p}N^{-4}\mathcal{O}(N^3p^3) = \mathcal{O}(p^2).
\end{align*}
To handle $(II)$, it holds by Lemmas \ref{lem:concentration_trace}, \ref{lem:trace_moment}-(3), and \ref{lem:trace_asymptotics}-(6) that under $p > N$,
\begin{align*}
(II) &= Np\cdot \E (bb_3 - b_2^2)b^{-5}(1 - b^5) \leq Np\cdot \E^{1/2} (bb_3 - b_2^2)^2\E^{1/4}b^{-20}\E^{1/4} (b^5 - 1)^4\\
&\leq Np\cdot \Big(|\E (bb_3 - b_2^2)| + \var^{1/2}(bb_3 - b_2^2)\Big)\cdot \E^{1/4}b^{-20}\E^{1/4} (b^5 - 1)^4\\
&= Np \cdot \mathcal{O}(N^{-1}p)\cdot \mathcal{O}(1) \cdot \mathcal{O}((Np)^{-1/2}) = \mathfrak{o}(p^2).
\end{align*}
Putting together the estimates for $(I)$ and $(II)$ yield that $\E\pnorm{\nabla T(X)}{F}^2 = \mathcal{O}(p^2)$ under the considered case $p > N$. The rest of the proof proceeds along the lines in the proof of Proposition \ref{prop:clt_covariance_test_nagao}, with the help of the variance formula in Proposition \ref{prop:ratio_john}-(3). The normal approximation error bound becomes a constant multiple of
\begin{align*}
\frac{(1\vee y)\cdot p}{p^2} = \frac{\frac{p}{n}\vee 1}{p} = \frac{1}{n\wedge p},
\end{align*}
as desired.
\end{proof}

\subsubsection{Contiguity}\label{subsec:john_ratio}

\begin{proof}[Proof of Proposition \ref{prop:ratio_john}]
We assume without loss of generality that $b(\Sigma)=\tr(\Sigma)/p=1$ (otherwise we replace $\Sigma$ by $\Sigma\cdot b^{-1}(\Sigma)$).	
	
\noindent (1). By Lemma \ref{lem:derivatives_john}, with $S_Z\equiv N^{-1}\sum_{i=1}^N Z_iZ_i^\top$, we have
\begin{align*}
\mathscr{T}_{\Sigma;\jo}&= \bigg\{ \frac{Z\Sigma^{1/2}\Sigma^{1/2}S_Z\Sigma^{1/2}}{b^2(\Sigma^{1/2} S_Z\Sigma^{1/2})} -Z\Sigma^{1/2}\frac{b_2(\Sigma^{1/2}S_Z\Sigma^{1/2})}{ b^3(\Sigma^{1/2} S_Z\Sigma^{1/2}) }\bigg\}\Sigma^{1/2}\\
& = \frac{Z\Sigma S_Z\Sigma}{b^2(\Sigma^{1/2} S_Z\Sigma^{1/2})}- Z\Sigma\frac{b_2(\Sigma^{1/2}S_Z\Sigma^{1/2})}{ b^3(\Sigma^{1/2} S_Z\Sigma^{1/2}) },
\end{align*}
so
\begin{align*}
&\mathscr{T}_{\Sigma;\jo}-\mathscr{T}_{I;\jo}\\
& =  \bigg\{ \frac{Z\Sigma S_Z\Sigma}{b^2(\Sigma^{1/2} S_Z\Sigma^{1/2})}- \frac{Z S_Z}{b^2(S_Z)}\bigg\}-\bigg\{Z\Sigma\frac{b_2(\Sigma^{1/2}S_Z\Sigma^{1/2})}{ b^3(\Sigma^{1/2} S_Z\Sigma^{1/2}) }-Z\frac{b_2(S_Z)}{ b^3(S_Z) }\bigg\}\\
& = \bigg\{ \frac{Z\Sigma S_Z\Sigma}{b^2(\Sigma^{1/2} S_Z\Sigma^{1/2})}- \frac{Z S_Z}{b^2(S_Z)}\bigg\} - Z\bigg\{\frac{b_2(\Sigma^{1/2}S_Z\Sigma^{1/2})}{ b^3(\Sigma^{1/2} S_Z\Sigma^{1/2})} - \frac{b_2(S_Z)}{ b^3(S_Z) }\bigg\} \\
&\qquad - (Z\Sigma-Z)\cdot\frac{b_2(\Sigma^{1/2}S_Z\Sigma^{1/2})}{ b^3(\Sigma^{1/2} S_Z\Sigma^{1/2})} \equiv V_1(Z)+V_2(Z)+V_3(Z).
\end{align*}
We will handle the Frobenius norms of $V_1(Z),V_2(Z),V_3(Z)$ separately below. For $V_1(Z)$,
\begin{align*}
\pnorm{V_1(Z)}{F}^2 &\lesssim \biggpnorm{\frac{Z\Sigma S_Z\Sigma}{b^2(\Sigma^{1/2} S_Z\Sigma^{1/2})}-\frac{ZS_Z}{b^2(\Sigma^{1/2} S_Z\Sigma^{1/2})}}{F}^2+ \biggpnorm{\frac{Z S_Z}{b^2(\Sigma^{1/2} S_Z\Sigma^{1/2})}- \frac{Z S_Z}{b^2(S_Z)}}{F}^2\\
&= \pnorm{Z\Sigma S_Z\Sigma-ZS_Z}{F}^2\cdot \frac{1}{b^4(\Sigma^{1/2} S_Z\Sigma^{1/2})}\\
&\qquad + \pnorm{  Z S_Z }{F}^2\cdot \bigg[\frac{b^2(\Sigma^{1/2} S_Z\Sigma^{1/2})-b^2(S_Z)}{b^2(\Sigma^{1/2} S_Z\Sigma^{1/2})b^2(S_Z)}\bigg]^2 \equiv V_{1,1}+V_{1,2}.
\end{align*}
Note that
\begin{align*}
V_{1,1}&\lesssim b^{-4}(\Sigma^{1/2} S_Z\Sigma^{1/2})\big(  \pnorm{Z\Sigma S_Z(\Sigma-I)}{F}^2 + \pnorm{Z(\Sigma-I) S_Z}{F}^2 \big)\\
&\lesssim \Big[ b^{-4}(\Sigma^{1/2} S_Z\Sigma^{1/2}) \cdot \pnorm{S_Z}{\op}^2 \Big]\cdot  \big(\pnorm{\Sigma}{\op}^2\vee 1\big)\cdot \pnorm{Z}{\op}^2\pnorm{\Sigma-I}{F}^2,\\
V_{1,2}& \leq \pnorm{S_Z}{\op}^2 \pnorm{Z}{F}^2 b^{-4}(\Sigma^{1/2} S_Z\Sigma^{1/2})b^{-4}(S_Z)\\
&\qquad\times \big(\tr((\Sigma-I)S_Z)/p\big)^2\big(b^2(\Sigma^{1/2} S_Z\Sigma^{1/2})\vee b^2(S_Z)\big)\\
&\lesssim \Big[ \pnorm{S_Z}{\op}^4 b^{-4}(\Sigma^{1/2} S_Z\Sigma^{1/2})b^{-4}(S_Z) \big(b^2(\Sigma^{1/2} S_Z\Sigma^{1/2})\vee b^2(S_Z)\big) \Big] \\
&\qquad \times p^{-1}\pnorm{Z}{F}^2  \pnorm{\Sigma-I}{F}^2.
\end{align*}
So under $p/N\leq M$, by Lemma \ref{lem:concentration_trace} and Lemma \ref{lem:kl}, we have
\begin{align*}
\E \pnorm{V_1(Z)}{F}^2 \lesssim_M N \big(\pnorm{\Sigma}{\op}^2\vee 1\big)  \pnorm{\Sigma-I}{F}^2.
\end{align*}
For $V_2(Z)$, 
\begin{align*}
&\pnorm{V_2(Z)}{F}^2= \pnorm{Z}{F}^2\Big(\frac{b_2(\Sigma^{1/2}S_Z\Sigma^{1/2})}{b^3(\Sigma^{1/2} S_Z\Sigma^{1/2})} - \frac{b_2(S_Z)}{b^3(S_Z)}\Big)^2\\
&\lesssim \pnorm{Z}{F}^2\bigg\{\Big(\frac{b_2(\Sigma^{1/2}S_Z\Sigma^{1/2})}{b^3(\Sigma^{1/2} S_Z\Sigma^{1/2})} - \frac{b_2(S_Z)}{b^3(\Sigma^{1/2} S_Z\Sigma^{1/2})}\Big)^2 \\
&\qquad + \Big(\frac{b_2(S_Z)}{b^3(\Sigma^{1/2} S_Z\Sigma^{1/2})} - \frac{b_2(S_Z)}{b^3(S_Z)}\Big)^2\bigg\}\equiv V_{2,1} + V_{2,2}.
\end{align*}
Note that 
\begin{align}\label{ineq:ratio_control_john_1}
&\big(b_2(\Sigma^{1/2}S_Z\Sigma^{1/2})-b_2(S_Z)\big)^2 = p^{-2}\tr^2\big(S_Z\Sigma S_Z\Sigma -S_Z^2\big) \nonumber\\
& \lesssim p^{-2} \Big\{\tr^2\big(S_Z(\Sigma - I)S_Z\Sigma\big) + \tr^2\big(S_Z^2(\Sigma-I)\big)\Big\} \nonumber\\
&\lesssim p^{-1} \pnorm{S_Z}{\op}^4\big(\pnorm{\Sigma}{F}^2/p+1\big) \pnorm{\Sigma-I}{F}^2 \\
& \lesssim p^{-1} \pnorm{S_Z}{\op}^4 \big(\pnorm{\Sigma}{\op}^2\vee 1\big)  \pnorm{\Sigma-I}{F}^2, \nonumber
\end{align}
so 
\begin{align*}
V_{2,1} 
&\lesssim \Big[ b^{-6}(\Sigma^{1/2} S_Z\Sigma^{1/2}) \pnorm{S_Z}{\op}^4\Big] \cdot \big(\pnorm{\Sigma}{\op}^2\vee 1\big) \cdot \big(p^{-1}\pnorm{Z}{F}^2\big)\cdot\pnorm{\Sigma-I}{F}^2,\\
V_{2,2} &\leq \pnorm{Z}{F}^2 b^{-6}(\Sigma^{1/2} S_Z\Sigma^{1/2})b^{-6}(S_Z) b_2^2(S_Z) \big(b(\Sigma^{1/2} S_Z\Sigma^{1/2}) - b(S_Z)\big)^2\\
&\qquad\times\Big(b^2(\Sigma^{1/2} S_Z\Sigma^{1/2}) + b(\Sigma^{1/2} S_Z\Sigma^{1/2})b(S_Z) + b^2(S_Z)\Big)^2\\
&\lesssim  \Big[ b^{-6}(\Sigma^{1/2} S_Z\Sigma^{1/2})b^{-6}(S_Z)\big(b^4(\Sigma^{1/2} S_Z\Sigma^{1/2})\vee b^4 (S_Z)\big) \pnorm{S_Z}{\op}^6 \Big] \\
&\qquad\times p^{-1}\pnorm{Z}{F}^2\pnorm{\Sigma - I}{F}^2.
\end{align*}
Hence under $p/N\leq M$, by Lemma \ref{lem:concentration_trace} and Lemma \ref{lem:kl}, we have
\begin{align*}
\E \pnorm{V_2(Z)}{F}^2 \lesssim_M N \big(\pnorm{\Sigma}{\op}^2\vee 1\big)  \pnorm{\Sigma-I}{F}^2.
\end{align*}
Lastly, recall that $\tr(\Sigma) = p$ so using trace H\"older inequality we have $\tr(S_Z\Sigma S_Z\Sigma)\leq \tr(\Sigma)\pnorm{S_Z\Sigma S_Z}{\op}\leq p\pnorm{S_Z}{\op}^2 \pnorm{\Sigma}{\op}$, so  $V_3(Z)$ satisfies
\begin{align*}
\pnorm{V_3(Z)}{F}^2 &\leq p^{-2}\pnorm{\Sigma -I}{F}^2\cdot \pnorm{Z}{\op}^2 \cdot b^{-6}(\Sigma S_Z)\cdot \tr^2(S_Z\Sigma S_Z\Sigma)\\
&\leq \Big[b^{-6}(\Sigma^{1/2} S_Z\Sigma^{1/2}) \pnorm{S_Z}{\op}^4 \Big]\cdot  \pnorm{\Sigma}{\op}^2\cdot   \pnorm{Z}{\op}^2 \pnorm{\Sigma -I}{F}^2
\end{align*}
Hence under $p/N\leq M$, by Lemma \ref{lem:concentration_trace} and Lemma \ref{lem:kl}, we have
\begin{align*}
\E \pnorm{V_3(Z)}{F}^2 \lesssim_M  N\big(\pnorm{\Sigma}{\op}^2\vee 1\big) \pnorm{\Sigma-I}{F}^2.
\end{align*}
Combining the estimates proves the claim. 

\noindent (2). Recall the normalization $b(\Sigma)=1$. Note that
\begin{align*}
&\E \bigg[\frac{b_2(\Sigma^{1/2} S_Z\Sigma^{1/2})}{b^2(\Sigma^{1/2} S_Z\Sigma^{1/2})}\bigg]\\
& = \frac{\E b_2(\Sigma^{1/2} S_Z\Sigma^{1/2})}{ \E b^2(\Sigma^{1/2} S_Z\Sigma^{1/2}) } + \E \bigg[b_2(\Sigma^{1/2} S_Z\Sigma^{1/2})\bigg(\frac{1}{b^2(\Sigma^{1/2} S_Z\Sigma^{1/2})}-\frac{1}{\E b^2(\Sigma^{1/2} S_Z\Sigma^{1/2}) }\bigg)\bigg]\\
& \stackrel{(*)}{=} \frac{(1+N^{-1}) \frac{b_2(\Sigma)}{b^2(\Sigma)}+\frac{p}{N}}{1+2\tr(\Sigma^2)/(Np^2)}+\E \bigg[b_2(\Sigma^{1/2} S_Z\Sigma^{1/2})\bigg(\frac{1}{b^2(\Sigma^{1/2} S_Z\Sigma^{1/2})}-\frac{1}{\E b^2(\Sigma^{1/2} S_Z\Sigma^{1/2}) }\bigg)\bigg]\\
& = \frac{b_2(\Sigma)}{b^2(\Sigma)}+\frac{p}{N}+\bigg\{\E \bigg[b_2(\Sigma^{1/2} S_Z\Sigma^{1/2})\bigg(\frac{1}{b^2(\Sigma^{1/2} S_Z\Sigma^{1/2})}-\frac{1}{\E b^2(\Sigma^{1/2} S_Z\Sigma^{1/2}) }\bigg)\bigg]\\
&\qquad +\bigg[(1+N^{-1}) \frac{b_2(\Sigma)}{b^2(\Sigma)}+\frac{p}{N}\bigg]\bigg(\frac{1}{1+2\tr(\Sigma^2)\cdot(Np^2)^{-1}}-1\bigg)+ N^{-1}\frac{b_2(\Sigma)}{b^2(\Sigma)} \bigg\}\\
&\equiv \frac{b_2(\Sigma)}{b^2(\Sigma)}+\frac{p}{N}+R(\Sigma).
\end{align*}
Here we use Lemma \ref{lem:trace_moment}-(1) in $(*)$ and 
\begin{align*}
R(\Sigma)&= \E \bigg[b_2(\Sigma^{1/2} S_Z\Sigma^{1/2})\bigg(\frac{1}{b^2(\Sigma^{1/2} S_Z\Sigma^{1/2})}-\frac{1}{\E b^2(\Sigma^{1/2} S_Z\Sigma^{1/2}) }\bigg)\bigg]\\
&\qquad +\bigg[(1+N^{-1}) \frac{b_2(\Sigma)}{b^2(\Sigma)}+\frac{p}{N}\bigg]\bigg(\frac{1}{1+2\tr(\Sigma^2)\cdot(Np^2)^{-1}}-1\bigg)+ N^{-1}\frac{b_2(\Sigma)}{b^2(\Sigma)}\\
& \equiv R_1(\Sigma)+R_2(\Sigma)+R_3(\Sigma).
\end{align*}
As
\begin{align*}
m_{\Sigma;\jo} &= \frac{N}{4}\E \tr\Big(\frac{\Sigma^{1/2}S_Z\Sigma^{1/2}}{b(\Sigma^{1/2}S_Z\Sigma^{1/2})}-I\Big)^2= \frac{Np}{4} \bigg\{ \E \bigg[\frac{b_2(\Sigma^{1/2} S_Z\Sigma^{1/2})}{b^2(\Sigma S_Z)}\bigg]-1\bigg\},
\end{align*}
we have
\begin{align*}
m_{\Sigma;\jo}-m_{I;\jo}& = \frac{Np}{4}\big(p^{-1}\pnorm{\Sigma-I}{F}^2+R(\Sigma)-R(I)\big)
\end{align*}
Now we handle $R_\ell(\Sigma)-R_\ell(I)$ for $\ell=1,2,3$. 

For $\ell=1$, 
\begin{align*}
&\bigabs{R_1(\Sigma)-R_1(I)}\\
&= \bigg\lvert \E \bigg[b_2(\Sigma^{1/2} S_Z\Sigma^{1/2})\bigg(\frac{1}{b^2(\Sigma^{1/2} S_Z\Sigma^{1/2})} -\frac{1}{\E b^2(\Sigma^{1/2} S_Z\Sigma^{1/2}) }\bigg)\\
&\qquad -b_2(S_Z)\bigg(\frac{1}{b^2(S_Z)}-\frac{1}{\E b^2(S_Z) }\bigg)\bigg] \bigg\rvert\\
&\leq \biggabs{ \E \bigg[ \big(b_2(\Sigma^{1/2} S_Z\Sigma^{1/2})-b_2(S_Z)\big)\bigg(\frac{1}{b^2( \Sigma^{1/2} S_Z\Sigma^{1/2})}-\frac{1}{\E b^2( \Sigma^{1/2} S_Z\Sigma^{1/2}) }\bigg)  \bigg]}\\
&\qquad + \bigg\lvert \E \bigg[b_2(S_Z)\bigg\{\bigg(\frac{1}{b^2(\Sigma^{1/2} S_Z\Sigma^{1/2})}-\frac{1}{ b^2( S_Z) }\bigg)   \\
&\qquad\qquad\qquad\qquad  - \bigg(\frac{1}{\E b^2(\Sigma^{1/2} S_Z\Sigma^{1/2})}-\frac{1}{ \E b^2( S_Z) }\bigg) \bigg\}\bigg]   \bigg\rvert \equiv R_{1,1}+R_{1,2}.
\end{align*}
The term $R_{1,1}$ can be handled as follows: by (\ref{ineq:ratio_control_john_1}) Lemmas \ref{lem:kl}, \ref{lem:concentration_trace}, and \ref{lem:trace_moment}, under $p/N\leq M$,
\begin{align*}
R_{1,1}&\lesssim \E^{1/4}\big(b_2(\Sigma^{1/2} S_Z\Sigma^{1/2})-b_2(S_Z)\big)^4 \cdot \E^{1/4} b^{-8}(\Sigma^{1/2}S\Sigma^{1/2})\\
&\quad\quad\quad\times\var^{1/2}\big(b^2(\Sigma^{1/2} S_Z\Sigma^{1/2})\big)\cdot \big(\E b^2(\Sigma^{1/2}S\Sigma^{1/2})\big)^{-1}\\
&\lesssim_M p^{-2}\cdot\Big[p^{-1/2}\big(p^{-1/2}\pnorm{\Sigma}{F}+1\big) \pnorm{\Sigma-I}{F}\Big]\cdot \var^{1/2}\big(\tr^2( \Sigma S_Z)\big)\\
& \lesssim_M (N^{1/2}p)^{-1} \big(p^{-1}\pnorm{\Sigma}{F}^2+1\big) \pnorm{\Sigma-I}{F}.
\end{align*}
For $R_{1,2}$, we have by Lemmas \ref{lem:kl} and \ref{lem:trace_moment} that, under $p/N\leq M$,
\begin{align*}
R_{1,2}& = \biggabs{\E \bigg[b_2(S_Z)\bigg(\frac{b^2(S_Z)-b^2(\Sigma^{1/2} S_Z\Sigma^{1/2})}{b^2(\Sigma^{1/2} S_Z\Sigma^{1/2})b^2(S_Z)}-\frac{\E b^2(S_Z)-\E b^2(\Sigma^{1/2} S_Z\Sigma^{1/2})}{ \E b^2(\Sigma^{1/2} S_Z\Sigma^{1/2}) \E b^2( S_Z)}\bigg)\bigg] }\\
&\leq \E b_2(S_Z) b^{-2}(\Sigma^{1/2} S_Z\Sigma^{1/2}) b^{-2}(S_Z)\\
&\qquad \times \bigabs{ b^2(S_Z)-b^2(\Sigma^{1/2} S_Z\Sigma^{1/2})-\E \big(b^2(S_Z)-b^2(\Sigma^{1/2} S_Z\Sigma^{1/2})\big) }\\
&\qquad + \bigabs{\E\big(b^2(S_Z)-b^2(\Sigma^{1/2} S_Z\Sigma^{1/2})\big)}\cdot \E b_2(S_Z) \\
&\qquad\qquad\times \biggabs{\frac{1}{b^2(\Sigma^{1/2} S_Z\Sigma^{1/2}) b^2(S_Z)}-\frac{1}{\E b^2(\Sigma^{1/2} S_Z\Sigma^{1/2}) \E b^2(S_Z)} }\\
&\lesssim_M \var^{1/2} \big(b^2(S_Z)-b^2(\Sigma^{1/2} S_Z\Sigma^{1/2})\big)\\
&\qquad + \bigabs{\E\big(b^2(S_Z)-b^2(\Sigma^{1/2} S_Z\Sigma^{1/2})\big)}\cdot \big(\var^{1/2}(b^2(\Sigma^{1/2} S_Z\Sigma^{1/2}))\vee \var^{1/2}(b^2( S_Z))\big)\\
&\stackrel{(\ast)}{\lesssim}_M  (N^{1/2}p)^{-1} \pnorm{\Sigma-I}{F}+ p^{-1/2}\pnorm{\Sigma-I}{F}\cdot (N^{1/2}p)^{-1}(\pnorm{\Sigma}{F}\vee p^{1/2})\\
& \lesssim (N^{1/2}p)^{-1} \big(p^{-1/2}\pnorm{\Sigma}{F}+1\big) \pnorm{\Sigma-I}{F}.
\end{align*}
Here in $(*)$ we use the fact that
\begin{align*}
&\bigabs{\E\big(b^2(S_Z)-b^2(\Sigma^{1/2} S_Z\Sigma^{1/2})\big)}\lesssim p^{-1} \E^{1/2} \tr^2\big((\Sigma-I)S_Z\big)\\
&\leq p^{-1/2} \pnorm{\Sigma-I}{F} \cdot \E^{1/2}\pnorm{S_Z}{\op}^2 \lesssim_M p^{-1/2} \pnorm{\Sigma-I}{F} .
\end{align*}
Hence
\begin{align*}
\abs{R_1(\Sigma)-R_1(I)}\lesssim_M (N^{1/2}p)^{-1} \big(p^{-1}\pnorm{\Sigma}{F}^2+1\big) \pnorm{\Sigma-I}{F}.
\end{align*}
For $\ell=2$, with $\mathfrak{a}(\Sigma)\equiv1/\big(1+2\tr(\Sigma^2)/(Np^2)\big)-1$ (then $\abs{\mathfrak{a}(\Sigma)}\leq 2/N$ and $\abs{\mathfrak{a}(I)}\leq 2/(Np)$), we have 
\begin{align*}
R_2(\Sigma) = (1+N^{-1}) b_2(\Sigma)\mathfrak{a}(\Sigma)+N^{-1}p\mathfrak{a}(\Sigma),
\end{align*}
so
\begin{align*}
\bigabs{R_2(\Sigma)-R_2(I)}&\lesssim_M \bigabs{b_2(\Sigma)\mathfrak{a}(\Sigma)-b_2(I)\mathfrak{a}(I) }+\abs{\mathfrak{a}(\Sigma)-\mathfrak{a}(I)}\equiv R_{2,1}+R_{2,2}.
\end{align*}
The two terms $R_{2,1},R_{2,2}$ can be handled as follows: using $\tr(\Sigma^2)\leq p^2$ under $b(\Sigma)=1$, we have
\begin{align*}
R_{2,1}&\lesssim b_2(\Sigma) \abs{\mathfrak{a}(\Sigma)-\mathfrak{a}(I)}+\abs{\mathfrak{a}(I)}\abs{b_2(\Sigma)-b_2(I)}\\
&\lesssim p^{-1}\tr(\Sigma^2) (Np^2)^{-1} \abs{\tr\big(\Sigma^2-I\big)} + (Np)^{-1} \cdot p^{-1}\cdot \abs{\tr(\Sigma^2-I)}\\
&\lesssim (Np^{1/2})^{-1} \big(p^{-1/2}\pnorm{\Sigma}{F}+1\big) \pnorm{\Sigma-I}{F},  \\
R_{2,2}&\lesssim (Np^{1/2})^{-1} \big(p^{-1/2}\pnorm{\Sigma}{F}+1\big) \pnorm{\Sigma-I}{F},
\end{align*}
so
\begin{align*}
\bigabs{R_2(\Sigma)-R_2(I)}&\lesssim_M (Np^{1/2})^{-1} \big(\pnorm{\Sigma}{F}/p^{1/2}+1\big) \pnorm{\Sigma-I}{F}.
\end{align*}
For $\ell=3$, 
\begin{align*}
\bigabs{R_3(\Sigma)-R_3(I)}&=N^{-1}\bigabs{b_2(\Sigma) -b_2(I)} \lesssim(Np^{1/2})^{-1} \big(p^{-1/2}\pnorm{\Sigma}{F}+1\big) \pnorm{\Sigma-I}{F}.
\end{align*}
Now with $Q_{\jo}(\Sigma)\equiv p\big(R(\Sigma)-R(I)\big)$, we have
\begin{align*}
\abs{Q_{\jo}(\Sigma)}&\lesssim_M p\max\{(N^{1/2}p)^{-1},(Np^{1/2})^{-1}\} \big(p^{-1}\pnorm{\Sigma}{F}^2+1\big) \pnorm{\Sigma-I}{F}\\
&\lesssim_M N^{-1/2}\big(p^{-1}\pnorm{\Sigma}{F}^2+1\big) \pnorm{\Sigma-I}{F},
\end{align*}
and
\begin{align*}
m_{\Sigma;\jo}-m_{I;\jo}& = \frac{N}{4}\big(\pnorm{\Sigma-I}{F}^2+Q_{\jo}(\Sigma)\big).
\end{align*}

\noindent (3). 
Recall $T_{\na}$ defined in (\ref{def:lrs_nagao}). Let
\begin{align*}
\Delta(X) &\equiv T_{\na}(X) - T_{\jo}(X).
\end{align*}
Then for any $\epsilon > 0$, there exists some $C_\epsilon>0$ such that under the null (i.e., $X_1,\ldots,X_n$ are i.i.d. $\mathcal{N}(0,I_p)$),
\begin{align}\label{ineq:variance_compare_john}
\big[(1-\epsilon)\sigma^2_{I;\na} -C_\epsilon \var_I(\Delta)\big]_+ 
\leq \sigma^2_{I;\jo}\leq (1+\epsilon)\sigma^2_{I;\na} + C_\epsilon \var_I(\Delta).
\end{align}
We will now bound $\var_I(\Delta)$. By Lemmas \ref{lem:derivatives_T_nagao}-(1) and \ref{lem:derivatives_john}-(1), we have for any $i,j\in[N]\times [p]$
\begin{align*}
\partial_{(ij)}\Delta(X) &=  \partial_{(ij)}T_{\na}(X) - \partial_{(ij)}T_{\jo}(X) \\
&=  \big(X(S-I) - N^{-1}\tr(S)X\big)_{ij}-\bigg[\frac{X}{b}\Big(\frac{S}{b}-I\Big)-\frac{X}{b(S)}\cdot b\Big( \Big(\frac{S}{b(S)}-I\Big)^2\Big)\bigg]_{ij}\\
& = \bigg[X(S-I)- \frac{X}{b}\Big(\frac{S}{b}-I\Big)\bigg]_{ij}+ \bigg[Xb\big((S-b(S)I)^2\big)\big(b^{-3}(S)-1\big)\bigg]_{ij}\\
&\qquad +\bigg[X\Big(b_2(S)-b^2(S)-N^{-1}\tr(S)\Big)\bigg]_{ij}\equiv \big(\Delta_1+\Delta_2+\Delta_3\big)_{ij}.
\end{align*}
We now handle $\Delta_1$-$\Delta_3$ separately below. For $\Delta_1$, by Lemmas \ref{lem:kl}, \ref{lem:concentration_trace}, and \ref{lem:trace_moment}, we have
\begin{align*}
\E \pnorm{\Delta_1}{F}^2 &\lesssim \E b^{-2}(1-b)^2 \pnorm{X(S-I)}{F}^2+ \E  b^{-4}(1-b)^2 \pnorm{X S}{F}^2\\
&\leq N\E b^{-2}(1-b)^2\pnorm{S}{\op}\pnorm{S-I}{F}^2 + N\E b^{-4}(b-1)^2\tr(S^3)\\
&\lesssim N \cdot (pN)^{-1} \cdot (1\vee y) (N^{-1}p^2) + N\E^{1/4} b^{-16}\E^{1/4}(b-1)^8 \E^{1/2}\tr^2(S^3)\\
&\stackrel{(*)}{\lesssim}\mathfrak{o}(p^2) + N\cdot \mathcal{O}(1) \cdot (Np)^{-1}\cdot \mathcal{O}(N^{-2}p^3\vee p^2) = \mathfrak{o}(p^2).
\end{align*}
Here in $(*)$ the first bound follows by direct calculation and the second bound follows as: by Lemma \ref{lem:trace_asymptotics}-(7),
\begin{align*}
\E\tr^2(S^3) &\leq \E \tr^2(S^2)\pnorm{S}{\op}^2 \leq \E^{1/2}\tr^4(S^2)\cdot \E^{1/2} \pnorm{S}{\op}^4\\
&\lesssim  p^4 \cdot (y^2\vee 1) = \mathcal{O}(N^{-4}p^6 \vee p^4).
\end{align*}
For $\Delta_2$, using $b\big((S-b(S)I)^2\big) \leq \pnorm{(S-b(S)I)^2}{\op}\lesssim \pnorm{S}{\op}^2\vee b^2(S)$, we have 
\begin{align*}
\E \pnorm{\Delta_2}{F}^2 &\lesssim \E b^{-6}(b^4\vee b^2\vee 1)(b-1)^2\big(\pnorm{S}{\op}^2\vee b^2\big)\pnorm{X}{F}^2\\
&\lesssim  (pN)^{-1}\cdot (pN) \cdot \E^{1/2}\big(\pnorm{S}{\op}^4\vee b^4\big)\asymp (1\vee y)^2 = \mathfrak{o}(p^2).
\end{align*}
For $\Delta_3$, let $h(S)\equiv b_2(S)-b^2(S)-N^{-1}\tr(S)$, we have
\begin{align*}
\E \pnorm{\Delta_3}{F}^2  &\lesssim \E \pnorm{X}{F}^2h^2(S) \leq N\E^{1/2}\tr^2(S)\cdot \E^{1/2} h^4(S)\\
&\lesssim Np\cdot \Big[\big(\E h(S)\big)^4 + \var^2\big(h(S)\big) + \E \pnorm{\nabla h(S)}{F}^4\Big]^{1/2},
\end{align*}
where the last inequality follows since
\begin{align*}
\E h^4 (S)& = [\E h^2(S)]^2 +\var(h^2(S)) \leq 2\big[\E h (S))^4 + 2\var^2(h(S))+\var(h^2(S)\big]\\
&\leq 2(\E h (S))^4 + 2\var^2(h(S))+ 4\E h^2(S) \pnorm{\nabla h(S)}{F}^2\\
&\leq 2(\E h (S))^4 + 2\var^2(h(S))+ 4 \tau \E h^4(S)+ C_\tau \E \pnorm{\nabla h(S)}{F}^4
\end{align*}
and choosing, say, $\tau=1/8$. For $\E h(S)$, Lemma \ref{lem:trace_moment} yields the direct evaluation
\begin{align*}
\E h(S) & = \frac{(1+N^{-1})p+N^{-1}p^2}{p}-\frac{p^2+2N^{-1}p}{p^2}-\frac{p}{N}\\
& = (1+N^{-1})+\frac{p}{N}-\Big(1+\frac{2}{Np}\Big)-\frac{p}{N} = \frac{1}{N}-\frac{2}{Np} = \mathcal{O}(N^{-1}).
\end{align*}
For $\var\big(h(S)\big)$, the Gaussian-Poincar\'e inequality \cite[Theorem 3.20]{boucheron2013concentration} yields that
\begin{align*} 
\var \big(h(S)\big)&\leq \E \sum_{ij} \big(\partial_{(i j)} h(S)\big)^2 \stackrel{(\ast\ast)}{=} \sum_{ij} \E \bigg(\frac{4X_i^\top S e_j}{Np}-\frac{4 X_{ij}b(S)}{Np}-\frac{2X_{ij}}{N^2}\bigg)^2\\
&\lesssim (Np)^{-2}\big(\E \pnorm{XS}{F}^2+\E b^2(S) \pnorm{X}{F}^2\big)+N^{-4}\E\pnorm{X}{F}^2\\
&\lesssim (Np^2)^{-1} \E\tr(S^2)\pnorm{S}{\op} + (Np)^{-2}(Np) + N^{-4}(Np)\\
&\lesssim (Np^2)^{-1}\E^{1/2}\tr^2(S^2)\E^{1/2}\pnorm{S}{\op}^2 + (Np)^{-1} + pN^{-3}\\
&\lesssim (Np^2)^{-1}\cdot p^2\cdot (1\vee y) + (Np)^{-1} + pN^{-3} = \mathfrak{o}(N^{-1}p).
\end{align*}
Here ($**$) follows from (\ref{ineq:derivative_john_1}). Lastly $\E \pnorm{\nabla h(S)}{F}^4$ can be bounded similarly:
\begin{align*}
\E \pnorm{\nabla h(S)}{F}^4&\lesssim (Np)^{-4}\big(\E \pnorm{XS}{F}^4+\E b^4(S) \pnorm{X}{F}^4\big)+N^{-8}\E\pnorm{X}{F}^4 
= \mathfrak{o}(N^{-2}p^2).
\end{align*}
Now by the Gaussian-Poincar\'e inequality \cite[Theorem 3.20]{boucheron2013concentration},
\begin{align*}
\var_I(\Delta)&\leq \E \sum_{ij} \big(\partial_{(ij)}\Delta(X)\big)^2 \lesssim \E\pnorm{\Delta_1}{F}^2+ \E\pnorm{\Delta_2}{F}^2 +\E\pnorm{\Delta_3}{F}^2 = \mathfrak{o}(p^2). 
\end{align*}
As $\sigma_{I;\na}^2\sim p^2/4\to \infty$ whenever $N \wedge p \to \infty$, by taking $\epsilon$ in (\ref{ineq:variance_compare_john}) slowly decaying to $0$ we conclude $\sigma_{I;\na}^2\sim \sigma_{I;\jo}^2$.

\noindent (4). By (1)-(2), as $\pnorm{\Sigma}{F}^2/p\lesssim \pnorm{\Sigma-I}{F}^2/p+1\lesssim \pnorm{\Sigma-I}{F}\vee 1$ [where we use $\pnorm{\Sigma-I}{F}\leq \pnorm{\Sigma}{F}+\sqrt{p}\leq p+\sqrt{p}$ under $\tr(\Sigma)=p$], we only need to prove that given $C_0$, we may find some constant $C_1>0$,
\begin{align}\label{ineq:power_john_1}
&\frac{\sqrt{N} \pnorm{\Sigma-I}{F}\big(1\vee \pnorm{\Sigma-I}{F} \big) }{\big(N \pnorm{\Sigma-I}{F}^2-C_0 N^{1/2} \pnorm{\Sigma-I}{F}\big(1\vee \pnorm{\Sigma-I}{F}\big) \big)_+\vee \sigma_{I;\jo}}\leq \frac{C_1}{(\sigma_{I;\jo}\wedge N)^{1/2}}.
\end{align}
Write $\alpha = \pnorm{\Sigma-I}{F}$, we only need to prove that
\begin{align}\label{ineq:power_john_2}
\frac{\sqrt{N}\alpha\vee \sqrt{N}\alpha^2}{\big(N\alpha^2-C_0 N^{1/2}\alpha\big)_+\vee \sigma_{I;\jo}}\leq \frac{C_1}{(\sigma_{I;\jo}\wedge N)^{1/2}}.
\end{align}
This follows as
\begin{align*}
\hbox{LHS of (\ref{ineq:power_john_2})}&\lesssim \bm{1}_{\alpha \leq 2C_0 N^{-1/2}} \frac{1}{\sigma_{I;J}}+ \bm{1}_{\alpha > 2C_0 N^{-1/2}}\frac{\sqrt{N}\alpha\vee \sqrt{N}\alpha^2}{N\alpha^2\vee \sigma_{I;\jo}}\\
&\lesssim \frac{1}{\sigma_{I;J}}+ \frac{1}{N^{1/2} \inf_{\alpha\geq 0}\big(\alpha \vee \frac{\sigma_{I;\jo}/N}{\alpha}\big) } + \frac{1}{N^{1/2}}\\
&\lesssim  \frac{1}{\sigma_{I;J}}+ \frac{1}{\sigma_{I;\jo}^{1/2}}+\frac{1}{N^{1/2}}\asymp \frac{1}{(\sigma_{I;\jo}\wedge N)^{1/2}}.
\end{align*}
The proof is complete. 
\end{proof}

\subsubsection{Completing of the proof for power expansion}\label{subsec:john_power}

\begin{proof}[Proof of Theorem \ref{thm:power_john}]
The proof essentially follows that of Theorem \ref{thm:power_nagao} by noting that the key property used therein is $\abs{Q_{\na}(\Sigma)}\leq C_M N^{-1/2}(\pnorm{\Sigma-I}{F}\vee 1)\pnorm{\Sigma-I}{F}$, while here we have $\abs{Q_{\jo}(\Sigma\cdot b^{-1}(\Sigma))}\leq C_M N^{-1/2}\pnorm{\Sigma\cdot b^{-1}(\Sigma)-I}{F}\leq N^{-1/2}(\pnorm{\Sigma\cdot b^{-1}(\Sigma)-I}{F} \vee 1)\pnorm{\Sigma\cdot b^{-1}(\Sigma)-I}{F}$.
\end{proof}

\appendix

\section{Second-order Poincar\'e inequality}
The main tool used for proving normal approximations is the following second-order Poincar\'e inequality due to \cite{chatterjee2009fluctuations}. Recall that $W^{1,2}(\gamma_n)$ is the Gaussian Sobolev space defined in (\ref{def:sobolev}).

\begin{lemma}[Second-order Poincar\'e inequality]\label{lem:sec_poincare}
	Let $\xi$ be an $n$-dimensional standard normal random vector. Let $F: \R^n \to \R$ be an element of $W^{1,2}(\gamma_n)$. Let $\xi'$ be an independent copy of $\xi$. Define $T:\R^n\to \R$ by
	\begin{align*}
	T(y)\equiv \int_0^1 \frac{1}{2\sqrt{t}} \iprod{\nabla F (y)}{ \E_{\xi^\prime} \nabla F(\sqrt{t}y + \sqrt{1-t}\xi^\prime)}\,\mathrm{d}{t}.
	\end{align*}
	Then with $W\equiv F(\xi)$,
	\begin{align*}
	&d_{\mathrm{TV}}\bigg( \frac{W-\E W}{\sqrt{ \mathrm{Var}(W)}}, \,\mathcal{N}(0,1)\bigg)\leq   \frac{2 \sqrt{\mathrm{Var}(T(\xi))} }{ \mathrm{Var}(W)}.
	\end{align*}
\end{lemma} 
If furthermore $F(\xi) \in W^{2,4}(\gamma_n)$, a standard application of Gaussian-Poincar\'e inequality leads to the bound $\var(T(\xi))\lesssim \E^{1/2} \pnorm{\nabla^2 F(\xi)}{\op}^4\cdot \E^{1/2}\pnorm{\nabla
 F(\xi)}{}^4$.

\section{Sobolev regularity of matrix functionals}\label{section:regularity}

\begin{lemma}
	The following hold for any $q\geq 1$. 
	\begin{enumerate}
		\item Let $f:\R^{N\times p}\to \R$ be defined by $f(X) = \log \det (X^\top X)$. If $N\geq p+1$, then $f \in W^{1,q}(\gamma_{N\times p})$ provided itself and its pointwise first derivatives live in $L_q(\gamma_{N\times p})$, $f \in W^{2,q}(\gamma_{N\times p})$ provided itself, its pointwise first, and second derivatives live in $L_q(\gamma_{N\times p})$.	 In particular, if $N,p$ are large enough with $p/N\leq 1-\epsilon$ for some $\epsilon \in (0,1)$, then $f \in W^{2,4}(\gamma_{N\times p})$. 
		\item Let $g_\ell:\R^{N\times p}\to \R$ be defined by $g_\ell(X) = \tr^{-\ell}(X^\top X)$ for $\ell \in \N$. Then $g_\ell \in W^{1,q}(\gamma_{N\times p})$ provided itself and its pointwise first derivatives live in $L_q(\gamma_{N\times p})$, $g_\ell \in W^{2,q}(\gamma_{N\times p})$ provided itself, its pointwise first, and second derivatives live in $L_q(\gamma_{N\times p})$. In particular, there exists some $N_\ell \in \N$ such that for $N\geq N_\ell$, $g_\ell \in W^{2,4}(\gamma_{N\times p})$. 
	\end{enumerate}
\end{lemma}
\begin{proof}
	We work with $q=2$ for simplicity.
	
	\noindent (1). We first prove the claim involving $f \in W^{1,2}(\gamma_{N\times p})$. Let $W^{r,p}(\R^d)$ be the standard Sobolev class on $\R^d$ (cf. \cite[Chapter 1.5]{bogachev1998gaussian}) and recall that $C^\infty_0(\R^d)$ is the class of smooth functions on $\R^d$ with compact support. By \cite[Proposition 1.5.2]{bogachev1998gaussian}, we only need to verify that $\zeta f \in W^{1,2}(\R^{N\times p})$, i.e., $\zeta f \in L_2(\R^{N\times p})$ and its first partial derivatives (in the sense of distributions) live in $L_2(\R^{N\times p})$ for every $\zeta \in C_0^\infty(\R^{N\times p})$. $\zeta f \in L_2(\R^{N\times p})$ follows from $N\geq p+1> p$. Using the absolute continuity on line characterization of the space $W^{1,2}(\R^{N\times p})$ (cf. \cite[Section 1.1.3]{mazya2011sobolev}), we only need to show that $\zeta f$ is absolutely continuous on almost all straight lines that are parallel to coordinate axes and the first pointwise derivatives of $\zeta f$ belong to $L_2(\R^{N\times p})$. As $\zeta$ has compact support, the latter requirement is satisfied by the assumption that $f$ and its first pointwise derivatives live in $L_2(\gamma_{N\times p})$. To show the almost absolute continuity, we only need to do so for $f$ on a compact subset of $\R^{N\times p}$. Identify $X \in \R^{N\times p}$ in the matrix form $X=[X_1\cdots X_p]$ where $X_j \in \R^N$ for $1\leq j\leq p$, and in the coordinate form $X = (X_1^\top,\ldots,X_p^\top)$. Let $L_j\equiv \{(X_1^\top,\ldots,X_p^\top)\in \R^{N\times p}: X_j\in \mathrm{lin}(X_1,\ldots,X_{j-1},X_{j+1},\ldots,X_p)\}$, and $\pi_{-(ij)}:\R^{N\times p}\to \R^{N\times p-1}$ be the natural projection that excludes $(X_j)_i$. Then $\pi_{-(ij)}(L_{j'})$ is a subset of $\R^{N\times p-1}$ of Lebesgue measure $0$ for each $(i,j) \in [N]\times [p], j' \in [p]$ under the condition $p\leq N-1$, as we may write
	\begin{align*}
	L_{j'} =\Big\{\big(X_1^\top,\ldots,X_{j'-1}^\top,\sum_{j\neq j'} \gamma_j X_j^\top,X_{j'+1}^\top,\ldots,X_p^\top
	\big): X_j \in \R^N, \gamma_j \in \R, j\neq j' \Big\}.
	\end{align*}
	Hence with $L\equiv \cup_j L_j$, $\pi_{-(ij)}(L)$ is a subset of $\R^{N\times p-1}$ of Lebesgue measure $0$ for every $(i,j) \in[N]\times [p]$. In particular, this means that for any $X_{-(ij)} \notin \pi_{-(ij)}(L)$, the map $f$ along the line $x_{(ij)}\mapsto (x_{(ij)},X_{-(ij)})$ does not touch $\{X \in \R^{N\times p}: \det(X^\top X)=0\}$, and hence is locally Lipschitz as $\nabla f(X) = 2X(X^\top X)^{-1}$. This verifies the almost absolute continuity property, and hence $f \in W^{1,2}(\gamma_{N\times p})$ provided itself and	 its first pointwise derivatives live in $L_2(\gamma_{N\times p})$. 
	
	The verification of $f \in W^{2,2}(\gamma_{N\times p})$ under $L_2$ integrability of the pointwise derivatives up to the second order is the same, upon noting the derivatives have singularities only at $\{X \in \R^{N\times p}: \det(X^\top X)=0\}$ (the precise derivative formula is given in Lemma \ref{lem:derivatives_T}).
	
	The last assertion follows from Lemma \ref{lem:moment_S_inverse} and (\ref{ineq:clt_covariance_test_4}) that establishes the $L_4$ integrability of the pointwise first and second derivatives, and the straightforward verification of the $L_4$ integrability of $f$ itself.
	
	\noindent (2). The singularity of $g_\ell$ occurs only at $\tr(X^\top X) = \sum_j \pnorm{X_j}{}^2 = 0$, i.e., $X=0$. The almost absolute continuity on line characterization is therefore easily verified. The $L_4$ integrability of the derivatives up to the second order follows from Lemma \ref{lem:concentration_trace}.
\end{proof}

\section{Moment and concentration (in)equalities for trace functionals}

\begin{lemma}\label{lem:4th_moment}
	Let $Z \in \R^{N\times p}$ be a random matrix whose entries are i.i.d. $\mathcal{N}(0,1)$, and $A \in \R^{p\times p}$. Then
	\begin{align*}
	\E \pnorm{Z A}{F}^4\leq 4N \pnorm{A^\top A}{F}^2+N^2 \pnorm{A}{F}^4\leq 5N^2 \pnorm{A}{F}^4.
	\end{align*}
\end{lemma}
\begin{proof}
	As $\pnorm{ZA}{F}^2 = \tr(ZAA^\top Z^\top) = \tr(A^\top Z^\top Z A)$, we have
	\begin{align*}
	\E \pnorm{Z A}{F}^4& =  \E \tr^2 (A^\top Z^\top Z A) = \var\big( \tr(A^\top Z^\top Z A)\big)+\big(\E \tr(A^\top Z^\top Z A) \big)^2\\
	& = \var\big( \tr(A^\top Z^\top Z A)\big)+ N^2 \tr^2(A^\top A).
	\end{align*}
	Further note that for any $(i,j) \in [N]\times [p]$,
	\begin{align*}
	\partial_{(ij)} \tr(A^\top Z^\top Z A)& = \tr\big(A^\top (e_ie_j^\top)^\top Z A\big)+\tr\big(A^\top Z^\top e_i e_j^\top A\big)\\
	& = (ZAA^\top)_{ij}+(AA^\top Z^\top)_{ji} = 2(ZAA^\top)_{ij},
	\end{align*}
	so by Gaussian-Poincar\'e inequality, 
	\begin{align*}
	&\var\big( \tr(A^\top Z^\top Z A)\big)\leq \E \sum_{i,j} \big(\partial_{(ij)} \tr(A^\top Z^\top Z A)\big)^2\\
	&=4 \E \pnorm{Z AA^\top}{F}^2 = 4 \E \tr (ZAA^\top AA^\top Z^\top) = 4N \tr(AA^\top AA^\top).
	\end{align*}
	Finally note that 
	\begin{align*}
	\tr(AA^\top AA^\top) = \pnorm{A^\top A}{F}^2 = \sum_i \lambda_i^4(A)\leq  \bigg(\sum_i \lambda_i^2 (A)\bigg)^2=\pnorm{A}{F}^4 =\tr^2(A^\top A).
	\end{align*}
	The claim follows.
\end{proof}

\begin{lemma}\label{lem:concentration_trace}
	 Let $S_Z\equiv N^{-1}\sum_{i=1}^N Z_iZ_i^\top$ where $Z_i$'s are i.i.d. $\mathcal{N}(0,I)$ in $\R^p$. Then there exists some universal constant $C>0$ such that for any non-negative definite matrix $\Sigma$ and any $t>0$, 
	\begin{align*}
	\Prob\bigg(N\bigabs{\big(\tr(\Sigma S_Z)-\tr(\Sigma)\big)}>t\bigg)\leq 2\exp\bigg(-\frac{t^2}{C(N\pnorm{\Sigma}{F}^2+ \pnorm{\Sigma}{\op} t)}\bigg).
	\end{align*}
	Consequently, $\Prob\big(\tr(\Sigma S_Z)<\tr(\Sigma)/2\big)\leq e^{-cN}$ for some universal $c>0$. Furthermore, for any $\ell \in \mathbb{Z}$ such that $\ell\geq -N/2$, there exists some $C_\ell>0$ such that [recall (\ref{def:trace_b})]
	\begin{align*}
	\E b^{\ell}(\Sigma^{1/2} S_Z\Sigma^{1/2})\leq C_\ell \cdot b^{\ell}(\Sigma).
	\end{align*}
\end{lemma}
\begin{proof}
	Let $X_i \equiv \Sigma^{1/2} Z_i$. Then $\tr(\Sigma S_Z) = N^{-1}\sum_{i=1}^N Z_i^\top \Sigma Z_i = N^{-1}\sum_{i=1}^N \pnorm{X_i}{}^2$, and $\E \tr(\Sigma S_Z) = \E\pnorm{X_i}{}^2= \tr(\Sigma)$. By Hanson-Wright inequality (cf. \cite[pp.39]{boucheron2013concentration}), 
	\begin{align*}
	\E \exp\bigg(\lambda \sum_{i=1}^N (\pnorm{X_i}{}^2-\E\pnorm{X_i}{}^2)\bigg)\leq \exp\bigg(\frac{\lambda^2\cdot N\pnorm{\Sigma}{F}^2}{1-2\lambda\pnorm{\Sigma}{\op}}\bigg),
	\end{align*}
	so by \cite[Theorem 2.3]{boucheron2013concentration}, we have
	\begin{align*}
	\Prob\bigg(N\bigabs{\big(\tr(\Sigma S_Z)-\tr(\Sigma)\big)}>t\bigg)\leq 2\exp\bigg(-\frac{t^2}{C(N\pnorm{\Sigma}{F}^2+ \pnorm{\Sigma}{\op} t)}\bigg).
	\end{align*}
	In particular, with $t\equiv N\tr(\Sigma)/2$, we have
	\begin{align*}
	\Prob\big(\tr(\Sigma S_Z)<\tr(\Sigma)/2\big)&\leq \exp\bigg(-\frac{N^2 \tr^2(\Sigma)}{ C\big(N\pnorm{\Sigma}{F}^2+ N \pnorm{\Sigma}{\op}\tr(\Sigma)\big)}\bigg)\leq e^{-c N}.
	\end{align*}
	For the expectation bound, let $\{\lambda_j\}_{j=1}^p$ be the eigenvalues of $\Sigma$ and assume without loss of generality that $\sum_{j=1}^p \lambda_j = 1$. Then
	\begin{align*}
	\E\tr^{\ell}(S_Z\Sigma) &= \E\Big(\frac{1}{N}\sum_{i=1}^N Z_i^\top\Sigma Z_i\Big) = \E\Big(\frac{1}{N}\sum_{i=1}^N Z_i^\top\textrm{diag}(\lambda_1,\ldots,\lambda_p) Z_i\Big)^{\ell}\\
	&= N^{\ell}\E\bigg(\sum_{i=1}^N\sum_{j=1}^p \lambda_jZ_{ij}^2\bigg)^{\ell} \equiv N^{\ell}\E\bigg(\sum_{j=1}^p \lambda_jY_j\bigg)^{\ell}\\
	& \stackrel{(*)}{\leq} N^{\ell}\E\sum_{j=1}^p \lambda_jY_j^{\ell} \stackrel{(**)}{\lesssim_\ell} N^{\ell}\cdot \sum_{j=1}^p \lambda_jN^{\ell} = 1.
	\end{align*}
	Here $(*)$ follows as the map $x\mapsto x^{\ell}$ is convex on $(0,\infty)$ for $\ell\in\mathbb{Z}$, and $(**)$ follows from the following calculations:
	\begin{itemize}
		\item If $\ell \in \mathbb{Z}_{\geq 1}$, $\E Y_1^\ell = \E \big(\chi^2(N)\big)^\ell\lesssim_\ell N^{\ell}$.
		\item If $\ell \in \mathbb{Z}_{\leq -1}$ and $\ell\geq -N/2$, then
		\begin{align*}
		\E Y_1^\ell &= \E \big(\chi^{-2}(N)\big)^{-\ell}= \int x^{-\ell} \frac{2^{-\frac{N}{2}}}{\Gamma(N/2)} x^{-\frac{N}{2}-1}e^{-\frac{1}{2x}}\d x = 2^{\ell}\frac{\Gamma(N/2+\ell)}{\Gamma(N/2)} \lesssim_\ell N^{\ell}.
		\end{align*}
	\end{itemize}
	The proof is complete.
\end{proof}

\begin{lemma}\label{lem:res_log_trace}
	Let $S_Z\equiv N^{-1}\sum_{i=1}^N Z_i Z_i^\top$ where $Z_i$'s are i.i.d. $\mathcal{N}(0,I)$, and $\Sigma \in \R^{p\times p}$ be a non-negative definite matrix. Recall the definition of $b_\ell(\Sigma)$ in (\ref{def:trace_b}). Then for some universal constants $C, c>0$,
	\begin{align*}
	\bigabs{\E\log \tr(\Sigma S_Z)- \log \tr(\Sigma) } \leq \frac{2b\big[ (\Sigma\cdot b^{-1}(\Sigma))^2\big] }{Np} +C e^{-cN}\Big(\frac{ b^{1/2}\big[ (\Sigma\cdot b^{-1}(\Sigma))^2\big] }{(Np)^{1/2}} \vee 1\Big).
	\end{align*}
\end{lemma}
\begin{proof}
	Let $E\equiv \{\tr(\Sigma (S_Z-I))/\tr(\Sigma)\geq -1/2\}$. By Lemma \ref{lem:concentration_trace}, $\Prob(E^c)\leq e^{-cN}$ for some universal constant $c>0$. As $\abs{\log(1+x)-x}\leq 4x^2$ for $x\geq -1/2$,
	\begin{align*}
	&\bigabs{\E\log \tr(\Sigma S_Z)- \log \tr(\Sigma) } = \biggabs{\E \log \bigg(1+\frac{\tr \big(\Sigma(S_Z-I)\big)}{\tr(\Sigma)}\bigg)\big(\bm{1}_E+\bm{1}_{E^c}\big)}\\
	&\leq \biggabs{\E \frac{\tr \big(\Sigma(S_Z-I)\big)}{\tr(\Sigma)} \bm{1}_E }+ 4 \E \bigg(\frac{\tr \big(\Sigma(S_Z-I)\big)}{\tr(\Sigma)} \bigg)^2  + \E\bigg[ \log\Big(\frac{\tr(\Sigma S_Z)}{\tr(\Sigma)}\Big) \bm{1}_{E^c}\bigg]\\
	&\leq \E^{1/2} \bigg[\frac{\tr \big(\Sigma(S_Z-I)\big)}{\tr(\Sigma)} \bigg]^2 \Prob^{1/2}(E^c) + 4 \E \bigg(\frac{\tr \big(\Sigma(S_Z-I)\big)}{\tr(\Sigma)} \bigg)^2 \\
	&\qquad + \E^{1/2} \log^2\Big(\frac{\tr(\Sigma S_Z)}{\tr(\Sigma)}\Big)\cdot\Prob^{1/2}(E^c)\equiv  (I)+(II)+(III).
	\end{align*}
	To handle $(I)$, note that by Gaussian-Poincar\'e inequality \cite[Theorem 3.20]{boucheron2013concentration},
	\begin{align*}
	&\E \tr^2 \big(\Sigma(S_Z-I)\big) \leq \E \sum_{i,j} \big[\partial_{(ij)} \tr \big(\Sigma(S_Z-I)\big)\big]^2\\
	& = \E \sum_{i,j} \bigg[N^{-1}\tr\bigg(\Sigma \sum_k \big(\delta_{ik} e_j Z_k^\top+\delta_{ik} Z_k e_j^\top\big)\bigg)\bigg]^2\\
	& = \E \sum_{i,j} \big[N^{-1} \tr \big(\Sigma e_j Z_i^\top+\Sigma Z_i e_j^\top\big)\big]^2 = \frac{4}{N^2} \sum_{i,j} \E Z_i^\top \Sigma e_j e_j^\top \Sigma Z_i = \frac{4\tr(\Sigma^2)}{N},
	\end{align*}
	so
	\begin{align*}
	(I) = \frac{2 e^{-cN/2} \tr^{1/2}(\Sigma^2)}{N^{1/2}\tr(\Sigma)} = 2 e^{-cN/2}\cdot \frac{ b^{1/2}\big[ (\Sigma\cdot b^{-1}(\Sigma))^2\big] }{(Np)^{1/2}}.
	\end{align*}
	The second term has closed-form expression: by Lemma \ref{lem:trace_moment}-(1),
	\begin{align*}
	(II)& =  \frac{2\tr(\Sigma^2)}{N \tr^2(\Sigma)} .
	\end{align*}
	To handle $(III)$, by using $0 \leq \log x \leq x-1$ for $x\geq 1$ and $-x^{-1}\leq \log x < 0$ for $x\in(0,1)$, we have
	\begin{align*}
	\E\log^2\Big(\frac{\tr(\Sigma S_Z)}{\tr(\Sigma)}\Big) &= \E\bigg[\log^2\Big(\frac{\tr(\Sigma S_Z)}{\tr(\Sigma)}\Big)\bm{1}\Big\{\frac{\tr(\Sigma S_Z)}{\tr(\Sigma)}\geq 1\Big\}\bigg]\\
	&\quad +\E\bigg[\log^2\Big(\frac{\tr(\Sigma S_Z)}{\tr(\Sigma)}\Big)\bm{1}\Big\{\frac{\tr(\Sigma S_Z)}{\tr(\Sigma)}< 1\Big\}\bigg]\\
	&\leq \E \bigg[\frac{\tr \big(\Sigma(S_Z-I)\big)}{\tr(\Sigma)} \bigg]^2 + \E\bigg[\frac{\tr(\Sigma)}{\tr(\Sigma S_Z)}\bigg]^2
    \lesssim \frac{\tr(\Sigma^2)}{N \tr^2(\Sigma)} + 1,
	\end{align*}
	where in the last inequality we apply Lemma \ref{lem:concentration_trace}.
	Hence
	\begin{align*}
	(III)&\lesssim e^{-\frac{cN}{2}}\Big[\frac{\tr^{1/2}(\Sigma^2)}{N^{1/2} \tr(\Sigma)} \vee 1\Big] .
	\end{align*}
	The proof is complete by collecting the bounds.
\end{proof}

\begin{lemma}\label{lem:trace_moment}
	Let $S_Z\equiv N^{-1}\sum_{i=1}^N Z_iZ_i^\top$ where $Z_i$'s are i.i.d. $\mathcal{N}(0,I)$ in $\R^p$, and $\Sigma \in \R^{p\times p}$ be a non-negative definite matrix.
	\begin{enumerate}
		\item There exists some absolute $C>0$ such that
		\begin{align*}
		&\E \tr\big[\big(\Sigma^{1/2}S_Z\Sigma^{1/2}\big)^2\big] = \big(1+N^{-1}\big)\tr(\Sigma^2)+N^{-1}\tr^2(\Sigma),\\
		&\E \tr^2(\Sigma^{1/2}S_Z\Sigma^{1/2}) = \tr^2(\Sigma)+2N^{-1}\tr(\Sigma^2),\\
		&\E \tr^2 \big[\big(\Sigma^{1/2}S_Z\Sigma^{1/2}\big)^2\big]\leq C\Big[ N^{-1}(1\vee (p/N))^3 \tr(\Sigma^4)+\tr^2(\Sigma^2)+N^{-2}\tr^4(\Sigma)\Big].
		\end{align*}
		\item There exists some absolute $C>0$ such that
		\begin{align*}
		&\var\big( \tr (\Sigma S_Z)\big)\leq  4N^{-1} \pnorm{\Sigma}{F}^2,\\
		&\var \big(\tr^2(\Sigma^{1/2}S_Z\Sigma^{1/2})\big)\leq C (N^{-2} \tr^2(\Sigma) )\cdot N \pnorm{\Sigma}{F}^2,\\
		&\var\big( \tr^2 (\Sigma^{1/2}S_Z \Sigma^{1/2})-\tr^2(S_Z)\big)\\
		&\qquad 
		\lesssim \big(N^{-2} \tr^2(\Sigma-I)\big)\cdot  N \pnorm{\Sigma}{F}^2 + (N^{-1}p)^2\cdot N \pnorm{\Sigma-I}{F}^2,\\
		&\var \big(\tr\big[\big(\Sigma^{1/2}S_Z\Sigma^{1/2}\big)^2\big]\big)\leq  C N^{-1}\big[1\vee (N^{-1}p)\big]^3 \tr(\Sigma^4). 
		\end{align*}
		\item Recall that $b(\Sigma) = \tr(\Sigma)/p$ from (\ref{def:trace_b}). For any $\ell \in \N$,
		\begin{align*}
		\E\abs{b(\Sigma^{1/2}S_Z\Sigma^{1/2})-b(\Sigma)}^\ell \leq C_1 \big(\pnorm{\Sigma}{F}N^{-1/2}p^{-1}\big)^\ell
		\end{align*}
		for some constant $C_1=C_1(\ell)$.
	\end{enumerate}

\end{lemma}
\begin{proof}
	Let $X_i$'s be i.i.d. $\mathcal{N}(0,\Sigma)$. We write $S\equiv \Sigma^{1/2}S_Z\Sigma^{1/2}$ in the proof for simplicity. 
	
	\noindent (1). Note that
	\begin{align*}
	\E \tr(S^2) & = N^{-2} \E \tr\bigg[\sum_{i,j}X_iX_i^\top X_jX_j^\top\bigg] \\
	& = N^{-2}\bigg[\sum_{i\neq j} \E \tr \big(X_iX_i^\top X_jX_j^\top\big)+\sum_{i=j} \E (X_i^\top X_j)^2\bigg]\\
	& = N^{-2} \Big[N(N-1) \tr(\Sigma^2)+N \E (Z_1^\top \Sigma Z_1)^2\Big]\\
	& \stackrel{(\ast)}{=} N^{-2} \Big[N(N+1) \tr(\Sigma^2)+N \tr^2(\Sigma)\Big] \\
	&= \big(1+N^{-1}\big)\tr(\Sigma^2)+N^{-1}\tr^2(\Sigma),
	\end{align*}
	and
	\begin{align*}
	\E \tr^2(S)& = \E \bigg(N^{-1}\sum_{i=1}^N X_i^\top X_i\bigg)^2 = N^{-2}\sum_{i,j} \E X_i^\top X_i X_j^\top X_j\\
	& = N^{-2} \bigg[\sum_{i\neq j} \E \pnorm{X_i}{}^2\E\pnorm{X_j}{}^2+\sum_i \E (X_i^\top X_i)^2\bigg]\\
	& = N^{-2} \bigg[N(N-1) \big(\E Z_1^\top \Sigma Z_1\big)^2+N \E\big(Z_1^\top \Sigma Z_1\big)^2\bigg]\\
	&\stackrel{(\ast\ast)}{=} N^{-2}\Big[N(N-1)\tr^2(\Sigma)+N\big(\tr^2(\Sigma)+2\tr(\Sigma^2)\big)\Big]\\
	& = \tr^2(\Sigma)+2N^{-1}\tr(\Sigma^2).
	\end{align*}
	Here $(\ast),(\ast\ast)$ follow by the following calculations: 
	\begin{align*}
	\E Z_1^\top \Sigma Z_1& = \E \bigg(\sum_j \lambda_j Z_{1j}^2\bigg) = \tr(\Sigma),\\
	\E (Z_1^\top \Sigma Z_1)^2& = \E \bigg(\sum_{j} \lambda_j Z_{1j}^2\bigg)^2 = 3\sum_j \lambda_j^2 +\sum_{j\neq j'} \lambda_j\lambda_{j'}\\
	& =  2\sum_j \lambda_j^2+\bigg(\sum_j \lambda_j\bigg)^2 = \tr^2(\Sigma)+2 \tr(\Sigma^2),
	\end{align*}
	where $\lambda_1,\ldots,\lambda_p$ are the eigenvalues of $\Sigma$. The final one follows as
	\begin{align*}
	\E \tr^2 \big[\big(\Sigma^{1/2}S_Z\Sigma^{1/2}\big)^2\big]& = \var \big( \tr \big[\big(\Sigma^{1/2}S_Z\Sigma^{1/2}\big)^2\big]\big)+ \Big(\E\tr \big[\big(\Sigma^{1/2}S_Z\Sigma^{1/2}\big)^2\big]\Big)^2\\
	&\lesssim N^{-1}\big(1\vee (N^{-1}p)\big)^3 \tr(\Sigma^4)+\tr^2(\Sigma^2)+N^{-2}\tr^4(\Sigma).
	\end{align*}
	The last inequality used (2) to be proved below. 
	
	\noindent (2).	For the first variance bound, note that
	\begin{align*}
	\frac{\partial}{\partial Z_{ij}} \tr (\Sigma S_Z)& =  N^{-1}\tr\big(\Sigma(e_j Z_i^\top+Z_i e_j^\top)\big)= 2N^{-1} (Z\Sigma)_{ij},
	\end{align*}
	so Gaussian-Poincar\'e inequality yields that
	\begin{align*}
	\var\big( \tr (\Sigma S_Z)\big)&\leq  \E \sum_{i,j} \bigg[\frac{\partial}{\partial Z_{ij}} \tr^2 (\Sigma S_Z) \bigg]^2= 4N^{-2} \E \pnorm{Z\Sigma}{F}^2=4N^{-1} \pnorm{\Sigma}{F}^2.
	\end{align*}
	For the second variance bound, note that
	\begin{align*}
	\frac{\partial}{\partial Z_{ij}} \tr^2 (\Sigma S_Z)& = 2 \tr (\Sigma S_Z)\cdot N^{-1}\tr\big(\Sigma(e_j Z_i^\top+Z_i e_j^\top)\big)\\
	& = 4N^{-1} \tr(\Sigma S_Z) (Z\Sigma)_{ij},
	\end{align*}
	so Gaussian-Poincar\'e inequality yields that
	\begin{align*}
	\var\big( \tr^2 (\Sigma S_Z)\big)&\leq  \E \sum_{i,j} \bigg[\frac{\partial}{\partial Z_{ij}} \tr^2 (\Sigma S_Z) \bigg]^2= 16N^{-2} \E \tr^2(\Sigma S_Z) \pnorm{Z\Sigma}{F}^2\\
	&\stackrel{(\ast)}{\leq} N^{-2} \tr^2(\Sigma) \cdot \E^{1/2} \pnorm{Z\Sigma }{F}^4\stackrel{(**)}{\lesssim} (N^{-2} \tr^2(\Sigma) )\cdot N \pnorm{\Sigma}{F}^2.
	\end{align*}
	Here in $(\ast)$ we use Lemma \ref{lem:concentration_trace}, and in $(**)$ we use Lemma \ref{lem:4th_moment}. 
	
	For the third variance bound, note that
	\begin{align*}
	\frac{\partial}{\partial Z_{ij}} \big( \tr^2 (\Sigma S_Z)-\tr^2(S_Z) \big) = 4N^{-1} \big(\tr(\Sigma S_Z) (Z\Sigma)_{ij}-\tr(S_Z) Z_{ij} \big).
	\end{align*}
	Hence
	\begin{align*}
	&\var\big( \tr^2 (\Sigma S_Z)-\tr^2(S_Z)\big)\\
	&\lesssim N^{-2} \E \tr^2\big((\Sigma-I)S_Z\big) \pnorm{Z\Sigma}{F}^2+N^{-2} \E \tr^2(S_Z) \pnorm{Z(\Sigma-I)}{F}^2\\
	&\lesssim \big[N^{-2} \tr^2(\Sigma-I)\big]\cdot  N \pnorm{\Sigma}{F}^2+ (N^{-1}p)^2\cdot N \pnorm{\Sigma-I}{F}^2.
	\end{align*}
	
	For the fourth variance bound, note that 
	\begin{align*}
	\frac{\partial}{\partial Z_{ij}} \tr (\Sigma S_Z\Sigma S_Z) & = 2N^{-1}\tr\Big[\Sigma S_Z \Sigma (e_j Z_i^\top+Z_i e_j^\top)\Big] =4N^{-1}\big(Z\Sigma S_Z \Sigma\big)_{ij},
	\end{align*}
	so by Gaussian-Poincar\'e inequality and Lemma \ref{lem:kl},
	\begin{align*}
	&\var \big( \tr (\Sigma S_Z\Sigma S_Z) \big)\leq 16N^{-2} \E \pnorm{Z\Sigma S_Z\Sigma}{F}^2\\
	& = 16N^{-2} \E \tr\big(Z\Sigma S_Z \Sigma \Sigma S_Z \Sigma Z^\top\big) = 16N^{-1} \E \tr\big(\Sigma S_Z \Sigma \Sigma S_Z \Sigma S_Z\big)\\
	&\leq 16 N^{-1} \E \pnorm{S_Z}{\op}^3 \tr(\Sigma^4) \lesssim N^{-1}\big(1\vee (N^{-1}p)\big)^3 \tr(\Sigma^4). 
	\end{align*}
	
	\noindent (3). This follows by integrating the tail of $\abs{b(\Sigma^{1/2}S_Z\Sigma^{1/2})-b(\Sigma)}$ in Lemma \ref{lem:concentration_trace}:
	\begin{align*}
	&\E\abs{b(\Sigma^{1/2}S_Z\Sigma^{1/2})-b(\Sigma)}^\ell = \int_0^\infty \ell t^{\ell-1} \Prob\big(\abs{b(\Sigma^{1/2}S_Z\Sigma^{1/2})-b(\Sigma)}>t\big)\,\d{t}\\
	&\lesssim_\ell \int_0^\infty t^{\ell-1} e^{-\frac{Np^2}{\pnorm{\Sigma}{F}^2}t^2}\ \d{t}+\int_0^\infty t^{\ell-1} e^{-\frac{Np}{\pnorm{\Sigma}{\op}}t}\ \d{t}\\
	&\lesssim_\ell \big(\pnorm{\Sigma}{F}N^{-1/2}p^{-1}\big)^\ell+ \big(\pnorm{\Sigma}{\op}N^{-1}p^{-1}\big)^\ell \asymp \big(\pnorm{\Sigma}{F}N^{-1/2}p^{-1}\big)^\ell.
	\end{align*}
	The proof is complete.	
\end{proof}

\begin{lemma}\label{lem:trace_asymptotics}
	Let $S_Z\equiv N^{-1}\sum_{i=1}^N Z_iZ_i^\top$ where $Z_i$'s are i.i.d. $\mathcal{N}(0,I)$ in $\R^p$. With $y\equiv p/N$ the following hold:
	\begin{enumerate}
		\item $\E\tr(S_Z^3) = py^2+3py+p+3y^2+3y+4y/N$.
		\item $\E\tr^3(S_Z) =  p^3+6py+ 8y/N$.
		\item $\E\tr(S_Z)\tr(S_Z^2) = p^2 y+ p^2 + py+ 4(y^2+y)+4y/N$.
		\item  $\E\tr^2(S_Z^2)$ equals
		\begin{align*} 
		&N^{-4}\big[Np(p+2)(p+4)(p+6) + N(N-1)\big(p(p+2)\big)^2\\
		&\quad\quad\quad + 2N(N-1)3p(p+2) + 4N(N-1)p(p+2)(p+4)\\
		&\quad\quad\quad + 4N(N-1)(N-2)p(p+2) + 2N(N-1)(N-2)p^2(p+2)\\
		&\quad\quad\quad + N(N-1)(N-2)(N-3)p^2\big].
		\end{align*}
		\item $\E\tr(S_Z)\tr(S_Z^3)$ equals
		\begin{align*}
		 &N^{-4}\big[Np(p+2)(p+4)(p+6) + N(N-1)p^2(p+2)(p+4)\\
		&\quad\quad\quad  +3N(N-1)p(p+2)^2 + 3N(N-1)p(p+2)(p+4)\\
		&\quad\quad\quad +3N(N-1)(N-2)p(p+2) + 3N(N-1)(N-2)p^2(p+2)\\
		&\quad\quad\quad +N(N-1)(N-2)(N-3)p^2\big].
		\end{align*}
		\item $\var\big(b(S_Z)b_3(S_Z) - b_2^2(S_Z)\big) = \mathcal{O}(p^2/N^2)$ in the asymptotic regime $p > N\rightarrow \infty$.
		\item For any $k,\ell\in\mathbb{N}$, $\E\tr^k(S_Z^\ell)\leq Cp^{k\ell}$ for some constant $C = C(k,\ell)> 0$.		
	\end{enumerate}
\end{lemma}
\begin{proof}
	Write $S_Z$ for $S$ in the proof for simplicity. Recall that if $R$ follows a chi-squared distribution with an integer $\nu$ degrees of freedom, then 
	\begin{align*}
	\E R^2 = \nu^2 + 2\nu, \quad \E R^3 &= \nu^3 + 6\nu^2 + 8\nu, \quad \E R^4 = \nu(\nu+2)(\nu+4)(\nu+6).
	\end{align*}
	Hence (1)-(3) follows from the following calculations: We have
	\begin{align*}
	\E\tr(S^3) &= N^{-3}\cdot\E\sum_{i_1,i_2,i_3} (Z_{i_1}^\top Z_{i_2}) (Z_{i_2}^\top Z_{i_3}) (Z_{i_3}^\top Z_{i_1})\\
	&= N^{-3}\cdot \Big(\sum_{|(i_1,i_2,i_3)|=1} \E\pnorm{Z_1}{}^6 + \sum_{|(i_1,i_2,i_3)|=2}\E\pnorm{Z_1}{}^4 + \sum_{|(i_1,i_2,i_3)|=3}\E\pnorm{Z_1}{}^2\Big)\\
	&= N^{-3}\cdot\big[N\cdot \E (\chi^2_p)^3 + (3N^2 - 3N)\E (\chi^2_p)^2 + N(N-1)(N-2)\cdot p\big]\\
	& = py^2+3py+p+3y^2+3y+4N^{-1}y,
	\end{align*}
	and
	\begin{align*}
	\E\tr^3(S) &= N^{-3} \E\bigg(\sum_{i=1}^N \pnorm{Z_i}{}^2\bigg)^3 = N^{-3}\E(\chi^2_{Np})^3
	 = p^3+6py+ 8N^{-1}y,
	\end{align*}
	and 
	\begin{align*}
	&\E\tr(S)\tr(S^2) = N^{-3} \E\Big(\sum_{i_1=1}^N \pnorm{Z_{i_1}}{}^2\Big)\Big(\sum_{i_2,i_3=1}^N (Z_{i_2}^\top Z_{i_3})^2\Big)\\
	&= N^{-3}\E\sum_{i_1,i_2,i_3=1}^N \pnorm{Z_{i_1}}{}^2(Z_{i_2}^\top Z_{i_3})^2\\
	&= N^{-3}\Big[\sum_{|(i_1,i_2,i_3)|=3}\E\pnorm{Z_1}{}^2\cdot \E(Z_2^\top Z_3)^2 + \sum_{(i_2=i_3)\neq i_1}\E\pnorm{Z_1}{}^2 \cdot \E\pnorm{Z_1}{}^4\\
	&\quad\quad\quad + \sum_{|(i_2,i_3)| = |(i_1,i_2,i_3)|=2}\E \pnorm{Z_1}{}^2(Z_1^\top Z_2)^2 + \sum_{|(i_1,i_2,i_3)|=1}\E\pnorm{Z_1}{}^6\Big]\\
	&= N^{-3}\big[N(N-1)(N-2)p^2 + (N^2- N)(p^3 + 2p^2)\\
	&\quad\quad\quad + 2(N^2 - N)\big(p^2 + 2p\big) + N\big(p^3 + 6p^2 + 8p\big)\big]\\
	&= p^2 y+ p^2 + py+ 4(y^2+y)+4N^{-1}y.
	\end{align*}
	
	\noindent (4). By definition, we have
	\begin{align}\label{eq:sum1_general}
	\E \tr^2(S^2) &= N^{-4}\E\Big(\sum_{i_1,i_1'=1}^N (X_{i_1}^\top X_{i_1'})^2\Big)^2= N^{-4}\sum_{i_1,i_1',i_2,i_2'=1}^N \E(X_{i_1}^\top X_{i_1'})^2(X_{i_2}^\top X_{i_2'})^2.
	\end{align}
	The right hand side of (\ref{eq:sum1_general}) breaks into $\sum_{i=1}^7 A_i$, where $A_1$, $A_2$-$A_4$, $A_5$-$A_6$, and $A_7$ correspond to the cases where $(i_1,i_1',i_2,i_2')$ take $1,2,3,4$ distinct values, respectively:
	\begin{itemize}
		\item ($A_1$)  When $(i_1,i_1',i_2,i_2')$ take $1$ value, there are $N$ such summands each of which take the value $\E\pnorm{X_1}{}^8 = p(p+2)(p+4)(p+6)$.
		\item ($A_2$)  When $(i_1,i_1',i_2,i_2')$ take 2 values with $(i_1 = i_1')\neq (i_2 = i_2')$, there are $N(N-1)$ such summands each of which takes the value $\E\pnorm{X_1}{}^4\pnorm{X_2}{}^4 = \big(\E\pnorm{X_1}{}^4\big)^2 = p^2(p+2)^2$.
		\item ($A_3$)  When $(i_1,i_1',i_2,i_2')$ take 2 values with $(i_1 = i_2)\neq (i_1' = i_2')$ or $(i_1 = i_2')\neq (i_2 = i_1')$, there are $2N(N-1)$ such summands each of which takes the value
		\begin{align*}
		&\E(X_1^\top X_2)^2(X_1^\top X_2)^2 = \E(X_1^\top X_2)^4 =\sum_{j_1,j_2,j_3,j_4=1}^p \big(\E X_{1,j_1}X_{1,j_2}X_{1,j_3}X_{1,j_4}\big)^2\\
		&= \sum_{|(j_1,j_2,j_3,j_4)| = 1} \big(\E X_{1,j_1}X_{1,j_2}X_{1,j_3}X_{1,j_4}\big)^2 + \sum_{|(j_1,j_2,j_3,j_4)| = 2} \big(\E X_{1,j_1}X_{1,j_2}X_{1,j_3}X_{1,j_4}\big)^2\\
		&= p\cdot 3^2 + 3p(p-1)\cdot 1 = 3p(p+2).
		\end{align*}
		\item ($A_4$)  When $(i_1,i_1',i_2,i_2')$ take 2 values of the form $(i_1 = i_2 = i_1')\neq i_2'$ or its variants, there are $4N(N-1)$ such summands each of which takes the value
		\begin{align*}
		\E\pnorm{X_1}{}^4(X_1^\top X_2)^2 = \E\tr(\pnorm{X_1}{}^4 X_1X_1^\top X_2X_2^\top) = \E\pnorm{X_1}{}^6 = p(p+2)(p+4).
		\end{align*}
		\item ($A_5$)  When $(i_1,i_1',i_2,i_2')$ take 3 values of the form $(i_1 = i_2)\neq i_1'\neq  i_2'$ or its variants, there are $4N(N-1)(N-2)$ such summands each of which takes the value $\E(X_1^\top 		X_2)^2(X_1^\top X_3)^2 = \E\pnorm{X_1}{}^4 = p(p+2)$.
		\item ($A_6$)  When $(i_1,i_1',i_2,i_2')$ take 3 values of the form $(i_1 = i_1')\neq  i_2 \neq  i_2'$ or its variants, there are $2N(N-1)(N-2)$ such summands each of which takes the value $\E(X_1^\top X_1)^2(X_2^\top X_3)^2 = p\cdot\E\pnorm{X_1}{}^4 = p^2(p+2)$.
		\item ($A_7$)  When $(i_1,i_1',i_2,i_2')$ take 4 values, there are $N(N-1)(N-2)(N-3)$ such summands each of which takes the value $\E(X_1^\top X_2)^2(X_3^\top X_4)^2 = p^2$.  
	\end{itemize}

	\noindent (5). By definition, we have
	\begin{align}\label{eq:sum2_general}
	\notag\E \big[\tr(S)\tr(S^3)\big] &= N^{-4}\E\Big(\sum_{i=1}^N \pnorm{X_{i}}{}^2\Big)\Big(\sum_{j_1,j_2,j_3=1}^N \big(X_{j_1}^\top X_{j_2}\big)\big(X_{j_2}^\top X_{j_3}\big)\big(X_{j_3}^\top X_{j_1}\big)\Big)\\
	&= N^{-4}\sum_{i,j_1,j_2,j_3=1}^N \E\Big[ \pnorm{X_{i}}{}^2\big(X_{j_1}^\top X_{j_2}\big)\big(X_{j_2}^\top X_{j_3}\big)\big(X_{j_3}^\top X_{j_1}\big)\Big].
	\end{align}
	The right hand side of (\ref{eq:sum2_general}) breaks into $\sum_{i=1}^7 B_i$, where $B_1$, $B_2$-$B_4$, $B_5$-$B_6$, and $B_7$ correspond to the cases where $(i,j_1,j_2,j_3)$ take $1,2,3,4$ distinct values, respectively:  
	\begin{itemize}
	\item ($B_1$)  When $(i,j_1,j_2,j_3)$ takes $1$ value, there are $N$ such summands in (\ref{eq:sum2_general}), each of which takes the value $\E\pnorm{X_1}{}^8 = p(p+2)(p+4)(p+6)$.
		\item ($B_2$)  When $(i,j_1,j_2,j_3)$ take 2 values with $i \neq (j_1 = j_2 = j_3)$, there are $N(N-1)$ such summands each of which takes the value $\E\pnorm{X_1}{}^2\pnorm{X_2}{}^6 = p^2(p+2)(p+4)$.
		\item ($B_3$)  When $(i,j_1,j_2,j_3)$ take 2 values of the form $(i = j_1)\neq (j_2 = j_3)$ and its variants, there are $3N(N-1)$ such summands each of which takes the value 
		\begin{align*}
		&\E\pnorm{X_1}{}^2(X_1^\top X_2)\pnorm{X_2}{}^2(X_2^\top X_1) = \E\tr\Big(\pnorm{X_1}{}^2 X_1X_1^\top \cdot \pnorm{X_2}{}^2 X_2X_2^\top\Big)\\
		&= \tr\Big(\E\pnorm{X_1}{}^2X_1X_1^\top \Big)^2 \stackrel{(*)}{=} \tr\big[(p+2)I_p\big]^2 = p(p+2)^2.
		\end{align*} 
		Here in $(*)$ we use the following fact by direct calculation $
		\big(\E\pnorm{X_1}{}^2X_1X_1^\top \big)_{k\ell} = \E\big[ \big(\sum_{m=1}^p X_{1,m}^2\big)X_{1,k}X_{1,\ell} \big]= (p+2)\delta_{k\ell}$. 
		\item ($B_4$)  When $(i,j_1,j_2,j_3)$ take 2 values of the form $(i = j_1 = j_2)\neq j_2$ or its variants, there are $3N(N-1)$ such summands each of which takes the value $\E\pnorm{X_1}{}^4(X_1^\top X_2)^2 = \E\pnorm{X_1}{}^6 = p(p+2)(p+4)$.
		\item ($B_5$)  When $(i,j_1,j_2,j_3)$ take 3 values of the form $(i = j_1)\neq j_2 \neq  j_3$ or its variants, there are $3N(N-1)(N-2)$ such summands each of which takes the value $\E\pnorm{X_1}{}^2(X_1^\top X_2)(X_2^\top X_3)(X_3^\top X_1) = \E\pnorm{X_1}{}^4 = p(p+2)$.
		\item ($B_6$)  When $(i,j_1,j_2,j_3)$ take 3 values of the form $i \neq (j_1 = j_2)\neq j_3$ or its variants, there are $3N(N-1)(N-2)$ such summands each of which takes the value $\E\pnorm{X_1}{}^2\pnorm{X_2}{}^2(X_2^\top X_3)^2 = \E\pnorm{X_1}{}^2 \E\pnorm{X_2}{}^4 = p^2(p+2)$.
		\item ($B_7$)  When $(i,j_1,j_2,j_3)$ take 4 values, there are $N(N-1)(N-2)(N-3)$ such summands each of which takes the value $\E\big[\pnorm{X_1}{}^2(X_2^\top X_3)(X_3^\top X_4)(X_4^\top X_2)\big] = p^2$.  
	\end{itemize}
	
	\noindent (6). Let $F(X)\equiv b(S)b_3(S) - b_2^2(S)$. Then using for any $(i,j)\in[N]\times [p]$, $
	\partial_{ij} b = 2(Np)^{-1}X_{ij}$, $\partial_{ij} b_2 = 4(Np)^{-1}X_i^\top Se_j$, $\partial_{ij}b_3 = 6(Np)^{-1}X_i^\top S^2e_j$, 
	we have
	\begin{align*}
	\partial_{ij} F (X) &= 2(Np)^{-1}X_{ij}\cdot b_3 + 6(Np)^{-1}X_i^\top S^2e_j\cdot b - 8(Np)^{-1}X_i^\top Se_j\cdot b_2\\
	&= (Np)^{-1}\Big(2b_3X + 6bXS^2 - 8b_2XS\Big)_{(ij)}.
	\end{align*}
	Hence by the Gaussian-Poincar\'e inequality, we have by direct calculation
	\begin{align*}
	\var(F(X)) &\leq \E\pnorm{\nabla F(X)}{F}^2= (Np)^{-1}\E\big(28bb_3^2 + 36b^2b_5 + 32b_2^2b_3 - 96bb_2b_4\big)\\
	&= (Np^4)^{-1}\E\Big[28\tr(S)\tr^2(S^3) + 36\tr^2(S^2)\tr(S^5) + 32\tr^2(S^2)\tr(S^3)\\
	&\quad\quad\quad - 96\tr(S)\tr(S^2)\tr(S^4)\Big]. 
	\end{align*}
	The rest of the proof follows from similar arguments as in (4) and (5) by explicit calculation and cancellation of higher order terms; we omit the details. 
	
	\noindent (7). This follows directly from the property of the chi-squared distribution:
	\begin{align*}
	\E\tr^k(S^\ell) &= N^{-k\ell}\E\Big[\sum_{i_1,\ldots,i_\ell=1}^N (X_{i_1}^\top X_{i_2})\cdots (X_{i_{\ell-1}}^\top X_{i_\ell})(X_{i_\ell}^\top X_{i_1})\Big]^k\\
	&\leq N^{-k\ell}\E\Big(\sum_{i_1,\ldots,i_\ell=1}^N \pnorm{X_{i_1}}{}^2\cdots \pnorm{X_{i_\ell}}{}^2\Big)^k = N^{-k\ell} \E \Big(\sum_{i=1}^{N}\pnorm{X_i}{}^2\Big)^{k\ell}\\
	&= N^{-k\ell}  \E \big(\chi^2(Np)\big)^{k\ell}\lesssim_{k,\ell} p^{k\ell}.
	\end{align*}
	The proof is complete.
\end{proof}

\section*{Acknowledgments}
The authors would like to thank an anonymous referee, an Associate Editor and the Editor for their very helpful comments and suggestions that significantly improved the quality of the paper. 

\bibliographystyle{amsalpha}
\bibliography{mybib}

\end{document}